\documentclass[10pt]{amsart}

\usepackage{mathrsfs,amssymb,amsmath}
\usepackage{verbatim}
\usepackage{hyperref}
\usepackage{graphicx}
\usepackage{enumerate}

\usepackage{vmargin}
\usepackage{subfigure}

\usepackage{layout}
\usepackage{booktabs}
\newcommand{\ra}[1]{\renewcommand{\arraystretch}{#1}}
\begin{document}

%**************************************************************V
%\layout
%**************************************************************

\newtheorem{theorem}{Theorem}[section]
\newtheorem{proposition}[theorem]{Proposition}
\newtheorem{lemma}[theorem]{Lemma}
\newtheorem{corollary}[theorem]{Corollary}
\newtheorem{conjecture}[theorem]{Conjecture}
\newtheorem{question}[theorem]{Question}
\newtheorem{problem}[theorem]{Problem}

\theoremstyle{definition}
\newtheorem{definition}[theorem]{Definition}
\newtheorem{example}[theorem]{Example}

\newtheorem{remark}[theorem]{Remark}

\def\theenumi{\roman{enumi}}

\numberwithin{equation}{section}

\renewcommand{\Re}{\operatorname{Re}}
\renewcommand{\Im}{\operatorname{Im}}

%*************************************************************V
\newcommand{\ind}{1\hspace{-.27em}\mbox{\rm l}}
\newcommand{\inde}{1\hspace{-.23em}\mathrm{l}}
\newcommand{\lqn}[1]{\noalign{\noindent $\displaystyle{#1}$}}
\newcommand{\bm}[1]{\mbox{\boldmath{$#1$}}}
\newcommand{\eps}{\varepsilon}
\newcommand{\lip}{{\rm Lip}}
\newcommand{\R}{\mathbb{R}}
\newcommand{\N}{\mathbb{N}}
\newcommand{\Z}{\mathbb{Z}}
\newcommand{\Sym}{\mathfrak{S}}
\newcommand{\Leb}{{\lambda}}
\newcommand{\md}{\mathbb{D}}
\newcommand{\me}{\mathbb{E}}
\newcommand{\Q}{\mathbb{Q}}
\newcommand{\1}{{\sf 1}}
\newcommand{\tmu}{\tilde{\mu}}
\newcommand{\tnu}{\tilde{\nu}}
\newcommand{\hmu}{\hat{\mu}}
\newcommand{\mb}{\bar{\mu}}
\newcommand{\bnu}{\overline{\nu}}
\newcommand{\tx}{\tilde{x}}
\newcommand{\tz}{\tilde{z}}
\newcommand{\A}{{\mathcal A}}
\newcommand{\skria}{{\mathcal A}}
\newcommand{\skrib}{{\mathcal B}}
\newcommand{\skric}{{\mathcal C}}
\newcommand{\skrid}{{\mathcal D}}
\newcommand{\skrie}{{\mathcal E}}
\newcommand{\skrif}{{\mathcal F}}
\newcommand{\skrig}{{\mathcal G}}
\newcommand{\skrih}{{\mathcal H}}
\newcommand{\skrii}{{\mathcal I}}
\newcommand{\skrik}{{\mathcal K}}
\newcommand{\skril}{{\mathcal L}}
\newcommand{\skrim}{{\mathcal M}}
\newcommand{\skrin}{{\mathcal N}}
\newcommand{\skrip}{{\mathcal P}}
\newcommand{\skriq}{{\mathcal Q}}
\newcommand{\skris}{{\mathcal S}}
\newcommand{\skrit}{{\mathcal T}}
\newcommand{\skriu}{{\mathcal U}}
\newcommand{\skriw}{{\mathcal W}}
\newcommand{\skrix}{{\mathcal X}}
\newcommand{\heap}[2]{\genfrac{}{}{0pt}{}{#1}{#2}}
\newcommand{\sfrac}[2]{ \, \mbox{$\frac{#1}{#2}$}}
 \renewcommand{\labelenumi}{(\roman{enumi})}
\newcommand{\ssup}[1] {{\scriptscriptstyle{({#1}})}}

\newcommand{\cal}{\mathcal}
\allowdisplaybreaks

%*************************************************************

\def \R {{\mathbb R}}
\def \HH {{\mathbb H}}
\def \N {{\mathbb N}}
\def \C {{\mathbb C}}
\def \Z {{\mathbb Z}}
\def \Q {{\mathbb Q}}
\def \TT {{\mathbb T}}
\newcommand{\T}{\mathbb T}
\def \Dc {{\mathcal D}}

\newcommand{\tr}[1] {\hbox{tr}\left( #1\right)}

\newcommand{\area}{\operatorname{area}}

\newcommand{\Norm}{\mathcal N}
\newcommand{\simgeq}{\gtrsim}%
\newcommand{\simleq}{\lesssim}
%%these are with amssym, else use {\stackrel{>}{\sim}}

\newcommand{\length}{\operatorname{length}}

\newcommand{\curve}{\mathcal C} % curve
\newcommand{\vE}{\mathcal E} % lattice points on the sphere
\newcommand{\Ec}{\mathcal {E}} % lattice points on the sphere
\newcommand{\Sc}{\mathcal{S}} % lattice points on the sphere

\newcommand{\dist}{\operatorname{dist}}
\newcommand{\supp}{\operatorname{supp}}
\newcommand{\spec}{\operatorname{spec}}
\newcommand{\diam}{\operatorname{diam}}

\newcommand{\Ccap}{\operatorname{Cap}}
\newcommand{\E}{\mathbb E}

\newcommand{\sumstar}{\sideset{}{^\ast}\sum}

\newcommand {\Zc} {\mathcal{Z}} % igor's nodal intersections number
\newcommand{\ninumber}{\Zc}

\newcommand{\zeigen}{E} % eigenvalue
\newcommand{\eigen}{m}%igor's eigenvalue converntion

\newcommand{\ave}[1]{\left\langle#1\right\rangle} % average

\newcommand{\Var}{\operatorname{Var}}
\newcommand{\Prob}{\operatorname{Prob}}

\newcommand{\var}{\operatorname{Var}}
\newcommand{\Cov}{{\rm{Cov}}}
\newcommand{\meas}{\operatorname{meas}}

\newcommand{\leg}[2]{\left( \frac{#1}{#2} \right)} % Legendre symbol

\renewcommand{\^}{\widehat}

\newcommand {\Rc} {\mathcal{R}}

\title[Fluctuations of critical points]
{Fluctuations of the total number of critical points \\ of random spherical harmonics }
\author{V. Cammarota and I. Wigman}

\address{Department of Mathematics, Universit\`a degli Studi di Roma Tor Vergata } \email{cammarot@mat.uniroma2.it}

\address{Department of Mathematics, King's College London }
\email{igor.wigman@kcl.ac.uk}

\date{\today}

%\thanks{Research Supported by ERC Grants n$^{\text{o}}$ 277742 \emph{Pascal} and n$^{\text{o}}$ 335141.}

\begin{abstract}
We determine the asymptotic law for the fluctuations of the total number of critical points of random Gaussian
spherical harmonics in the high degree limit. Our results have implications on the sophistication degree of
an appropriate percolation process for modelling nodal domains of eigenfunctions on generic compact surfaces or billiards.
\end{abstract}

\maketitle

\section{Introduction and main results}

%\begin{verbatim} fatto con ricerca
%All double (or more) spacing should be eliminated
%\end{verbatim}

\subsection{Critical points of random spherical harmonics}
It is well-known that the eigenvalues $\lambda$ of the Laplacian $\Delta$ on the $2$-dimensional round unit sphere ${\cal S}^2$, 
satisfying the Schr\"{o}dinger equation
\begin{equation*}
\Delta f+\lambda f=0
\end{equation*}
are of the form
$\lambda=\lambda_{\ell}=\ell(\ell+1)$ for some integer $\ell\ge 1$.
For any given eigenvalue $\lambda_{\ell}$ of the above form, the corresponding eigenspace is the
$(2 \ell+1)$-dimensional space of spherical harmonics of degree $\ell$; we can choose an arbitrary $L^{2}$-orthonormal
basis $\left\{ Y_{\mathbb{\ell }m}(.)\right\} _{-\ell \le m \le \ell }$, and consider random eigenfunctions of the form
\begin{equation} \label{rsh}
f_{\ell }(x)=\frac{1}{\sqrt{2\ell +1}}\sum_{m=-\ell }^{\ell }a_{\ell m}Y_{\ell m}(x),
\end{equation}
where the coefficients $\left\{ a_{\mathbb{\ell }m}\right\}_{-\ell \le m \le \ell}$ are
independent, standard Gaussian variables.

The random fields
$$\{f_{\ell }(x), \; x\in {\cal S}^2\}$$
are centred Gaussian and the law of $f_{\ell }$ in \eqref{rsh} is invariant with respect to the choice of $\{Y_{\ell m}\}$.
Also, $f_{\ell }$ are isotropic, meaning that
the probability laws of $f_{\ell }(\cdot )$ and $f_{\ell }^{g}(\cdot
):=f_{\ell }(g\cdot )$ are the same for every rotation $g\in SO(3)$.
By the addition theorem for spherical harmonics \cite[Theorem 9.6.3]{AAR} the
covariance function of $f_{\ell}$ is given by
\begin{equation*}
\mathbb{E}[f_{\ell }(x)f_{\ell }(y)]=P_{\ell }(\cos d(x,y)),
\end{equation*}%
where $P_{\ell }$ are the usual Legendre polynomials,
$$\cos d(x,y)=\cos \theta _{x}\cos \theta _{y}+\sin \theta _{x}\sin
\theta _{y}\cos (\varphi _{x}-\varphi _{y})$$
is the spherical geodesic
distance between $x$ and $y$, $\theta \in \lbrack 0,\pi ]$, $\varphi \in
\lbrack 0,2\pi )$ are standard spherical coordinates and $(\theta
_{x},\varphi _{x})$, $(\theta _{y},\varphi _{y})$ are the spherical
coordinates of $x$ and $y$ respectively.

\vspace{3mm}

Our primary focus is the total number of critical points of $f_{\ell}$
\begin{equation*}
\mathcal{N}^{c}(f_\ell )=\#\{x\in {\cal S}^2: \nabla f_{\ell }(x)=0\}.
\end{equation*}
It is known \cite{Nicolaescu_2, CMW} that, as $\ell \to \infty$, the expected 
total number of critical points $\mathcal{N}^{c}(f_\ell )$ is asymptotic to
$$\mathbb{E}[\mathcal{N}^{c}(f_\ell )]=\frac{2}{\sqrt 3} \ell^2+O(1).$$
An upper bound for the variance of the number of critical points $\mathcal{N}^{c}(f_\ell )$ was also derived \cite{CMW}:
\begin{equation*}
{\rm Var}(\mathcal{N}^{c}(f_\ell ))= O(\ell^{\frac 5 2});
\end{equation*}
in fact, it is likely that the same method yields the stronger result
\begin{equation*}
{\rm Var}(\mathcal{N}^{c}(f_\ell ))= O(\ell^2\log{\ell}).
\end{equation*}
It was conjectured \cite{CMW} that the true asymptotic behaviour of the variance is
\begin{align} \label{3:02}
{\rm Var}(\mathcal{N}^{c}(f_\ell ))={\rm const} \cdot \ell^2 \log \ell+O(\ell^2).
\end{align}

More generally let $I \subseteq \mathbb{R}$ be any interval and $\mathcal{N}^{c}_I(f_\ell )$ be the number
of critical points of $f_{\ell}$ with value in $I$:
$$\mathcal{N}^{c}_I(f_\ell )=\#\{x\in {\cal S}^2: f_{\ell }(x) \in I , \nabla f_{\ell }(x)=0\};$$
it was proved in \cite[Theorem 1.2]{CMW} that as $\ell \to \infty$ it holds that
\begin{equation*}
\text{Var}(\mathcal{N}^{c}_I(f_\ell ))=\ell^3 \nu^c(I)+O(\ell^{5/2}),
\end{equation*}
where the leading constant $\nu^c(I)$ was evaluated explicitly.
For some intervals $I$, such as, for example $I=\mathbb{R}$ (corresponding
to the total number of critical points), the leading constant $\nu^c(I)$ vanishes, and, accordingly,
the order of magnitude of the variance is smaller than $\ell^3$.
In this paper we prove \eqref{3:02}, i.e. we determine the precise asymptotic
shape for the variance of the total number of critical points of $f_{\ell}$.

\subsection{Statement of the main result}

The principal result of this paper is the following:
\begin{theorem} \label{th1}
As $\ell \to \infty$
\begin{align*}
{\rm Var}({\cal N}^c(f_{\ell}))&= \frac{1}{3^3 \pi^2} \ell^2 \log \ell+O(\ell^2).%\\
% \text{Var}[{\cal N}^e_{\mathbb R}(f_{\ell})]&= \frac{1}{3^3 \cdot 2^2 \pi^2} \ell^2 \log \ell+O(\ell^2),\\
 %\text{Var}[{\cal N}^s_{\mathbb R}(f_{\ell})]&= \frac{1}{3^3 \cdot 2^2 \pi^2} \ell^2 \log \ell+O(\ell^2).
\end{align*}
The constant in the $O(\cdot)$ term is universal.
\end{theorem}
As in \cite{CMW}, our argument is based on an approximate version of the Kac-Rice formula for counting the number of zeros of the gradient
of $f_{\ell}$ (see Section \ref{kac-rice}). It is easy to adapt the same approach to separate critical points
into extrema and saddles; in fact, we have the following:

 \begin{remark}
Let $\mathcal{N}^{e}(f_\ell)$ and $\mathcal{N}^{s}(f_\ell)$ be the total number of extrema and saddles of $f_{\ell}$
\begin{equation*}
\mathcal{N}^{e}(f_\ell)=\#\{x\in {\cal S}^2: \nabla f_{\ell }(x)=0,\text{det}(\nabla ^{2}f_{\ell
}(x))>0\},
\end{equation*}%
\begin{equation*}
\mathcal{N}^{s}(f_\ell)=\#\{x\in {\cal S}^2
: \nabla f_{\ell }(x)=0,\text{det}(\nabla ^{2}f_{\ell
}(x))<0\}.
\end{equation*}%
As $\ell \to \infty$ we have that
\begin{align}
{\rm Var}({\cal N}^e(f_{\ell}))&= \frac{1}{ 2^2 \cdot 3^3 \pi^2} \ell^2 \log \ell+O(\ell^2), \label{N}\\
 {\rm Var}({\cal N}^s(f_{\ell}))&= \frac{1}{2^2 \cdot 3^3 \pi^2} \ell^2 \log \ell+O(\ell^2). \label{NN}
\end{align}
The asymptotic laws for the fluctuations of the total number of extrema and saddles in \eqref{N} and \eqref{NN} 
follow immediately from Theorem \ref{th1} and Morse Theory. In fact, 
$${\cal N}^c(f_{\ell})={\cal N}^e(f_{\ell})+{\cal N}^s(f_{\ell})$$ 
and, via Morse Theory, it is possible to prove that 
$${\cal N}^e(f_{\ell})= \frac{{\cal N}^c(f_{\ell})}{2}+1.$$
\end{remark}
\vspace{0.2cm}

 \begin{remark}
For the intervals $I \ne \R$ such that the constant $\nu^c(I)$, $\nu^e(I)$ or $\nu^s(I)$ vanish, the variance of the number of critical points, extrema and saddles in $I$ has the following asymptotic behaviour: as $\ell \to \infty$
\begin{align} \label{propinI}
{\rm Var}({\cal N}^a_I(f_{\ell}))= [\mu^a(I)]^2 \ell^2 \log \ell + O(\ell^2),
\end{align}
where we use $a=c,e,s$ to denote critical points extrema and saddles,
$
 \mu^a(I)=\int_I \mu^a(t) d t,
$
and the functions $\mu^a$, for $a=c,e,s$ are defined in \eqref{cr}-\eqref{sad} and derived in Appendix \ref{ProofpropinI}.
\end{remark}

\subsection{Nodal domains and percolation}

The nodal domains of $f_{\ell}$ are the connected components of the complement of the nodal lines $f^{-1}_{\ell}(0)$, i.e.
the connected components of
$${\cal S}^2 \setminus f^{-1}_{\ell}(0).$$
Let $N(f_{\ell})$ be the number of nodal domains of $f_{\ell}$.
Nazarov and Sodin \cite{N&S} proved that there exists a constant $a>0$ such that the expected number of nodal
domains is asymptotic to
\begin{align} \label{N-S}
\mathbb{E}[N(f_{\ell})]\sim a \ell^2.
\end{align}

Little is known about the leading constant $a$ in \eqref{N-S}. For once the nodal domains
number is bounded from above by the total number of critical points; the latter inequality
could be improved by a factor of $2$ by separating the critical points into extrema and saddles
(for example, via Morse Theory),
an approach pursued by Nicolaescu ~\cite{Nicolaescu_2} yielding the upper bound $$a\le \frac{1}{\sqrt 3};$$
while it is possible to improve the latter bound by using other local estimates (e.g. ~\cite{Krishnapur}),
these are far off the numerical Monte-Carlo simulations or the conjectured values of $a$.

\vspace{3mm}

To the other end, other than the Nastasescu's ~\cite{Nastasescu} explicating the Nazarov-Sodin
``barrier" construction ~\cite{N&S} (yielding a tiny lower bound on $a$),
to our best knowledge, no lower bound for $a$ is known rigorously.
Bogomolny and Schmit ~\cite{Bogomolny-Schmit} conjectured that, as $\ell\rightarrow\infty$,
nodal domains of $f_{\ell}$ (more generally, deterministic Laplace eigenfunctions on generic compact surfaces or billiards)
are described by the clusters in a rectangular lattice bond percolation-like process
with $\approx \ell^{2}$ sites (called the Percolation Model), and in particular that the true value of $a$ equals
the leading constant $$a=\frac{3\sqrt{3}-5}{\pi}\approx 0.0624$$ for the asymptotic number of connected clusters
in the Percolation Model. Here we think of the maxima and minima of $f_{\ell}$ as rigidly arranged along
two mutually dual percolation lattices; adjacent maxima are connected independently with probability $\frac{1}{2}$,
if and only if the dual minima are disconnected.

Some recent simulations ~\cite{Nastasescu,Belyaev Kereta}
showed deviations of about $4.5\%$ between the predicted constant for $a$ and its numerical values;
these cannot be attributed to numerical errors. It is then desirable to come up with a more sophisticated percolation
model\footnote{We would like to thank Dmitry Belyaev for discussing connections between our work and percolation.}
~\cite{Belyaev Kereta,Belyaev private} that would match these constant more precisely, where, in particular,
the arrangement critical points of $f_{\ell}$ would exhibit some degree of randomness, less rigid than rectangular lattice.
The variance \eqref{propinI} of the total number of critical points (or the extrema) of $f_{\ell}$ is then crucial
in determining the rigidity or flexibility of the (random) percolation sites.

\subsection{Acknowledgements}

The research leading to these results has received funding from the
European Research Council under the European Union's Seventh
Framework Programme (FP7/2007-2013) / ERC grant agreements
n$^{\text{o}}$ 277742 (V.C.) and n$^{\text{o}}$ 335141 (I.W.).
We are grateful to Dmitry Belyaev, Domenico Marinucci and Zeev Rudnick for
some useful discussions and suggestions and to Mikhail Sodin
for discussion especially with relations to nodal domains.

\section{On (approximate) Kac-Rice formula for computing $2$nd (factorial) moment}
\label{kac-rice}

In this section we express the second factorial moment of ${\cal N}^c(f_{\ell})$ via Kac-Rice
formula. Let $\mathcal{E}\subseteq \R^{n}$ be a nice Euclidian domain,
and $g:\mathcal{E}\rightarrow \R^{n}$ a centred Gaussian
random field, a.s. smooth. Define the $2$-point correlation function
$$K_{2}=K_{2;g}:\mathcal{E}^{2}\rightarrow\R$$ of the zeros of $g$ as
\begin{equation}
\label{eq:2pnt corr def}
K_{2}(x,y) = \phi_{(g(x),g(y))}(\mathbf{0},\mathbf{0})\cdot
\E[ |\det J_{g}(x) | \cdot |\det J_{g}(y) | \big| g(x)=g(y)=\mathbf{0} ],
\end{equation}
where $\phi_{(g(x),g(y))}$ is the Gaussian probability density of $(g(x),g(y)) \in \mathbb{R}^{2n}$
and $J_{g}(x)$, $J_{g}(y)$ are the Jacobian matrices of $g$ at $x$ and $y$ respectively.
In view of~\cite[Theorem 6.3]{azaiswschebor} (see also~\cite[Proposition 1.2]{azaiswschebor}) the $2$nd
factorial moment of $g^{-1}(0)$ is given by
\begin{equation}
\label{eq:Kac-Rice 2nd fact}
\E[\# g^{-1}(0)\cdot (\# g^{-1}(0)-1)] =
\int\limits_{\mathcal{E}^{2}}K_{2}(x,y)dxdy,
\end{equation}
provided that the Gaussian distribution of $(g(x),g(y))\in \R^{2n}$
is non-degenerate for all $(x,y)\in\mathcal{E}^{2}$. Moreover, for
$\mathcal{D}_{1},\mathcal{D}_{2}\subseteq\mathcal{E}$ two nice disjoint
domains, we have
\begin{equation}
\label{eq:Kac-Rice 2nd fact disjoint}
\E[(\# g^{-1}(0)\cap \mathcal{D}_{1}) \cdot (\# g^{-1}(0)\cap \mathcal{D}_{2})] =
\iint\limits_{\mathcal{D}_{1}\times \mathcal{D}_{2}}K_{2}(x,y)dxdy,
\end{equation}
under the same non-degeneracy assumption for all $(x,y)\in \mathcal{D}_{1}\times \mathcal{D}_{2}$.
To apply \eqref{eq:Kac-Rice 2nd fact disjoint} and \eqref{eq:Kac-Rice 2nd fact}
in our case we work with the spherical coordinates on ${\cal S}^2$
and use an explicit orthonormal frame (see Section \ref{K-R_coord}).

To apply Kac-Rice formula in our case we will work with spherical coordinates on 
${\cal S}^2$ and choose an explicit orthogonal frame, see \eqref{orthogonalf} below. Counting the critical points of $f_{\ell}$ is then equivalent to counting the
zeros of the map $[0,\pi] \times [0,2 \pi] \to \mathbb{R}^2$ given by
$x \to \nabla f_{\ell}(x)$; accordingly for $x \ne \pm y$ the two-point
correlation function of critical points of $f_{\ell}$ is (cf. \eqref{eq:2pnt corr def})
\begin{equation}
\label{eq:K2 glob gen}
K_{2,\ell}(x,y) = \phi_{(\nabla f_{\ell}(x),\nabla f_{\ell}(y))}(\mathbf{0},\mathbf{0})\cdot
\E[ |\det H_{f_{\ell}}(x) | \cdot |\det H_{f_{\ell}}(y) | \big| \nabla f_{\ell}(x)=\nabla f_{\ell}(y)=\mathbf{0} ],
\end{equation}
where $H_{f_{\ell}}(x)$ and $H_{f_{\ell}}(y)$ are the Hessian matrices of $f_{\ell}$ at $x$ and $y$ respectively.
Here~\cite[Theorem 6.3]{azaiswschebor} (see also \cite[Theorem 11.2.1]{adlertaylor}) would yield
\begin{align} \label{precise_K-R}
\mathbb{E}[{\cal N}^c(f_{\ell}) \cdot ({\cal N}^c(f_{\ell})-1)]=\iint_{{\cal S}^2 \times {\cal S}^2} K_{2,\ell}(x,y) d x d y,
\end{align}
under the condition that for all $x,y\in \mathcal{S}^{2}$, the Gaussian distribution of
$\left(\nabla f(x), \nabla f(y)\right)\in \R^{4}$ were non-degenerate.
We can easily adapt the definition of the $2$-point correlation in \eqref{eq:K2 glob gen}
to separate the critical points into extrema and saddles, or count critical points with values lying in $I$ (see Appendix \ref{K-RinI} and \cite{CMW}).

Note that the rotational invariance of $f_{\ell}$ implies that the function $K_{2,\ell}$ in \eqref{eq:K2 glob gen} depends on the points $x$, $y$ only via their geodesic distance $\phi=d(x,y)$; with a slight abuse of notations we write
$
K_{2,\ell}(\phi)=K_{2,\ell}(x,y).
$
Also, note that $K_{2,\ell}(\phi)$ is everywhere nonnegative.

We do not validate the non-degeneracy assumption of the $4\times 4$ covariance matrices of $\left(\nabla f(x), \nabla f(y)\right)$
depending on both $x$ and $y$ (and $\ell$); instead 
we prove that the precise Kac-Rice formula \eqref{precise_K-R} holds up to an
admissible error, an approach inspired by \cite{rudnickwigman}. We recall here the main steps of the proof of the {\em approximate Kac-Rice} formula and refer to \cite[Section 3]{CMW} for a complete proof.
The argument is based on a partitioning of the integration domain in \eqref{precise_K-R}; we apply \eqref{eq:Kac-Rice 2nd fact disjoint} on the valid slices we bound the contribution of the rest.\\

For $x\in {\cal S}^2$, $r>0$ let ${\cal B}(x,r)=\{y \subseteq {\cal S}^2: d(x,y) \le r\}$ be a closed spherical cap on ${\cal S}^2$. For $\varepsilon>0$ we say that $$\Xi_\varepsilon=\{\xi_{1,\varepsilon}, \dots, \xi_{N,\varepsilon}\}\subseteq {\cal S}^2$$
is a maximal $\varepsilon$-net if for every $i\ne j$ we have $d(\xi_{i,\varepsilon}, \xi_{j,\varepsilon})>\varepsilon$, and also every $x\in {\cal S}^2$ satisfies $$d(x,\Xi_\varepsilon) \le \varepsilon.$$
That is, informally speaking an $\varepsilon$-net is a collection of $\varepsilon$-separated points,
whose $\varepsilon$-thickening covers the whole of ${\cal S}^2$.
The number $N$ of points in a
$\varepsilon$-net on the sphere can be bounded from above and from below;
indeed it satisfies the following \cite[Lemma 5]{BKMP}:
\begin{align}\label{nett}
\frac{4}{\varepsilon^2} \le N \le \frac{4}{\varepsilon^2} \pi^2.
\end{align}

Given a maximal $\varepsilon$-net, it is natural to partition the sphere into disjoint sets, each of them associated with a single point in the net. This task is accomplished by the Voronoi cells construction \cite[Section 11.2]{MaPeCUP}:
%Given a maximal $\varepsilon$-net
%we define a family of Voronoi cells:
\begin{definition}
Let $\Xi_\varepsilon$ be a maximal $\varepsilon$-net. For all $\xi_{i,\varepsilon} \in \Xi_\varepsilon$, the associated {\it family of Voronoi cells} is defined by
$${\cal V}(\xi_{i,\varepsilon},\varepsilon)=\{x \in {\cal S}^2: \forall j \ne i, \; d(x, \xi_{i,\varepsilon}) \le d(x,\xi_{j,\varepsilon})\}.$$
\end{definition}
\noindent Each Voronoi cell is associated to a single point on the net. The Voronoi cells are disjoint, save to boundary overlaps, and cover the
whole sphere. \\

\noindent It is possible to prove the following:
\begin{proposition} \label{pafkao}
There exists a constant $C>0$ sufficiently big, such that the following approximate Kac-Rice holds:
\begin{align} \label{afkao2}
{\rm Var} \left( \mathcal{N}^{c} (f_{\ell })\right) = \int_{\mathcal{W}} K_{2,\ell}(x,y) \; d x d y - (\E [ \mathcal{N}^{c} (f_{\ell })])^2 + O(\ell^2),
\end{align}
where $\mathcal{W}$ is the union of all tuples of points belonging to Voronoi cells far from the domain of degeneracy, i.e.,
\begin{equation*}
\mathcal{W}=\bigcup _{d( {\cal V}(\xi _{i,\varepsilon }), {\cal V}(\xi _{j,\varepsilon })) \in (C/\ell, \pi-C/\ell)} {\cal V}(\xi _{i,\varepsilon })\times {\cal V} (\xi _{j,\varepsilon }).
\end{equation*}
\end{proposition}
\vspace{0.3cm}

\subsection{On the proof of Proposition \ref{pafkao}}
\noindent Note that, almost surely, the summation of the critical points over the
Voronoi cells equals the total number of critical points, therefore we write the variance
of the total number of critical points as
\begin{equation} \label{mir1}
{\rm Var} \left( \mathcal{N}^{c} (f_{\ell })\right) =\sum_{\xi_{i,\varepsilon}, \xi_{j,\varepsilon}
\in \Xi_\varepsilon } {\rm Cov} \left( \mathcal{N}^{c}(f_{\ell }; {\cal V}(\xi _{i,\varepsilon }) ),\mathcal{N}
^{c}(f_{\ell }; {\cal V}(\xi _{j,\varepsilon }) )\right),
\end{equation}
where
$${\cal N}^c(f_\ell; {\cal V}(\xi_{i,\varepsilon},\varepsilon))=\#\{x \in {\cal V}(\xi_{i,\varepsilon},\varepsilon): \nabla f_\ell(x)=0\}.$$\\

 The main steps of the proof of Proposition \ref{pafkao} are the following. In \cite[Lemma 3.2]{CMW} it was proved that there exists a constant $C>0$ sufficiently big,
such that, in the regime $d({\cal V}(\xi _{i,\varepsilon }), {\cal V}(\xi _{j,\varepsilon })) \in (C/\ell, \pi-C/\ell)$, the covariance
matrix is nonsingular and so Kac-Rice formula holds exactly. This gives the first term in \eqref{afkao2}.\\

In the regime $d({\cal V}(\xi _{i,\varepsilon }), {\cal V}(\xi _{j,\varepsilon })) \in [0,C/\ell] \cup[ \pi-C/\ell,\pi]$, using Cauchy-Schwartz inequality, we can bound the covariance as
\begin{equation}
\label{csbound}
\left\vert {\rm Cov} \left( \mathcal{N}^{c}(f_{\ell }; {\cal V} (\xi _{i,\varepsilon })),%
\mathcal{N}^{c}(f_{\ell }; {\cal V}(\xi _{j,\varepsilon }))\right) \right\vert
\leq \sqrt{{\rm Var} \left( \mathcal{N}^{c}(f_{\ell }; {\cal V}(\xi _{i,\varepsilon
}))\right) }\cdot\sqrt{{\rm Var} \left( \mathcal{N}^{c}(f_{\ell }; {\cal V}(\xi
_{j,\varepsilon }))\right) }.
\end{equation}
In \cite[Section 4.2]{CMW} the
non-degeneracy of the covariance matrix was proved for sufficiently close points $x$, $y$, i.e., it was proved that there exists a constant $c >0$ sufficiently small
such that for $\varepsilon=c/\ell$ the Kac-Rice formula holds precisely:
\begin{align}
{\rm Var} \left( \mathcal{N}^{c}(f_{\ell };{\cal V}(\xi _{\varepsilon ,i}))\right)
&=\iint_{{\cal V} (\xi _{\varepsilon ,i})\times {\cal V}(\xi _{\varepsilon ,i})} K_{2,\ell}(x,y) dxdy +\mathbb{E}\left[ \mathcal{N}^{c}(f_{\ell }; {\cal V}(\xi _{\varepsilon ,i}))
\right] -(\mathbb{E}\left[ \mathcal{N}^{c}(f_{\ell }; {\cal V}(\xi _{\varepsilon
,i}))\right])^{2}. \label{lario11}
\end{align}
Now, in view of \cite[Lemma 3.6]{CMW}, there exists a constant $c>0$ such that, for $d(x,y)<c/\ell$, one has 
$$K_{2,\ell}(x,y)=O(\ell^4),$$ 
and since ${\cal B}(\xi_{i,\varepsilon},\varepsilon/2) \subseteq {\cal V}(\xi_{i,\varepsilon},\varepsilon) \subseteq {\cal B}(\xi_{i,\varepsilon},\varepsilon)$, we have $\text{Vol}({\cal V}(\xi_{i,\varepsilon},\varepsilon)) \approx \varepsilon^2$. Then the first term in \eqref{lario11} is bounded by
$$
\iint_{{\cal V} (\xi _{\varepsilon ,i})\times {\cal V}(\xi _{\varepsilon ,i})} K_{2,\ell}(x,y) dxdy \le \ell^4 \cdot (\pi \varepsilon^2)^2=O(1),
$$
moreover, by \cite[Proposition 1.1]{CMW}, for the expectation in \eqref{lario11} we have
$$
\mathbb{E}\left[ \mathcal{N}^{c}(f_{\ell }; {\cal V}(\xi _{\varepsilon ,i})) \right] \le \mathbb{E}\left[ \mathcal{N}^{c}(f_{\ell }; {\cal B}(\xi _{\varepsilon ,i}))
\right] \le \pi \varepsilon^2 \ell^2 =O(1).
$$
Then, using \eqref{csbound}, we can bound the covariance as
\begin{equation*}
%\label{csbound}
\left\vert {\rm Cov} \left( \mathcal{N}^{c}(f_{\ell }; {\cal V} (\xi _{i,\varepsilon })),%
\mathcal{N}^{c}(f_{\ell }; {\cal V}(\xi _{j,\varepsilon }))\right) \right\vert
= O(1),
\end{equation*}
and since by \eqref{nett} there are $O(\ell ^{2})$ pairs of Voronoi cells
at distance $d(x,y) \in [0, C/\ell] \cup [\pi- C/\ell, \pi]$%\cite[Lemma 5]{BKMP}
, we finally obtain
\begin{equation*}
\sum_{d( {\cal V}(\xi _{i,\varepsilon }),{\cal V}(\xi _{j,\varepsilon })) \in [0, C/\ell] \cup [\pi- C/\ell, \pi]
}\left\vert {\rm Cov}\left( \mathcal{N}^{c}(f_{\ell }; {\cal V}(\xi _{i,\varepsilon })),%
\mathcal{N}^{c}(f_{\ell }; {\cal V}(\xi _{j,\varepsilon }))\right) \right\vert
=O(\ell ^{2}).
\end{equation*}

\section{Proof of Theorem \ref{th1}}

\subsection{Kac-Rice formula in coordinate system} \label{K-R_coord}
To study the asymptotic behaviour of the two-point correlation
function we
write a more explicit frame-dependent formula by using the orthogonal frames \eqref{orthogonalf} so that, by the isotropic property of $f_{\ell}$,
$K_{2,\ell}$ depends only on the geodesic distance $\phi=d(x,y)$.

 For $x,y\in {\cal S}^2$ we define the following random vector
\begin{equation*}
Z_{\ell ;x,y}=(\nabla f_{\ell }(x),\nabla f_{\ell }(y),\nabla ^{2}f_{\ell
}(x),\nabla ^{2}f_{\ell }(y)).
\end{equation*}
To write the Kac-Rice formula in coordinate system, given $x$, $y \in {\cal S}^2$, we consider two local orthogonal frames
$\{e_{1}^{x}, e_{2}^{x}\}$ and
$\{e_{1}^{y},e_{2}^{y}\}$ defined in some neighbourhood of $x$ and $y$ respectively.
This gives rise to the (local) identifications
\begin{equation} \label{ident_isom}
T_{x}({\cal S}^2)\cong \mathbb{R}^{2}\cong T_{y}({\cal S}^2),
\end{equation}
so that we do not have to work with
probability densities defined on tangent planes which depend on
the points $x$ and $y$ respectively. Under the identification
\eqref{ident_isom} the random vector $Z_{\ell ;x,y}$ is a
$\mathbb{R}^{10}$ centred Gaussian random vector.

By the isotropic property of $f_{\ell}$
it is convenient to perform our computations along a specific
geodesic. In particular, we focus on the equatorial line $x=(\pi
/2,\phi )$, $y=(\pi /2,0)$ and we work with the orthogonal frames
\begin{equation} \label{orthogonalf}
\left\{e_{1}^{x}=\frac{\partial }{\partial \theta _{x}},\;e_{2}^{x}=\frac{
\partial }{\partial \varphi _{x}}\right\},\hspace{0.5cm}\left\{e_{1}^{y}=\frac{
\partial }{\partial \theta _{y}},\;e_{2}^{y}=\frac{\partial }{\partial
\varphi _{y}}\right\}.
\end{equation}

 Let $\Delta_{\ell}(\phi)$ be the conditional covariance matrix of the scaled Gaussian vector
 \begin{equation*}
\frac{\sqrt 8}{ \lambda_{\ell}} (\nabla ^{2}f_{\ell }(x),\nabla ^{2}f_{\ell }(y)\big| \nabla f_{\ell}(x)=\nabla f_{\ell }(y)={\bf 0}).
\end{equation*}
With the choice \eqref{orthogonalf} the covariance matrix $\Delta_{\ell}(\phi)$ is of the following form
\begin{equation*} %\label{1agosto}
\Delta _{\ell }(\phi )=\left(
\begin{array}{cc}
\Delta _{1,\ell }(\phi ) & \Delta _{2,\ell }(\phi ) \\
\Delta _{2,\ell }(\phi ) & \Delta _{1,\ell }(\phi )%
\end{array}%
\right),
\end{equation*}%
where
\begin{align} \label{D1}
\Delta _{1,\ell }(\phi )&=\left(
\begin{array}{ccc}
3-\frac{16\beta _{2,\ell }^{2} (\phi)}{\lambda _{\ell } (\lambda _{\ell
}^{2}-4\alpha _{2,\ell }^{2} (\phi))}-\frac{2}{\lambda _{\ell }} & 0 & 1-\frac{%
16\beta _{2,\ell } (\phi) \beta _{3,\ell }(\phi)}{\lambda _{\ell }(\lambda _{\ell
}^{2}-4\alpha _{2,\ell }^{2}(\phi))}+\frac{2}{\lambda _{\ell }} \\
0 & 1-\frac{16\beta _{1,\ell }^{2}(\phi)}{\lambda _{\ell } (\lambda _{\ell
}^{2}-4\alpha _{1,\ell }^{2}(\phi))} & 0 \\
1-\frac{16\beta _{2,\ell }(\phi)\beta _{3,\ell }(\phi)}{\lambda _{\ell } (\lambda
_{\ell }^{2}-4\alpha _{2,\ell }^{2}(\phi))}+\frac{2}{\lambda _{\ell }} & 0 & 3-%
\frac{16\beta _{3,\ell }^{2}(\phi)}{\lambda _{\ell } (\lambda _{\ell
}^{2}-4\alpha _{2,\ell }^{2}(\phi))}-\frac{2}{\lambda _{\ell }}%
\end{array}%
\right),
\end{align}%
\begin{align} \label{D2}
\Delta _{2,\ell }(\phi )&=\left(
\begin{array}{ccc}
8\frac{\gamma _{1,\ell }(\phi)+\frac{4\alpha _{2,\ell }(\phi)\beta _{2,\ell }^{2}(\phi)}{
4\alpha _{2,\ell }^{2}(\phi)-\lambda _{\ell }^{2}}}{\lambda _{\ell }^2} & 0 & 8\frac{\gamma _{3,\ell }(\phi)+\frac{4\alpha _{2,\ell }(\phi)\beta
_{2,\ell }(\phi)\beta _{3,\ell }(\phi)}{4\alpha _{2,\ell }^{2}(\phi)-\lambda _{\ell }^{2}}}{\lambda _{\ell }^2} \\
0 & 8\frac{\gamma _{2,\ell }(\phi)+\frac{4\alpha _{1,\ell }(\phi)\beta _{1,\ell }^{2}(\phi)}{%
4\alpha _{1,\ell }^{2}(\phi)-\lambda _{\ell }^{2}}}{\lambda _{\ell }^2} & 0 \\
8\frac{\gamma _{3,\ell }(\phi)+\frac{4\alpha _{2,\ell }(\phi)\beta _{2,\ell }(\phi)\beta
_{3,\ell }(\phi)}{4\alpha _{2,\ell }^{2}(\phi)-\lambda _{\ell }^{2}}}{\lambda _{\ell }^2} & 0 & 8\frac{\gamma _{4,\ell }(\phi)+\frac{4\alpha _{2,\ell
}(\phi) \beta _{3,\ell }^{2}(\phi)}{4\alpha _{2,\ell }^{2}(\phi)-\lambda _{\ell }^{2}}}{\lambda _{\ell }^2}%
\end{array}%
\right)
\end{align}
with
\begin{align*}
\alpha _{1,\ell }(\phi )&=P_{\ell }^{\prime }(\cos \phi ),\hspace{0.3cm}
\alpha _{2,\ell }(\phi )=-\sin ^{2}\phi P_{\ell }^{\prime \prime }(\cos \phi
)+\cos \phi P_{\ell }^{\prime }(\cos \phi ),
\end{align*}

\begin{align*}
\beta _{1,\ell }(\phi )&=\sin \phi P_{\ell }^{\prime \prime }(\cos \phi ),\hspace{0.3cm}
\beta _{2,\ell }(\phi )=\sin \phi \cos \phi P_{\ell }^{\prime \prime }(\cos
\phi )+\sin \phi P_{\ell }^{\prime }(\cos \phi ),\\
\beta _{3,\ell }(\phi )&=-\sin ^{3}\phi P_{\ell }^{\prime \prime \prime
}(\cos \phi )+3\sin \phi \cos \phi P_{\ell }^{\prime \prime }(\cos \phi
)+\sin \phi P_{\ell }^{\prime }(\cos \phi ),
\end{align*}

\begin{align*}
\gamma _{1,\ell }(\phi )&=(2+\cos ^{2}\phi )P_{\ell }^{\prime \prime }(\cos
\phi )+\cos \phi P_{\ell }^{\prime }(\cos \phi ), \hspace{0.3cm}
\gamma _{2,\ell }(\phi )=-\sin ^{2}\phi P_{\ell }^{\prime \prime \prime
}(\cos \phi )+\cos \phi P_{\ell }^{\prime \prime }(\cos \phi ), \\
\gamma _{3,\ell }(\phi )&=-\sin ^{2}\phi \cos \phi P_{\ell }^{\prime \prime
\prime }(\cos \phi )+(-2\sin ^{2}\phi +\cos ^{2}\phi )P_{\ell }^{\prime
\prime }(\cos \phi )+\cos \phi P_{\ell }^{\prime }(\cos \phi ), \\
\gamma _{4,\ell }(\phi )&=\sin ^{4}\phi P_{\ell }^{\prime \prime \prime
\prime }(\cos \phi )-6\sin ^{2}\phi \cos \phi P_{\ell }^{\prime \prime
\prime }(\cos \phi )+(-4\sin ^{2}\phi +3\cos ^{2}\phi )P_{\ell }^{\prime \prime }(\cos \phi
)+\cos \phi P_{\ell }^{\prime }(\cos \phi ).
\end{align*}
For a proof of \eqref{D1} and \eqref{D2} we refer to \cite[Appendix A and Appendix B]{CMW}. We also introduce the vector $\bf{a}$ that collects
the perturbing elements of the covariance matrix $\Delta_{\ell}(\phi)$:
$$\mathbf{a}=\mathbf{a}_{\ell}(\phi)=(a_{1,\ell}(\phi), a_{2,\ell}(\phi),a_{3,\ell}(\phi),a_{4,\ell}(\phi),a_{5,\ell}(\phi),a_{6,\ell}(\phi),a_{7,\ell}(\phi),a_{8,\ell}(\phi))$$
with $a_{i,\ell}(\phi)$, $i=1,\dots,8$, defined by
\begin{align*}
\Delta _{1,\ell }(\phi )=\left(
\begin{array}{ccc}
3+a_{1,\ell }(\phi ) & 0 & 1+a_{4,\ell }(\phi ) \\
0 & 1+a_{2,\ell }(\phi ) & 0 \\
1+a_{4,\ell }(\phi ) & 0 & 3+a_{3,\ell }(\phi )
\end{array}
\right), \;\;\;\;
\Delta _{2,\ell }(\phi )=\left(
\begin{array}{ccc}
a_{5,\ell }(\phi ) & 0 & a_{8,\ell }(\phi ) \\
0 & a_{6,\ell }(\phi ) & 0 \\
a_{8,\ell }(\phi ) & 0 & a_{7,\ell }(\phi )%
\end{array}%
\right).
\end{align*}\\

\noindent In what follows, with a slight abuse of notation, we write the conditional covariance matrix
$\Delta_{\ell}(\phi)$ as a function of $\mathbf{a}$
\begin{equation*}
\Delta_{\ell}(\phi)=\Delta (\mathbf{a}_{\ell}(\phi))=\Delta(\mathbf{a})=\left(
\begin{array}{cc}
\Delta _{1}(\mathbf{a}) & \Delta _{2}(\mathbf{a}) \\
\Delta _{2}(\mathbf{a}) & \Delta _{1}(\mathbf{a})%
\end{array}
\right),
\end{equation*}
where
\begin{equation*}
\Delta _{1}(\mathbf{a})=\left(
\begin{array}{ccc}
3+a_{1} & 0 & 1+a_{4} \\
0 & 1+a_{2} & 0 \\
1+a_{4} & 0 & 3+a_{3}%
\end{array}
\right), \;\; \;\;\;\;\;
\Delta _{2}(\mathbf{a})=\left(
\begin{array}{ccc}
a_{5} & 0 & a_{8} \\
0 & a_{6} & 0 \\
a_{8} & 0 & a_{7}%
\end{array}%
\right).
\end{equation*}\\

 At this point we may write the $2$-point correlation function $K_{2,\ell}$ in \eqref{eq:K2 glob gen} as a function
of the perturbing elements $a_{i,\ell}(\phi)$, $i=1,\dots,8$ of the covariance matrix:
\begin{equation*}
K_{2,\ell }(\phi) = \frac{\lambda_{\ell}^4}{8^2} \frac{1}{(2 \pi)^{2} \sqrt{{\rm det}(A_{\ell}(\phi))}} q (\mathbf{a}_{\ell}(\phi)),
\end{equation*}
where $A_{\ell}(\phi)$ is the covariance matrix of the Gaussian random vector $(\nabla f_{\ell}(x), \nabla f_{\ell }(y))$
\begin{equation*}
A_{\ell }(\phi )=\left(
\begin{array}{cccc}
\frac{\lambda _{\ell }}{2} & 0 & \alpha _{1,\ell }(\phi ) & 0 \\
0 & \frac{\lambda _{\ell }}{2} & 0 & \alpha _{2,\ell }(\phi ) \\
\alpha _{1,\ell }(\phi ) & 0 & \frac{\lambda _{\ell }}{2} & 0 \\
0 & \alpha _{2,\ell }(\phi ) & 0 & \frac{\lambda _{\ell }}{2}
\end{array}
\right),
\end{equation*}
see \cite[Appendix B]{CMW}, and $q (\mathbf{a}_{\ell}(\phi))$ is the conditional expectation
\begin{align*}
q (\mathbf{a}_{\ell}(\phi))&= \frac{1}{(2 \pi)^3 \sqrt{{\rm det}(\Delta_{\ell}(\phi) )}} \iint_{\R^3 \times \R^3} |z_1 z_3 - z_2^2| \cdot |w_1 w_3 - w_2^2| \\
&\;\; \times \exp \Big\{- \frac 1 2 (z_1,z_2,z_3, w_1,w_2,w_3 ) \Delta_{\ell}(\phi)^{-1} (z_1,z_2,z_3, w_1,w_2,w_3 )^t \Big\} d z_1 dz_2 dz_3 dw_1 d w_2 d w_3.
\end{align*}
The determinant of $A_{\ell}(\phi)$ can be easily computed so that we obtain
\begin{equation} \label{eq:=K2=*q}
K_{2,\ell }(\phi) = \frac{\lambda_{\ell}^4}{8^2} \frac{1}{ \pi ^{2} \sqrt{(\lambda _{\ell }^{2}-4\alpha _{2,\ell }^{2}(\phi
))(\lambda _{\ell }^{2}-4\alpha _{1,\ell }^{2}(\phi )) }} q (\mathbf{a}_{\ell}(\phi)).
\end{equation}

\subsection{Taylor expansion of the two-point correlation function} \label{sec-taylor}

To study the asymptotic behaviour of the variance
in the long-range regime, we investigate now the asymptotic
behaviour of \eqref{afkao2}, i.e., the high energy asymptotic
behaviour of
\begin{align} \label{sitt}
 \frac{\lambda_{\ell}^2}{8} \int_{C/\ell}^{\pi - C/\ell} \frac{\sin \phi}{ \sqrt{(1-4\alpha _{2,\ell }^{2}(\phi)/\lambda _{\ell }^{2})(1-4\alpha _{1,\ell }^{2}(\phi )/\lambda _{\ell }^{2})} }
 q (\mathbf{a}_{\ell}(\phi)) d \phi - (\E[{\cal N}^c (f_{\ell})])^2.
\end{align}
In the range $\phi \in (C/\ell, \pi- C/\ell)$ the conditional covariance matrix $\Delta_{\ell}(\phi)=\Delta({\mathbf a})$ %of the scaled Gaussian vector
% \begin{equation*}
%\frac{\sqrt 8}{ \lambda_{\ell}} (\nabla ^{2}f_{\ell }(x),\nabla ^{2}f_{\ell }(y)\big| \nabla f_{\ell}(x)=\nabla f_{\ell }(y)={\bf 0}),
%\end{equation*}
is a small perturbation of the $6 \times 6$ matrix $U$ where
\begin{equation*}
U=\left(
\begin{array}{cc}
U_1 & {\mathbf 0} \\
{\mathbf 0} & U_1
\end{array}
\right), \hspace{1cm}
U_1=\left(
\begin{array}{ccc}
3 & 0 & 1 \\
0 & 1 & 0 \\
1 & 0 & 3
\end{array}
\right).
\end{equation*}
The elements $a_i$, $i=1,\dots, 8$ are in fact uniformly small for $\phi \in ( C/\ell, \pi- C/\ell)$, see Lemma \ref{perturbing} below. 
Consequently we may use perturbation theory \cite[Theorem 1.5]{kato} to
yield that the Gaussian expectation $q$ is an analytic functions of the perturbing elements $a_i$, $i=1,\dots, 8$ and we can expand it into a Taylor polynomial around $\mathbf{a}=0$ as follows: 

\begin{align}\label{noemi}
& \int_{C/\ell}^{\pi - C/\ell} \frac{\sin \phi}{ \sqrt{(1-4\alpha _{2,\ell }^{2}(\phi)/\lambda _{\ell }^{2})(1-4\alpha _{1,\ell }^{2}(\phi )/\lambda _{\ell }^{2})} } q (\mathbf{a}_{\ell}(\phi)) d \phi \nonumber \\
& = A_{0,\ell } \; q
(\mathbf{0})+\sum_{i_1=1}^{8} A_{i_1,\ell
}\;\Big[\frac{\partial }{\partial a_{i_1} }q (\mathbf{a} )\Big]_{\mathbf{a}=\mathbf{0}}
+\frac{1}{2}\sum_{i_1,i_2=1}^{8}A_{i_1 i_2,\ell }\;\Big[\frac{\partial
^{2}}{\partial a_{i_1} \partial a_{i_2}}q (\mathbf{a}
)\Big]_{\mathbf{a}=\mathbf{0}} \nonumber \\
&\;\;+\frac{1}{3!}\sum_{i_1,i_2,i_3=1}^{8}A_{i_1 i_2 i_3,\ell }\;\Big[\frac{\partial
^{3}}{\partial a_{i_1} \partial a_{i_2} \partial a_{i_3}}q (\mathbf{a}
)\Big]_{\mathbf{a}=\mathbf{0}} +\frac{1}{4!}\sum_{i_1,i_2,i_3,i_4=1}^{8}A_{i_1 i_2 i_3 i_4,\ell }\;\Big[\frac{\partial
^{4}}{\partial a_{i_1} \partial a_{i_2} \partial a_{i_3} \partial a_{i_4} }q (\mathbf{a}
)\Big]_{\mathbf{a}=\mathbf{0}} \nonumber \\
&\;\; +\int_{C/\ell }^{\pi -C/\ell}\frac{\sin \phi
}{\sqrt{(1-4\alpha _{2,\ell }^{2}(\phi )/\lambda _{\ell
}^{2})(1-4\alpha _{1,\ell }^{2}(\phi )/\lambda _{\ell
}^{2})}}O(||\mathbf{a}||^{5})\;d\phi,
\end{align}
where we adopted the following notation: for $i_1,i_2,i_3,i_4=1,\dots, 8$,
\begin{align*}
A_{0,\ell }& =\int_{C/\ell }^{\pi -C/\ell }\frac{1 }{\sqrt{%
(1-4\alpha _{2,\ell }^{2}(\phi )/\lambda _{\ell }^{2})(1-4\alpha _{1,\ell
}^{2}(\phi )/\lambda _{\ell }^{2})}}\sin \phi \; d\phi , \\
A_{i_1,\ell }& =\int_{C/\ell }^{\pi -C/\ell }\frac{a_{i_1,\ell }(\phi )}{\sqrt{%
(1-4\alpha _{2,\ell }^{2}(\phi )/\lambda _{\ell }^{2})(1-4\alpha _{1,\ell
}^{2}(\phi )/\lambda _{\ell }^{2})}}\sin \phi \;d\phi, \\
A_{i_1, \dots i_k,\ell }& =\int_{C/\ell }^{\pi -C/\ell }\frac{a_{i_1,\ell }(\phi
) \cdots a_{i_k,\ell }(\phi )}{\sqrt{(1-4\alpha _{2,\ell }^{2}(\phi )/\lambda _{\ell
}^{2})(1-4\alpha _{1,\ell }^{2}(\phi )/\lambda _{\ell }^{2})}}\sin \phi
\;d\phi, \hspace{0.5cm} k=2,3,4. %\\
%A_{ijk,\ell }& =\int_{C/\ell }^{\pi -C/\ell }\frac{a_{i,\ell }(\phi
%)a_{j,\ell }(\phi ) a_{k,\ell }(\phi ) }{\sqrt{(1-4\alpha _{2,\ell }^{2}(\phi )/\lambda _{\ell
%}^{2})(1-4\alpha _{1,\ell }^{2}(\phi )/\lambda _{\ell }^{2})}}\sin \phi
%\;d\phi, \\
%A_{ijkl,\ell }& =\int_{C/\ell }^{\pi -C/\ell }\frac{a_{i,\ell }(\phi
%)a_{j,\ell }(\phi ) a_{k,\ell }(\phi ) a_{l,\ell }(\phi ) }{\sqrt{(1-4\alpha _{2,\ell }^{2}(\phi )/\lambda _{\ell
%}^{2})(1-4\alpha _{1,\ell }^{2}(\phi )/\lambda _{\ell }^{2})}}\sin \phi
%\;d\phi,
\end{align*}
Note that to obtain the exact asymptotic behaviour of the variance of the total number
of critical point we need to sharpen the bounds obtained in \cite{CMW}; for this reason we have expanded
$q$ in \eqref{noemi} up to order four (instead of order three as in \cite{CMW}).

\subsection{Asymptotics for the two-point correlation function} \label{sec-asymp}
We now study the decay rate of $A_{i_1, \dots i_k,\ell }$.
In particular, we improve the bounds obtained in \cite[Lemma 4.3 and Lemma 4.4]{CMW} for the $A_{i_1, \dots i_k,\ell }$ to $O(\ell^2)$.
Such refinement requires a more careful investigation of the tail decay of the perturbing elements $\bf{a}_{\ell}(\phi)$ of $\Delta_{\ell}(\phi)$ that are expressed in terms of the first four derivatives of Legendre polynomials as shown in \eqref{D1}-\eqref{D2}.

The tail decay, for $\ell \to \infty$, of the first four derivatives of Legendre polynomials, is derived in Appendix \ref{horror} using the high degree asymptotics of the Legendre polynomials and their derivatives, i.e., Hilb asymptotics. In particular, to improve the bounds obtained in \cite{CMW}, we apply here a more general version of the Hilb asymptotic derived in \cite[Lemma 1]{frenzen&wong} (see also \cite[Theorem 8.21.5]{szego}).

 All the work for establishing the asymptotics of the perturbing elements $\bf{a}_{\ell}(\phi)$ (see Lemma \ref{perturbing} in the next section) leads to the high energy asymptotic behaviour of the terms $A_{i_1, \dots i_k,\ell }$ in Lemma \ref{L-zero} and Lemma \ref{L-uno}. In particular we see that the main contribution
to the $A_{i_1, \dots i_k,\ell }$ comes from the leading non-oscillatory terms in the Taylor expansion \eqref{noemi}, so we obtain Lemma \ref{L-zero} and Lemma \ref{L-uno}
by bounding the contribution of the oscillatory terms and error terms.\\

 We first show that the first term in the expansion \eqref{noemi} cancels out the squared expectation in \eqref{sitt}:
\begin{lemma} \label{L-zero}
As $\ell \to \infty$, 
\begin{align*}
\frac{\lambda_{\ell }^{2} }{8} A_{0,\ell } \, q(\mathbf{0})-(\mathbb{E}\left[ {\cal N}^c(f_{\ell}) \right])^{2}= \frac{1}{8} \left[ 2 \ell^3 +\frac{2 \cdot 3^2}{\pi^2} \ell^2 \log \ell \right] \, q(\mathbf{0}) +O(\ell^2).
\end{align*}
\end{lemma}

\vspace{0.5cm}
 Then we study the high frequency asymptotic behaviour of the other terms:
\begin{lemma} \label{L-uno}
As $\ell \to \infty$, for $i \ne 3$, we have
%\begin{align*}
 $\lambda _{\ell }^{2} A_{i,\ell}=O(\ell^{2})$,
%\end{align*}
whereas for $i=3$, we get
\begin{align*}
 \lambda _{\ell }^{2} A_{3,\ell} = \left[ - 16 \ell^3 - \frac{2^5 \cdot 3 }{\pi^2}\ell^2 \log \ell \right] +O(\ell^{2}),
%, \hspace{0.5cm}
%2 \lambda _{\ell }^{2} A_{7,\ell
%} = 16 \sqrt{\frac 2 \pi} \sin(\frac{\ell \pi}{2}) \; \ell^{5/2} +O(\ell^{2}).
\end{align*}
%\end{lemma}
%\begin{lemma} \label{L-due}
for $(i,j) \ne (3,3), (7,7)$ we have
%\begin{align*}
 $\lambda _{\ell }^{2} A_{ij,\ell}=O(\ell^{2})$,
%\end{align*}
instead for $(i,j)=(3,3), (7,7)$ we have
\begin{align*}
 \lambda _{\ell }^{2} \frac 1 2 A_{33,\ell
} = \frac{3 \cdot 2^7}{\pi^2} \ell^2 \log \ell +O(\ell^{2}), \hspace{0.5cm}
 \lambda _{\ell }^{2} \frac 1 2 A_{77,\ell
} = \left[ 32 \ell^3 - \frac{ 2^6}{\pi^2} \ell^2 \log \ell \right] +O(\ell^{2}),
\end{align*}
%\end{lemma}
%\begin{lemma} \label{L-tre}
 for $(i,j,k)\ne(3,7,7)$ we have
%\begin{align*}
 $\lambda _{\ell }^{2} A_{ijk,\ell}=O(\ell^{2})$,
%\end{align*}
and
\begin{align*}
\lambda _{\ell }^{2} \frac {3}{3!} A_{377,\ell
} =- \frac{2^9}{ \pi^2} \ell^2 \log \ell +O(\ell^{2}),
\end{align*}
%\end{lemma}
%\begin{lemma} \label{L-quattro}
for $(i,j,k,l)\ne(7,7,7,7)$ we have $\lambda _{\ell }^{2} A_{ijkl,\ell}=O(\ell^{2})$, whereas
\begin{align*}
&\lambda _{\ell }^{2} \frac 1 {4!} A_{7777,\ell
} = \frac{ 2^{9}}{\pi^2} \ell^2 \log \ell +O(\ell^{2}).
\end{align*}
\end{lemma}
\noindent The proofs of Lemma \ref{L-zero} and Lemma \ref{L-uno} are postponed to Section \ref{auxiliary}.\\

 In view of Lemma \ref{L-zero} and Lemma \ref{L-uno}, we immediately see that, as $\ell \to \infty$, \eqref{sitt} has the following leading terms

\begin{align} \label{leading_var}
\text{Var}({\cal N}^c(f_{\ell}))&= \frac 1 8 \left[ 2 \ell^3 + \frac{2 \cdot 3^2}{ \pi^2} \ell^2 \log \ell \right] q(\mathbf{0})
+ \frac 1 8 \left[ - 16 \ell^3 - \frac{2^5 \cdot 3 }{\pi^2} \ell^2 \log \ell \right] \; \Big[\frac{\partial }{\partial a_{3} } q (\mathbf{a})\Big]_{\mathbf{a} =\mathbf{0}} \nonumber \\
&+ \frac 1 8 \left[ 32 \ell^3 - \frac{ 2^6}{\pi^2} \ell^2 \log \ell \right] \Big[\frac{\partial^2 }{\partial a_{7} \partial a_{7} } q (\mathbf{a} )\Big]_{\mathbf{a} =\mathbf{0}} + \frac 1 8 \left[ \frac{ 3 \cdot 2^7}{\pi^2} \ell^2 \log \ell \right] \Big[\frac{\partial^2 }{\partial a_{3} \partial a_{3} } q (\mathbf{a})\Big]_{\mathbf{a} =\mathbf{0}} \nonumber \\
&+ \frac 1 8 \left[ - \frac{ 2^9}{ \pi^2} \ell^2 \log \ell \right] \Big[\frac{\partial^3 }{\partial a_{3} \partial a_{7} \partial a_{7} } q (\mathbf{a})\Big]_{\mathbf{a} =\mathbf{0}}
+ \frac 1 8 \left[ \frac{ 2^9}{ \pi^2} \ell^2 \log \ell \right] \Big[\frac{\partial^4 }{\partial a_{7}\partial a_{7} \partial a_{7} \partial a_{7} } q (\mathbf{a})\Big]_{\mathbf{a} =\mathbf{0}} +O(\ell^2).
\end{align}

\vspace{0.4cm}

\subsection{Evaluation of the leading constant} \label{eval-const}

Let $Y=(Y_{1},Y_{2},Y_{3})$ be a centred jointly Gaussian random vector with covariance matrix
\begin{equation*}
U_1=\left(
\begin{array}{ccc}
3 & 0 & 1 \\
0 & 1 & 0 \\
1 & 0 & 3%
\end{array}%
\right),
\end{equation*}%
and let ${\cal I}_r$, $r=0,2,4$, be the Gaussian expectations of the form
$${\cal I}_r=\E[|Y_1 Y_3-Y_2^2| (Y_1-3 Y_3)^r].$$

The relevant derivatives in \eqref{leading_var} and $q(\mathbf{0})$ are evaluated in the following two lemmas. We first note that
\begin{align} \label{qhq0}
q(\mathbf{0}) = (\E[|Y_1 Y_3-Y_2^2| ])^2={\cal I}_0^2.
\end{align}
\begin{lemma} \label{der_q_hat}
One has
\begin{align}
 &\Big[\frac{\partial }{\partial a_{3} } q (\mathbf{a})\Big]_{\mathbf{a} =\mathbf{0}}=\frac{1}{2^3}[ - 3 {\cal I}_0^2+\frac{1}{2^3} {\cal I}_0 {\cal I}_2], \label{d_q_h} \\
 &\Big[\frac{\partial^2 }{\partial a_{7} \partial a_{7} } q (\mathbf{a} )\Big]_{\mathbf{a} =\mathbf{0}}=\frac{1}{2^6} [3 {\cal I}_0-\frac{1}{2^3} {\cal I}_2]^2,\\
&\Big[\frac{\partial^2 }{\partial a_{3} \partial a_{3} } q (\mathbf{a})\Big]_{\mathbf{a} =\mathbf{0}}=\frac{1}{2^{11}} [2^6 \cdot 3^2 {\cal I}_0^2- 2^4 \cdot 3 {\cal I}_0 {\cal I}_2+ \frac{1}{2^2} {\cal I}_0 {\cal I}_4+\frac{1}{2^2} {\cal I}_2^2], \\
&\Big[\frac{\partial^3 }{\partial a_{3} \partial a_{7} \partial a_{7} } q (\mathbf{a})\Big]_{\mathbf{a} =\mathbf{0}}=\frac{1}{2^{19}} [-2^{10} \cdot 3^4 {\cal I}_0^2+2^7 \cdot 3^4 {\cal I}_0 {\cal I}_2- 2^4 \cdot 3 {\cal I}_0 {\cal I}_4+ 2 {\cal I}_2 {\cal I}_4- 2^{5} \cdot 3^2 {\cal I}_2^2 ], \\
&\Big[\frac{\partial^4 }{\partial a_{7}\partial a_{7} \partial a_{7} \partial a_{7} } q (\mathbf{a})\Big]_{\mathbf{a} =\mathbf{0}} = \frac{1}{2^{24}} [2^6 \cdot 3^3 {\cal I}_0+{\cal I}_4- 2^4 \cdot 3^2 {\cal I}_2]^2. \label{d_q_h_1}
\end{align}
\end{lemma}

\noindent Substituting \eqref{qhq0} and \eqref{d_q_h}-\eqref{d_q_h_1} into \eqref{leading_var} we obtain the following simple form for the variance
\begin{align} \label{vvv}
\text{Var}({\cal N}^c(f_{\ell}))&= \frac{[{\cal I}_2- 2^3 \cdot 5 \; {\cal I}_0]^2}{2^{10}} \ell^3 + \frac{[2^6 \cdot 3 \cdot 17 \; {\cal I}_0-2^4 \cdot 11\; {\cal I}_2+{\cal I}_4]^2}{2^{18}\, \pi^2 } \ell^2 \log \ell +O(\ell^2).
\end{align}

In the next lemma we compute the Gaussian expectations ${\cal I}_r$, $r=0,2,4$.

\begin{lemma} \label{G-e}
One has
$$
{\cal I}_0=\frac{2^2}{\sqrt 3}, \hspace{1cm} {\cal I}_2=\frac{2^5 \cdot 5}{\sqrt 3}, \hspace{1cm} {\cal I}_4=\frac{2^8 \cdot 5^2 \cdot 7}{3 \sqrt 3}.
$$
\end{lemma}
\noindent The proofs of Lemma \ref{der_q_hat} and Lemma \ref{G-e} are postponed to the next section. \\

The statement of Theorem \ref{th1} now follows upon substituting the values of ${\cal I}_r$, obtained in Lemma \ref{G-e}, into \eqref{vvv}.

%\begin{remark}
%Introducing the corresponding conditions on the Hessian in \eqref{eq:K2 glob gen} to separate the critical points into extrema and saddles and following the lines of the previous proof, it is easy to derive the analogous results for the variance of total number of extrema and saddles in \eqref{N} and \eqref{NN}. In fact, the main difference in the proof of \eqref{N} and \eqref{NN} consists of evaluating the leading constants in terms of Gaussian expectations of the form ${\cal I}^e_r$ and ${\cal I}^s_r$ for extrema and saddles respectively, where
%$${\cal I}^e_r=\E[|Y_1 Y_3-Y_2^2| \ind_{\{Y_1 Y_3-Y_2^2>0\}}(Y_1-3 Y_3)^r], \hspace{0.5cm} {\cal I}^s_r=\E[|Y_1 Y_3-Y_2^2| \ind_{\{Y_1 Y_3-Y_2^2<0\}}(Y_1-3 Y_3)^r], \hspace{0.3cm} r=0,2,4.$$
%\end{remark}

\section{Proofs of auxiliary lemmas} \label{auxiliary}

To prove Lemma \ref{L-zero} and Lemma \ref{L-uno} we first derive, in the next lemma, the asymptotic behaviour of the terms appearing in the perturbing elements of the covariance matrix $\Delta_{\ell}(\phi)$.

\begin{lemma} \label{perturbing}
Let $h_0(0)$, $h_0(1)$ and $h_{1}(0)$ be the constants $h_0(0)=\sqrt{\frac 2 \pi}$, $h_0(1)=-\frac 1 8 \sqrt{\frac 2 \pi}$, $h_{1}(0)= \sqrt{\frac 2 \pi}$ and let $\psi_{n,\ell+u}$ be the functions $\psi_{n,\ell+u}=(\ell+u+1/2) \phi- n \pi /2-\pi/4$ where $\ell \ge 1$, $n,u=0,1$, $\phi \in [C/\ell, \pi/2]$ and $C$ be any positive constant. We have the following estimates.

For $\alpha_{i,\ell}(\phi)$, $i=1,2$, we get
\begin{align}
\frac{\alpha^2_{1,\ell}(\phi)}{\ell^2 (\ell+1)^2}&=\phi^{-3} O(\ell^{-3}), \label{a1}\\
\frac{\alpha^2_{2,\ell}(\phi)}{\ell^2 (\ell+1)^2}&=h^2_0(0) \cos^2 \psi_{0,\ell} \frac{1}{\ell \sin \phi} + 2 h^2_0(0) \sin \psi_{0,\ell+1} \cos \psi_{0,\ell} \frac{1}{\ell^2 \sin^2 \phi}\label{a2} \\
&\;\;- h_0(0) h_0(1) \sin \psi_{0,\ell} \cos \psi_{0,\ell} \frac{1}{\ell^2 \phi \sin \phi} - h^2_0(0) \sin \psi_{0,\ell-1} \cos \psi_{0,\ell} \frac{1}{\ell^2 \sin^2 \phi} \nonumber\\
&\;\;+ h^2_0(0) \cos \phi \cos \psi_{0,\ell} \sin \psi_{0,\ell} \frac{1}{\ell^2 \sin^2 \phi}+\phi^{-1} O(\ell^{-2}), \nonumber \\
\frac{\alpha^4_{2,\ell}(\phi)}{\ell^4 (\ell+1)^4}&=h^4_0(0) \cos^4 \psi_{0,\ell} \frac{1}{\ell^2 \sin^2 \phi} + \phi^{-3} O(\ell^{-3}). \label{a3}
\end{align}

For $\beta_{i,\ell}(\phi)$, $i=1,2,3$, we have
\begin{align}
\frac{\beta^2_{1,\ell}(\phi)}{\ell^{3} (\ell+1)^{3}}&=\phi^{-3} O(\ell^{-3}), \label{b12}\\
\frac{\beta_{2,\ell}(\phi)}{\ell^{3/2} (\ell+1)^{3/2}}&=-h_0(0) \cos \psi_{0,\ell} \cos \phi \frac{1}{\ell^{1+1/2} \sin^{1+1/2} \phi}+ \phi^{-2-1/2} O(\ell^{-2-1/2}), \label{b2}\\
\frac{\beta^2_{2,\ell}(\phi)}{\ell^{3} (\ell+1)^{3}}&=\phi^{-3} O(\ell^{-3}),\label{b22}\\
\frac{\beta_{3,\ell}(\phi)}{\ell^{3/2} (\ell+1)^{3/2}}&=-h_0(0) \sin \psi_{0,\ell} \frac{1}{\ell^{1/2} \sin^{1/2} \phi} \sum_{j=0}^1 \binom{-1/2}{j} \frac{1}{2^j} \frac{1}{\ell^j} + \frac 3 2 h_0(0) \sin \psi_{0,\ell} \frac{1}{\ell^{1+1/2} \sin^{1/2} \phi} \label{b3} \\
&\;\;+3 h_0(0) \cos \psi_{0,\ell+1} \frac{1}{\ell^{1+1/2} \sin^{1+1/2} \phi}-h_0(0) \sin \psi_{0,\ell} \frac{1}{\ell^{1+1/2} \sin^{1/2}} \nonumber \\
&\;\;- h_0(1) \cos \psi_{0,\ell} \frac{1}{\ell^{1+1/2} \phi \sin^{1/2} \phi}+ h_1(0) \sin \psi_{1,\ell} A_1(\phi) \frac{1}{\ell^{1+1/2} \sin^{1/2} \phi} \nonumber \\
&\;\;-\frac{1}{2} h_0(0) (\cos \psi_{0,\ell+1} +5 \cos \psi_{0,\ell-1}) \frac{1}{\ell^{1+1/2} \sin^{1+1/2} \phi} \nonumber \\
&\;\;-3 h_0(0) \cos \phi \cos \psi_{0,\ell} \frac{1}{\ell^{1+1/2} \sin^{1+1/2} \phi} + \ell^{-2-1/2} O(\phi^{-2-1/2}), \nonumber \\
\frac{\beta^2_{3,\ell}(\phi)}{\ell^{3} (\ell+1)^{3}}&=h^2_0(0) \sin^2 \psi_{0,\ell} \frac{1}{\ell \sin \phi} - 6 h^2_0(0) \sin \psi_{0,\ell} \cos \psi_{0,\ell+1} \frac{1}{\ell^2 \sin^2 \phi} \label{b32} \\
&\;\;+ 2 h_0(0) h_0(1) \sin \psi_{0,\ell} \cos \psi_{0,\ell} \frac{1}{\ell^2 \phi \sin \phi}\nonumber \\
&\;\;-\frac{1}{4} h_0(0) h_1(0) \sin \psi_{0,\ell} \sin \psi_{1,\ell} \Big( \frac{1}{\ell^2 \sin^2 \phi} \cos \phi - \frac{1}{\ell^2 \phi \sin \phi} \Big) \nonumber\\
&\;\;+h^2_0(0) \sin \psi_{0,\ell} (\cos \psi_{0,\ell+1} +5 \cos \psi_{0,\ell-1}) \frac{1}{\ell^2 \sin^2 \phi} \nonumber \\
&\;\;+ 6 h^2_0(0) \sin \psi_{0,\ell} \cos \phi \cos \psi_{0,\ell} \frac{1}{\ell^2 \sin^2 \phi}+ \phi^{-1} O(\ell^{-2}), \nonumber \\
\frac{\beta^4_{3,\ell}(\phi)}{\ell^{6} (\ell+1)^{6}}&=h^4_0(0) \sin^4 \psi_{0,\ell} \frac{1}{\ell^2 \sin^2 \phi}+ \phi^{-3} O(\ell^{-3}) \label{b34},
\end{align}
\noindent where
$$
A_1(\phi)=\Big[ (\alpha+1)^2 -\frac 1 4 \Big] \Big( \frac{1-\phi \cot \phi}{2 \phi} \Big)- \frac{(\alpha+1)^2-\beta^2}{4} \tan \frac{\phi}{2}.
$$

And finally for $\gamma_{i,\ell}(\phi)$, $i=1,2,3,4$, we have
\begin{align}
\frac{\gamma_{1,\ell}(\phi)}{\ell^2 (\ell+1)^2}&=\phi^{-2-1/2} O(\ell^{-2-1/2}), \label{g1}\\
\frac{\gamma_{2,\ell}(\phi)}{\ell^2 (\ell+1)^2}&= - h_0(0) \sin \psi_{0,\ell} \frac{1}{\ell^{1+1/2} \sin^{1+1/2} \phi} + \phi^{-2-1/2} O(\ell^{-2-1/2}), \label{g2} \\
\frac{\gamma_{3,\ell}(\phi)}{\ell^2 (\ell+1)^2}&= - h_0(0) \sin \psi_{0,\ell} \cos \phi \frac{1}{\ell^{1+1/2} \sin^{1+1/2} \phi} + \phi^{-2-1/2} O(\ell^{-2-1/2}), \label{g3}\\
\frac{\gamma_{4,\ell}(\phi)}{\ell^2 (\ell+1)^2}&= h_0(0) \cos \psi_{0,\ell} \frac{1}{\ell^{1/2} \sin^{1/2} \phi} \sum_{j=0}^1 \binom{-1/2}{j} \frac{1}{2^j} \frac{1}{\ell^j} -2 h_0(0) \cos \psi_{0,\ell} \frac{1}{\ell^{1+1/2} \sin^{1/2} \phi} \label{g4}\\
&\;\;+ 4 \cos \psi_{1,\ell+1} \frac{1}{\ell^{1+1/2} \sin^{1+1/2} \phi} + h_0(0) \cos \psi_{0,\ell} \frac{1}{\ell^{1+1/2} \sin^{1/2} \phi} \nonumber \\
&\;\; -h_0(1) \sin \psi_{0,\ell} \frac{1}{\ell^{1+1/2} \phi \sin^{1/2} \phi} + h_1(0) \cos \psi_{1,\ell} A_1(\phi) \frac{1}{\ell^{1+1/2} \sin^{1/2} \phi} \nonumber \\
&\;\;- \frac 3 2 h_0(0) (\sin \psi_{0,\ell+1} +3 \sin \psi_{0,\ell-1}) \frac{1}{\ell^{1+1/2} \sin^{1+1/2} \phi}\nonumber \\
&\;\;-6 h_0(0) \cos \phi \sin \psi_{0,\ell} \frac{1}{\ell^{1+1/2} \sin^{1+1/2} \phi} + \phi^{-2-1/2} O(\ell^{-2-1/2}), \nonumber \\
 \frac{\gamma^2_{4,\ell}(\phi)}{\ell^4 (\ell+1)^4}&=h^2_0(0) \cos^2 \psi_{0,\ell} \frac{1}{\ell \sin \phi} +4 h_0(0) \cos \psi_{1,\ell+1} \cos \psi_{0,\ell} \frac{1}{\ell^2 \sin^2 \phi} \label{g42} \\
&\;\; - h_0(0) h_0(1)\sin \psi_{0,\ell} \cos \psi_{0,\ell} \frac{1}{\ell^2 \phi \sin \phi} - \frac 3 2 h^2_0(0) (\sin \psi_{0,\ell+1} +3 \sin \psi_{0,\ell-1}) \cos \psi_{0,\ell} \frac{1}{\ell^2 \sin^2 \phi} \nonumber \\
&\;\;-6 h^2_0(0) \cos \phi \sin \psi_{0,\ell} \cos \psi_{0,\ell} \frac{1}{\ell^2 \sin^2 \phi}+ \phi^{-1} O(\ell^{-2}),\nonumber\\
\frac{\gamma^3_{4,\ell}(\phi)}{\ell^6 (\ell+1)^6}&=h^3_0(0) \cos^3 \psi_{0,\ell} \frac{1}{\ell^{3/2} \sin^{3/2} \phi} +\phi^{-2-1/2} O(\ell^{-2-1/2}), \label{g43}\\
\frac{\gamma^4_{4,\ell}(\phi)}{\ell^8 (\ell+1)^8}&= h^4_0(0) \cos^4 \psi_{0,\ell} \frac{1}{\ell^{2} \sin^{2} \phi} +\phi^{-3} O(\ell^{-3}) \label{g44}.
\end{align}
\end{lemma}
\begin{proof}
The proof follows immediately from the tail decay of the derivatives of Legendre
polynomials derived in Appendix \ref{horror}. Recalling that $\alpha_{1,\ell}(\phi)=P'_{\ell}(\cos \phi)$ and in view of \eqref{P's} we obtain
\begin{align*}
\frac{\alpha_{1,\ell}(\phi)}{\ell (\ell+1)}&=h_0(0) \sin \psi_{0,\ell} \frac{1}{\ell^{1+1/2} \sin^{1+1/2} \phi}+ \phi^{-2-1/2} O(\ell^{-2-1/2})+ \phi^{-1} O(\ell^{-2}).%,\\
%\frac{\alpha^2_{1,\ell}(\phi)}{\ell^2 (\ell+1)^2}&=\phi^{-3} O(\ell^{-3}).
\end{align*}

\noindent Similarly, plugging \eqref{P's} and \eqref{P''s} into $\alpha_{2,\ell}(\phi)$, we have
\begin{align*}
\frac{\alpha_{2,\ell}(\phi)}{\ell (\ell+1)}&= h_0(0) \cos \psi_{0,\ell} \frac{1}{\ell^{1/2} \sin^{1/2} \phi} \sum_{j=0}^1 \binom{-1/2}{j} \frac{1}{2^j} \frac{1}{\ell^j}
- h_0(0) \cos \psi_{0,\ell} \frac{1}{\ell^{1+1/2} \sin^{1/2} \phi} \\
&\;\;+ 2 h_0(0) \sin \psi_{0,\ell+1} \frac{1}{\ell^{1+1/2} \sin^{1+1/2} \phi} + h_0(0) \cos \psi_{0,\ell} \frac{1}{\ell^{1+1/2} \sin^{1/2}} - h_0(1) \sin \psi_{0,\ell} \frac{1}{\ell^{1+1/2} \phi \sin^{1/2} \phi}\\
&\;\;+ h_1(0) \cos \psi_{1,\ell} A_1(\phi) \frac{1}{\ell^{1+1/2} \sin^{1/2} \phi} - h_0(0) \sin \psi_{0,\ell-1} \frac{1}{\ell^{1+1/2} \sin^{1+1/2} \phi}\\
&\;\;+ h_0(0) \cos \phi \sin \psi_{0,\ell} \frac{1}{\ell^{1+1/2} \sin^{1+1/2} \phi}+ \phi^{-2-1/2} O(\ell^{-2-1/2}) +O(\ell^{-2}),%\\
%\frac{\alpha^2_{2,\ell}(\phi)}{\ell^2 (\ell+1)^2}&=\cos^2 \psi_{0,\ell} h^2_0(0) \frac{1}{\ell \sin \phi} + 2 \sin \psi_{0,\ell+1} h^2_0(0) \cos \psi_{0,\ell} \frac{1}{\ell^2 \sin^2 \phi}\\&\;\;- \sin \psi_{0,\ell} \cos \psi_{0,\ell} h_0(0) h_0(1) \frac{1}{\ell^2 \phi \sin \phi} - \sin \psi_{0,\ell-1} \cos \psi_{0,\ell} h^2_0(0) \frac{1}{\ell^2 \sin^2 \phi}\\
%&\;\;+ \cos \phi \cos \psi_{0,\ell} \sin \psi_{0,\ell} h^2_0(0) \frac{1}{\ell^2 \sin^2 \phi}+\phi^{-1} O(\ell^{-2}),\\
%\frac{\alpha^4_{2,\ell}(\phi)}{\ell^4 (\ell+1)^4}&=\cos^4 \psi_{0,\ell} h^4_0(0) \frac{1}{\ell^2 \sin^2 \phi} + \phi^{-3} O(\ell^{-3}),
\end{align*}
\begin{align*}
\frac{\beta_{1,\ell}(\phi)}{\ell^{3/2} (\ell+1)^{3/2}}&=- h_0(0) \cos \psi_{0,\ell} \frac{1}{\ell^{1+1/2} \sin^{1+1/2} \phi}+\phi^{-2-1/2} O(\ell^{-2-1/2}), %\\
%\frac{\beta^2_{1,\ell}(\phi)}{\ell^{3} (\ell+1)^{3}}&=\phi^{-3} O(\ell^{-3}),
\end{align*}
and the asymptotic behaviour of $\frac{\beta_{2,\ell}(\phi)}{\ell^{3/2} (\ell+1)^{3/2}}$ and $\frac{\gamma_{1,\ell}(\phi)}{\ell^2 (\ell+1)^2}$ as in the statement.
%\begin{align*}
%\frac{\beta_{2,\ell}(\phi)}{\ell^{3/2} (\ell+1)^{3/2}}&=-\cos \psi_{0,\ell} \cos \phi h_0(0) \frac{1}{\ell^{1+1/2} \sin^{1+1/2} \phi}+ \phi^{-2-1/2} O(\ell^{-2-1/2}), %\\
%%\frac{\beta^2_{2,\ell}(\phi)}{\ell^{3} (\ell+1)^{3}}&=\phi^{-3} O(\ell^{-3}),
%\end{align*}
%\begin{align*}
%\frac{\gamma_{1,\ell}(\phi)}{\ell^2 (\ell+1)^2}&=\phi^{-2-1/2} O(\ell^{-2-1/2}).
%\end{align*}
\noindent From \eqref{P's}, \eqref{P''s} and \eqref{P'''s} we obtain the decay rate of $\frac{\beta_{3,\ell}(\phi)}{\ell^{3/2} (\ell+1)^{3/2}}$, $\frac{\gamma_{2,\ell}(\phi)}{\ell^2 (\ell+1)^2}$
and $\frac{\gamma_{3,\ell}(\phi)}{\ell^2 (\ell+1)^2}$.
\noindent Finally, in view of \eqref{P's}, \eqref{P''s}, \eqref{P'''s} and \eqref{P''''s}, we obtain the asymptotic behaviour of $\frac{\gamma_{4,\ell}(\phi)}{\ell^2 (\ell+1)^2}$.
\end{proof}

\noindent We exploit now Lemma \ref{perturbing} to obtain the bounds for the terms $A_{0,\ell}$, $A_{i_1,\ell}$ and $A_{i_1, \dots i_k,\ell}$ for $k=2,3,4$, $i_1, \dots i_k=1,\dots,8$. %in the statement of Lemma \ref{L-zero}.

 \begin{proof}[Proof of Lemma \ref{L-zero}] In view of \eqref{a1} and \eqref{a2} we first obtain
 \begin{align*}
A_{0,\ell }& =\int_{C/\ell }^{\pi -C/\ell} \frac{\sin \phi }{\sqrt{%
(1-4\alpha _{2,\ell }^{2}(\phi )/\lambda _{\ell }^{2})(1-4\alpha _{1,\ell
}^{2}(\phi )/\lambda _{\ell }^{2})}}d\phi \\
&=\int_{C/\ell }^{\pi -C/\ell} \left( 1+ 2 \frac{\alpha _{2,\ell }^{2}(\phi )}{\lambda _{\ell }^{2}} + 2 \cdot 3 \frac{\alpha_{2,\ell }^{4}(\phi )}{\lambda _{\ell }^{4}} \right) \sin \phi d\phi +O(\ell^{-2})\\
&= \cos \left( C/\ell \right) +2 \int_{C/\ell }^{\pi -C/\ell} \frac{\alpha _{2,\ell }^{2}(\phi )}{\lambda _{\ell }^{2}} \sin \phi d\phi +2 \cdot 3 \int_{C/\ell }^{\pi -C/\ell} \frac{\alpha _{2,\ell }^{4}(\phi )}{\lambda _{\ell }^{4}} \sin \phi d\phi +O(\ell^{-2}).
\end{align*}
Now, from \eqref{a2}, we have
\begin{align*}
2 \int_{C/\ell }^{\pi -C/\ell} \frac{\alpha _{2,\ell }^{2}(\phi )}{\lambda _{\ell }^{2}} \sin \phi d\phi &=2 \int_{C/\ell }^{\pi -C/\ell} \left[ \frac{h^2_0(0) \cos^2 \psi_{0,\ell}}{\ell \sin \phi } + 2 \binom{-1/2}{1} h^2_0(0) \frac{\sin \psi_{0,\ell+1} \cos \psi_{0,\ell}}{\ell^2 \sin^2 \phi} - h_0(0) h_0(1) \frac{\cos \psi_{0,\ell} \sin \psi_{0,\ell}}{\ell^2 \phi \sin \phi} \right.\\
&\;\; -\left. h^2_0(0) \frac{\cos \psi_{0,\ell} \sin \psi_{0,\ell-1}}{\ell^2 \sin^2 \phi} + h^2_0(0) \frac{\cos \psi_{0,\ell} \cos \phi \sin \psi_{0,\ell}}{\ell^2 \sin^2 \phi}
 \right] \sin \phi d\phi +O(\ell^{-2})\\
&=2 \int_{C/\ell }^{\pi -C/\ell} \left[ \frac{h^2_0(0) \cos^2 \psi_{0,\ell}}{\ell \sin \phi } \right] \sin \phi d\phi +O(\ell^{-2}) \\
&= \frac{4}{\pi} \int_{C/\ell}^{\pi -C/\ell} \frac{1}{\ell \sin \phi} \cos^2[(\ell+1/2) \phi-\pi/4] \sin \phi d \phi + O(\ell^{-2})\\
&= \frac{2}{\pi} \frac{1}{\ell} \int_{C/\ell}^{\pi -C/\ell} d \phi + \frac{2}{\pi} \frac{1}{\ell} \int_{C/\ell}^{\pi -C/\ell} \cos[2 (\ell+1/2) \phi-2 \pi/4 ] d \phi +O(\ell^{-2})\\
&= \frac{2}{\ell}+O(\ell^{-2}),
\end{align*}
and from \eqref{a3}, and the equality $\cos^4 \psi_{0,\ell}=\frac 3 8+ \frac 1 8 [-\cos(2 \phi (2 \ell+1) ) +4 \sin(\phi (2 \ell+1))]$, we have
\begin{align*}
2 \cdot 3 \int_{C/\ell }^{\pi -C/\ell} \frac{\alpha _{2,\ell }^{4}(\phi )}{\lambda _{\ell }^{4}} \sin \phi d\phi
&=2 \cdot 3 \frac{1}{\ell^2} \int_{C/\ell }^{\pi -C/\ell} h_0^4(0) \frac{1}{\sin \phi} \cos^4 \psi_{0,\ell} d \phi+O(\ell^{-2})\\
&=2 \cdot 3 \frac 3 8 \frac{2^2}{\pi^2} \frac{1}{\ell^2} \int_{C/\ell }^{\pi -C/\ell} \frac{1}{\sin \phi} d \phi+O(\ell^{-2})= \frac{2 \cdot 3^2}{\pi^2} \frac{\log \ell}{\ell^2}+O(\ell^{-2}).
\end{align*}
Therefore the statement follows since we obtain
\begin{align*}
{ \lambda _{\ell }^{2} A_{0,\ell } \; q(\mathbf{0}) -(\mathbb{E}\left[ {\cal N}^c(f_{\ell}) \right])^{2}= \lambda _{\ell }^{2} \left[ \frac{2}{\ell}+\frac{2 \cdot 3^2}{\pi^2} \frac{\log \ell}{\ell^2}+O(\ell^{-2}) \right] q(\mathbf{0}) =
\left[ 2 \ell^3 +\frac{2 \cdot 3^2}{\pi^2} \ell^2 \log \ell \right] q(\mathbf{0})+O(\ell^2).}
\end{align*}
\end{proof}

\begin{proof}[Proof of Lemma \ref{L-uno}]
 In view of \eqref{a1}, \eqref{a2}, \eqref{b12} and \eqref{b22}, we immediately obtain that $A_{1,\ell }, A_{2,\ell }=O(\ell^{-2})$:
\begin{align*}
A_{1,\ell }&=-\frac{16}{\lambda_{\ell}^3}\int_{C/\ell}^{\pi -C/\ell} \frac{\beta^2_{2,\ell}(\phi)}{(1-4 \alpha_2^2/\lambda_{\ell}^2)^{3/2} (1-4 \alpha_1^2/\lambda_{\ell}^2)^{1/2}} \sin \phi d \phi=-16 \int_{C/\ell}^{\pi -C/\ell} \frac{\beta^2_{2,\ell}(\phi)}{\ell^6} \sin \phi d \phi+O(\ell^{-2})=O(\ell^{-2}),
\end{align*}
\begin{align*}
A_{2,\ell }&=-\frac{16}{\lambda_{\ell}^3}\int_{C/\ell}^{\pi -C/\ell } \frac{\beta^2_{1,\ell}(\phi)}{(1-4 \alpha_2^2/\lambda_{\ell}^2)^{1/2} (1-4 \alpha_1^2/\lambda_{\ell}^2)^{3/2}} \sin \phi d \phi=-16 \int_{C/\ell}^{\pi -C/\ell} \frac{\beta^2_{1,\ell}(\phi)}{\ell^6} \sin \phi d \phi+O(\ell^{-2})=O(\ell^{-2}).
\end{align*}
Form \eqref{a1}, we have
\begin{align*}
A_{3,\ell}&=-\frac{16}{\lambda_{\ell}^3}\int_{C/\ell}^{\pi -C/\ell} \frac{\beta^2_{3,\ell}(\phi)}{(1-4 \alpha_2^2/\lambda_{\ell}^2)^{3/2} (1-4 \alpha_1^2/\lambda_{\ell}^2)^{1/2}} \sin \phi d \phi =-\frac{16}{\lambda_{\ell}^3}\int_{C/\ell}^{\pi -C/\ell} \beta^2_{3,\ell}(\phi) \left[ 1+ \frac 3 2 \frac{4 \alpha_2^2}{\lambda_{\ell}^2} \right] \sin \phi d \phi + O(\ell^{-2}),
\end{align*}
where, in view of \eqref{b32}, we obtain a leading non-oscillatory term in
\begin{align*}
-\frac{16}{\lambda_{\ell}^3} \int_{C/\ell}^{\pi -C/\ell} \beta^2_{3,\ell}(\phi) \sin \phi d \phi &=-16 \int_{C/\ell}^{\pi -C/\ell} \frac{\beta^2_{3,\ell}(\phi) }{\ell^6} \sin \phi d \phi+O(\ell^{-2})\\
&=-16 \int_{C/\ell}^{\pi -C/\ell} \left[ h_0^2(0) \frac{\sin^2 \psi_{0,\ell}}{\ell \sin \phi} \right] \sin \phi d \phi+O(\ell^{-2})\\
%&= -16 \; h_0^2(0) \; \frac{1}{\ell} \int_{C/\ell}^{\pi -C/\ell} \sin^2 \psi_{0,\ell} d \phi+O(\ell^{-2})\\
&= -16 \; h_0^2(0) \; \frac{1}{\ell} \int_{C/\ell}^{\pi -C/\ell} \left[ \frac 1 2 - \frac 1 2 \cos(2 \psi_{0,\ell}) \right] d \phi+O(\ell^{-2})\\
&= -16 \; h_0^2(0) \; \frac{1}{\ell} \; \frac 1 2 \; \pi +O(\ell^{-2})= - \frac{16}{\ell}+O(\ell^{-2})
\end{align*}
and, from \eqref{a2}, \eqref{b32} and the equality $\sin^2 \psi_{0,\ell} \cos^2 \psi_{0,\ell}=\frac 1 4 \cos^2[(2 \ell+1) \phi]=\frac 1 8 + \frac 1 8 \cos [2 \phi (2 \ell+1)]$, we have
\begin{align*}
-\frac{16}{\lambda_{\ell}^3} \frac {3\cdot 4} {2} \int_{C/\ell}^{\pi -C/\ell} \beta^2_{3,\ell}(\phi) \frac{ \alpha_2^2}{\lambda_{\ell}^2} \sin \phi d \phi&=
-16 \frac {3\cdot 4} {2} \int_{C/\ell}^{\pi -C/\ell} \frac{\beta^2_{3,\ell}(\phi)}{\ell^6} \frac{ \alpha_2^2}{\ell^4} \sin \phi d \phi +O(\ell^{-2})\\
&= -16 \frac {3\cdot 4} {2} h_0^4(0) \frac{1}{\ell^2} \int_{C/\ell}^{\pi -C/\ell} \left[ \frac{\sin^2 \psi_{0,\ell} \cos^2 \psi_{0,\ell}}{\sin^2 \phi} \right] \sin \phi d \phi +O(\ell^{-2})\\
&=-16 \frac {3\cdot 4} {2} h_0^4(0) \frac 1 8 \frac{1}{\ell^2} \int_{C/\ell}^{\pi -C/\ell} \frac{1}{\sin \phi} d \phi +O(\ell^{-2})\\
&= -\frac{3 \cdot 4^2}{\pi^2} \frac{2}{\ell^2} \left[ \log \left( \cos \left( \frac{C}{2 \ell} \right)\right)- \log \left( \sin \left( \frac{C}{2 \ell} \right)\right) \right]+O(\ell^{-2})\\
&= -\frac{3 \cdot 4^2}{\pi^2} \frac{2}{\ell^2} \log \ell+O(\ell^{-2}).
\end{align*}
Then
\begin{align*}
{ \lambda _{\ell }^{2} A_{3,\ell} = - 16 \ell^3 - \frac{2^5 \cdot 3 }{\pi^2} \ell^2 \log \ell +O(\ell^{2}).}
\end{align*}
The terms $A_{4,\ell}, A_{5,\ell}, A_{6,\ell}, A_{7,\ell}, A_{8,\ell}$ are $O(\ell^{-2})$. In fact, for $A_{4,\ell}$, using \eqref{a1}, \eqref{a2}, \eqref{b2}, \eqref{b3} and the trigonometric equality $-\cos \psi_{0,\ell} \sin \psi_{0,\ell}=1/2 \cos[(2 \ell+1) \phi]$, we have
\begin{align*}
A_{4,\ell}&=-\frac{16}{\lambda_{\ell}^3} \int_{C/\ell}^{\pi -C/\ell} \frac{ \beta_{2,\ell}(\phi) \beta_{3,\ell}(\phi)}{(1-4 \alpha_2^2/\lambda_{\ell}^2)^{3/2} (1-4 \alpha_1^2/\lambda_{\ell}^2)^{1/2}} \sin \phi d \phi\\
&= -\frac{16}{\ell^6} \int_{C/\ell}^{\pi -C/\ell} \beta_{2,\ell}(\phi) \beta_{3,\ell}(\phi) \sin \phi d \phi+O(\ell^{-2})\\
&= -16 \int_{C/\ell}^{\pi -C/\ell} \left[ - h_0(0) \frac{\cos \psi_{0,\ell} \cos \phi }{\ell^{1+1/2} \sin^{1+1/2} \phi} \right] \left[- h_0(0) \frac{\sin \psi_{0,\ell}}{\ell^{1/2} \sin^{1/2} \phi} \right] \sin \phi d \phi+O(\ell^{-2})=O(\ell^{-2}),
\end{align*}
for $A_{5,\ell}$, form \eqref{a1}, \eqref{a2}, \eqref{b22} and \eqref{g1}, we immediately have
\begin{align*}
&A_{5,\ell}\\
&=\frac{8}{\lambda_{\ell}^2} \int_{C/\ell}^{\pi -C/\ell} \frac{\gamma_{1,\ell}(\phi)}{ (1-4 \alpha_2^2/\lambda_{\ell}^2)^{1/2} (1-4 \alpha_1^2/\lambda_{\ell}^2)^{1/2} } \sin \phi d \phi - \frac{8 \cdot 4}{\lambda_{\ell}^4} \int_{C/\ell}^{\pi -C/\ell} \frac{\alpha_{2,\ell}(\phi) \beta_{2,\ell}^2(\phi)}{ (1-4 \alpha_2^2/\lambda_{\ell}^2)^{3/2} (1-4 \alpha_1^2/\lambda_{\ell}^2)^{1/2} } \sin \phi d \phi\\
&= 8 \int_{C/\ell}^{\pi -C/\ell} \frac{\gamma_{1,\ell}(\phi)}{\ell^4} \sin \phi d \phi
- 8 \cdot 4 \int_{C/\ell}^{\pi -C/\ell} \frac{\alpha_{2,\ell}(\phi) \beta_{2,\ell}^2(\phi)}{\ell^8} \sin \phi d \phi=O(\ell^{-2}),
\end{align*}
\noindent the asymptotic behaviour of $A_{6,\ell}$ follows from \eqref{a1}, \eqref{a2}, \eqref{b12} and \eqref{g2}
\begin{align*}
&A_{6,\ell}\\
&=\frac{8}{\lambda_{\ell}^2} \int_{C/\ell}^{\pi -C/\ell} \frac{\gamma_{2,\ell}(\phi)}{ (1-4 \alpha_2^2/\lambda_{\ell}^2)^{1/2} (1-4 \alpha_1^2/\lambda_{\ell}^2)^{1/2} } \sin \phi d \phi - \frac{8 \cdot 4}{\lambda_{\ell}^4} \int_{C/\ell}^{\pi -C/\ell} \frac{\alpha_{1,\ell}(\phi) \beta_{1,\ell}^2(\phi)}{ (1-4 \alpha_2^2/\lambda_{\ell}^2)^{1/2} (1-4 \alpha_1^2/\lambda_{\ell}^2)^{3/2} } \sin \phi d \phi\\
&= 8 \int_{C/\ell}^{\pi -C/\ell} \frac{\gamma_{2,\ell}(\phi)}{\ell^4} \sin \phi d \phi +O(\ell^{-2})= - 8 \int_{C/\ell}^{\pi -C/\ell} \frac{1}{\ell^{1+1/2} \sin^{1+1/2} \phi} \sin \psi_{0,\ell} h_0(0) \sin \phi d \phi+O(\ell^{-2})=O(\ell^{-2}),
\end{align*}
\noindent the asymptotic behaviour of $A_{7,\ell}$ follows from \eqref{a1}, \eqref{a2}, \eqref{b32} and \eqref{g4}, in fact we have
\begin{align*}
&A_{7,\ell}\\
&=\frac{8}{\lambda_{\ell}^2} \int_{C/\ell}^{\pi -C/\ell} \frac{\gamma_{4,\ell}(\phi)}{ (1-4 \alpha_2^2/\lambda_{\ell}^2)^{1/2} (1-4 \alpha_1^2/\lambda_{\ell}^2)^{1/2} } \sin \phi d \phi - \frac{8 \cdot 4}{\lambda_{\ell}^4} \int_{C/\ell}^{\pi -C/\ell} \frac{\alpha_{2,\ell}(\phi) \beta_{3,\ell}^2(\phi)}{ (1-4 \alpha_2^2/\lambda_{\ell}^2)^{3/2} (1-4 \alpha_1^2/\lambda_{\ell}^2)^{1/2} } \sin \phi d \phi\\
&= \frac{8}{\lambda_{\ell}^2} \int_{C/\ell}^{\pi -C/\ell} \gamma_{4,\ell}(\phi) \sin \phi d \phi+
 \frac{8}{\lambda_{\ell}^2} \frac 4 2 \frac{1}{\lambda_{\ell}^2} \int_{C/\ell}^{\pi -C/\ell} \gamma_{4,\ell}(\phi) \alpha^2_{2,\ell}(\phi) \sin \phi d \phi- \frac{8 \cdot 4}{\lambda_{\ell}^4} \int_{C/\ell}^{\pi/2} \alpha_{2,\ell}(\phi) \beta_{3,\ell}^2(\phi) \sin \phi d \phi+O(\ell^{-2})\\
 &= \frac{8}{{\ell}^4} \left( 1-\frac 2 \ell \right) \int_{C/\ell}^{\pi -C/\ell} \gamma_{4,\ell}(\phi) \sin \phi d \phi+
 \frac{16}{{\ell}^8} \int_{C/\ell}^{\pi -C/\ell} \gamma_{4,\ell}(\phi) \alpha^2_{2,\ell}(\phi) \sin \phi d \phi \\
 &\;\;- \frac{8 \cdot 4}{{\ell}^8} \int_{C/\ell}^{\pi -C/\ell} \alpha_{2,\ell}(\phi) \beta_{3,\ell}^2(\phi) \sin \phi d \phi+O(\ell^{-2})= \frac{8}{{\ell}^4} \int_{C/\ell}^{\pi -C/\ell} \gamma_{4,\ell}(\phi) \sin \phi d \phi+O(\ell^{-2}),
\end{align*}
where, using integration by parts,
\begin{align*}
&\frac{8}{{\ell}^4} \int_{C/\ell}^{\pi -C/\ell} \gamma_{4,\ell}(\phi) \sin \phi d \phi\\
&=8 h_0(0) \frac{1}{\ell^{1/2}} \int_{C/\ell}^{\pi -C/\ell} \frac{\cos \psi_{0,\ell}}{\sin^{1/2} \phi} \sin \phi \; d \phi\\
&= 8 h_0(0) \frac{1}{\ell^{1/2}} \int_{C/\ell}^{\pi -C/\ell} \cos [(\ell+1/2) \phi-\pi/4] \sin^{1/2} \phi \; d \phi\\
&= 8 h_0(0) \frac{1}{\ell^{1/2}(\ell+1/2)} \left\{ \left. \sin[(\ell+1/2) \phi-\pi/4] \sin^{1/2} \phi \right|_{C/\ell}^{\pi -C/\ell}
- \frac 1 2 \int_{C/\ell}^{\pi -C/\ell} \frac{\sin[(\ell+1/2) \phi-\pi/4]}{\sin^{1/2} \phi} d \phi \right\} \\
&= 8 \sqrt{\frac 2 \pi} \frac{1}{\ell^{3/2}} \left[ \sin[(\ell+1/2) \pi/2-\pi/4] \sin^{1/2} (\pi-C/\ell) -\sin[(\ell+1/2) C/\ell -\pi/4] \sin^{1/2} (C/\ell) \right] +O(\ell^{-2}),
\end{align*}
and finally $A_{8,\ell}$ follows from \eqref{a1}, \eqref{a2}, \eqref{b2}, \eqref{b3} and \eqref{g3}:
\begin{align*}
&A_{8,\ell}\\
&=\frac{8}{\lambda_{\ell}^2} \int_{C/\ell}^{\pi -C/\ell} \frac{\gamma_{3,\ell}(\phi)}{ (1-4 \alpha_2^2/\lambda_{\ell}^2)^{1/2} (1-4 \alpha_1^2/\lambda_{\ell}^2)^{1/2} } \sin \phi d \phi - \frac{8 \cdot 4}{\lambda_{\ell}^4} \int_{C/\ell}^{\pi -C/\ell} \frac{\alpha_{2,\ell}(\phi) \beta_{2,\ell}(\phi) \beta_{3,\ell}(\phi)}{ (1-4 \alpha_2^2/\lambda_{\ell}^2)^{3/2} (1-4 \alpha_1^2/\lambda_{\ell}^2)^{1/2} } \sin \phi d \phi\\
&= \frac{8}{\ell^4} \int_{C/\ell}^{\pi -C/\ell} \gamma_{3,\ell}(\phi)\sin \phi d \phi+O(\ell^{-2})= -8 h_0(0) \frac{1}{\ell^{1+1/2}} \int_{C/\ell}^{\pi -C/\ell} \frac{\sin \psi_{0,\ell} \cos \phi }{\sin^{1/2} \phi} d \phi +O(\ell^{-2})=O(\ell^{-2}).
\end{align*}
\vspace{0.3cm}

We study now the asymptotic behaviour of the higher order terms of the form $A_{i_1, \dots i_k,\ell}$ with $k=2,3,4$.
Note that each term of the form
$$A_{i_1, \dots i_k,\ell}\;\; \text{with}\;\; (i_1, \dots i_k) \ne (3,3), (3,7), (7,7), (3,7,7), (7,7,7), (7,7,7,7),$$
is of order $O(\ell^{-2})$. This implies that, to prove the statement, it is enough to analyse the high energy asymptotic behaviour of the following terms. \\

We first note that $A_{77,\ell}$ produces a leading non-oscillating term, in fact
\begin{align*}
A_{77,\ell}&= \int_{C/\ell}^{\pi -C/\ell} \frac{a^2_{7,\ell}(\phi)}{ (1-4 \alpha_2^2/\lambda_{\ell}^2)^{1/2} (1-4 \alpha_1^2/\lambda_{\ell}^2)^{1/2} } \sin \phi d \phi
\end{align*}
where
\begin{align*}
a^2_{7,\ell}(\phi)&= \frac{8^2}{\lambda_{\ell}^4} \left[ \gamma_4- \frac{4}{\lambda_{\ell}^2} \frac{\alpha_2 \; \beta_3^2}{ (1-4 \alpha_2^2/\lambda_{\ell}^2)} \right]^2 = \frac{8^2}{\lambda_{\ell}^4} \left[ \gamma_4^2 +\frac{16}{\lambda_{\ell}^4} \frac{\alpha^2_2 \; \beta_3^4}{ (1-4 \alpha_2^2/\lambda_{\ell}^2)^2} - 2 \gamma_4 \frac{4}{\lambda_{\ell}^2} \frac{\alpha_2 \; \beta_3^2}{ (1-4 \alpha_2^2/\lambda_{\ell}^2)} \right],
\end{align*}
so that
\begin{align} \label{roba}
A_{77,\ell}&= 8^2 \int_{C/\ell}^{\pi -C/\ell} \frac{\gamma_4^2}{\ell^8} \left( 1+2 \frac{\alpha_2^2}{\ell^4} \right) \sin \phi d \phi
- \frac{8^3}{\ell^{12}} \int_{C/\ell}^{\pi -C/\ell} \gamma_4 \alpha_2 \beta_3^2 \sin \phi d \phi +O(\ell^{-2}).
\end{align}
Now, in view of \eqref{g42} and \eqref{a2}, we obtain
\begin{align*}
8^2 \int_{C/\ell}^{\pi -C/\ell} \frac{\gamma_4^2}{\ell^8} \sin \phi d \phi = 8^2 h^2_0(0) \frac{1}{\ell} \int_{C/\ell}^{\pi -C/\ell} \cos^2 \psi_{0,\ell} d \phi+O(\ell^{-2})=\frac{2 \cdot 32}{\ell} +O(\ell^{-2}),
\end{align*}
\begin{align*}
2 \cdot 8^2 \int_{C/\ell}^{\pi -C/\ell} \frac{\gamma_4^2}{\ell^8} \frac{\alpha_2^2}{\ell^4} \sin \phi d \phi
&=2 \cdot 8^2 \int_{C/\ell}^{\pi -C/\ell} \left(h^2_0(0) \frac{\cos^2 \psi_{0,\ell}}{\ell \sin \phi} \right) \left( h^2_0(0) \frac{\cos^2 \psi_{0,\ell}}{\ell \sin \phi} \right) \sin \phi d \phi +O(\ell^{-2})\\
&= 2 \cdot 8^2 h^4_0(0) \frac{1}{\ell^2} \int_{C/\ell}^{\pi -C/\ell} \Big[\frac 3 8 + \frac 1 8 \cos(4 \psi_{0,\ell})+ \frac 1 2 \cos(2 \psi_{0,\ell}) \Big] \frac{1}{\sin \phi}d \phi +O(\ell^{-2})\\
&= 2 \cdot 8^2 h^4_0(0) \frac 3 8 \frac{1}{\ell^2} \int_{C/\ell}^{\pi -C/\ell} \frac{1}{\sin \phi} d \phi +O(\ell^{-2})\\
%&= 2 \cdot 2 \cdot 8^2 h^4_0(0) \frac 3 8 \frac{1}{\ell^2} \log \ell +O(\ell^{-2})\\
&= 2 \cdot 2 \cdot 8^2 \frac{2^2}{\pi^2} \frac 3 8 \frac{1}{\ell^2} \log \ell +O(\ell^{-2})= \frac{2 \cdot 3 \cdot 2^6}{\pi^2} \frac{1}{\ell^2} \log \ell +O(\ell^{-2}),
\end{align*}
for the last term in \eqref{roba} we use \eqref{b32}, \eqref{g4} and
%$$\cos^2 \psi^-_{\ell} \cos^2 \psi^+_{\ell}=\frac 1 8+ \frac 1 4 \cos (2 \psi^+_{\ell})+ \frac 1 4 \cos (2 \psi^-_{\ell})+ \frac 1 8 \cos(2(2 \ell+1) \phi)$$
$$\cos^2 \psi_{0,\ell} \sin^2 \psi_{0,\ell}=\frac 1 8+ \frac 1 4 \cos [2 (\ell+1/2) \phi+ \pi/2 ]+ \frac 1 4 \cos [2 (\ell+1/2) \phi- \pi/2]+ \frac 1 8 \cos[2(2 \ell+1) \phi]$$
to obtain
\begin{align*}
&- \frac{8^3}{\ell^{12}} \int_{C/\ell}^{\pi -C/\ell} \gamma_4 \alpha_2 \beta_3^2 \sin \phi d \phi\\
& =-8^3 \int_{C/\ell}^{\pi -C/\ell} \left(h_0(0) \frac{\cos \psi_{0,\ell}}{\ell^{1/2} \sin^{1/2} \phi} \right) \left( h_0(0) \frac{\cos \psi_{0,\ell}}{\ell^{1/2} \sin^{1/2} \phi} \right) \left(h^2_0(0) \frac{\sin^2 \psi_{0,\ell}}{\ell \sin \phi} \right) \sin \phi d \phi+O(\ell^{-2})\\
& =-8^3 h^4_0(0) \frac{1}{\ell^2} \int_{C/\ell}^{\pi -C/\ell} \frac{\cos^2 \psi_{0,\ell} \, \sin^2 \psi_{0,\ell} }{ \sin \phi} d \phi+O(\ell^{-2})\\
& =-8^2 h^4_0(0) \frac{1}{\ell^2} \int_{C/\ell}^{\pi -C/\ell} \frac{1}{ \sin \phi} d \phi+O(\ell^{-2})
%&= -2 \cdot 8^2 \frac{2^2}{\pi^2} \frac{1}{\ell^2} \log \ell +O(\ell^{-2})\\
= - \frac{2 \cdot 2^8}{\pi^2} \frac{1}{\ell^2} \log \ell +O(\ell^{-2}).
\end{align*}
Therefore
\begin{align*}
{ \lambda _{\ell }^{2} \frac 1 2 A_{77,\ell} =\lambda _{\ell }^{2} \left[ \frac{32}{\ell} + \frac{3 \cdot 2^6}{\pi^2} \frac{1}{\ell^2} \log \ell - \frac{2^8}{\pi^2} \frac{1}{\ell^2} \log \ell+O(\ell^{-2}) \right] = 32 \ell^3 - \frac{ 2^6}{\pi^2} \ell^2 \log \ell +O(\ell^{2}).}
\end{align*}
We apply now \eqref{a1}, \eqref{a2} and \eqref{b34} to study the asymptotic behaviour of $A_{33,\ell}$:
\begin{align*}
A_{33,\ell}&= \int_{C/\ell}^{\pi -C/\ell} \frac{a^2_{3,\ell}(\phi)}{ (1-4 \alpha_2^2/\lambda_{\ell}^2)^{1/2} (1-4 \alpha_1^2/\lambda_{\ell}^2)^{1/2} } \sin \phi d \phi
\end{align*}
where
\begin{align*}
a^2_{3,\ell}(\phi)= 16^2 \frac{\beta_3^4}{\lambda_{\ell}^6 (1-4 \alpha_2^2/\lambda_{\ell}^2)^2},
\end{align*}
that is
\begin{align*}
A_{33,\ell}&= 16^2 \int_{C/\ell}^{\pi -C/\ell} \frac{\beta_3^4}{ \lambda_{\ell}^6 (1-4 \alpha_2^2/\lambda_{\ell}^2)^{5/2} (1-4 \alpha_1^2/\lambda_{\ell}^2)^{1/2} } \sin \phi d \phi =16^2 \int_{C/\ell}^{\pi -C/\ell} \frac{\beta_3^4}{ \lambda_{\ell}^6} \sin \phi d \phi+O(\ell^{-2})\\
&= 16^2 \frac{1}{\ell^2} \int_{C/\ell}^{\pi -C/\ell} h^4_0(0) \frac{\sin^4 \psi_{0,\ell} }{\sin \phi} d \phi+O(\ell^{-2})
= 16^2 \frac{2}{\ell^2} \frac 3 8 \frac{2^2}{\pi^2} \log \ell + O(\ell^{-2})= \frac{2 \cdot 3 \cdot 2^7}{\pi^2} \frac{1}{\ell^2} \log \ell+O(\ell^{-2})
\end{align*}
since $\sin^4 \psi_{0,\ell}=\frac 3 8 - \frac 1 8 \cos[2 \phi (2 \ell+1)] - \frac 1 2 \sin [\phi (2 \ell+1)]$. Therefore
\begin{align*}
{ \lambda _{\ell }^{2} \frac 1 2 A_{33,\ell} =\lambda _{\ell }^{2} \left[ \frac{3 \cdot 2^7}{\pi^2} \frac{1}{\ell^2} \log \ell+O(\ell^{-2}) \right]
= \frac{3 \cdot 2^7}{\pi^2} \ell^2 \log \ell +O(\ell^{2}).}
\end{align*}
The terms $A_{37,\ell}$ and $A_{777,\ell}$ are both $O(\ell^{-2})$ since their leading non-constant term are oscillating. One has \begin{align*}
A_{37,\ell}&= \int_{C/\ell}^{\pi -C/\ell} \frac{a_{3,\ell}(\phi) a_{7,\ell}(\phi)}{ (1-4 \alpha_2^2/\lambda_{\ell}^2)^{1/2} (1-4 \alpha_1^2/\lambda_{\ell}^2)^{1/2} } \sin \phi d \phi
\end{align*}
where, by \eqref{a1} and \eqref{a2},
\begin{align*}
 \frac{a_{3,\ell}(\phi) a_{7,\ell}(\phi)}{ (1-4 \alpha_2^2/\lambda_{\ell}^2)^{1/2} (1-4 \alpha_1^2/\lambda_{\ell}^2)^{1/2} } =-16 \frac{ \frac{\beta_3^2}{\lambda_{\ell}^3 (1-4 \alpha_2^2/\lambda_{\ell}^2)} 8\frac{\gamma _{4}+\frac{4\alpha _{2
} \beta _{3}^{2}}{4\alpha _{2}^{2}-\lambda _{\ell }^{2}}}{\lambda _{\ell }^2}}{(1-4 \alpha_2^2/\lambda_{\ell}^2)^{1/2} (1-4 \alpha_1^2/\lambda_{\ell}^2)^{1/2}}= -16 \cdot 8 \; \frac{\beta_3^2}{\ell^6} \frac{\gamma_4}{\ell^4}+O(\frac{1}{\ell^3 \phi^3}).
\end{align*}
Then, in view of \eqref{b32} and \eqref{g4}, we have
\begin{align*}
A_{37,\ell}&= -16 \cdot 8 \; \int_{C/\ell}^{\pi -C/\ell} \frac{\beta_3^2}{\ell^6} \frac{\gamma_4}{\ell^4} \sin \phi d \phi +O(\ell^{-2})\\
&=-16 \cdot 8 \; \int_{C/\ell}^{\pi -C/\ell} \left[ h^2_0(0) \frac{\sin^2 \psi_{0,\ell}}{\ell \sin \phi} \right] \left[ h_0(0) \frac{\cos \psi_{0,\ell}}{\ell^{1/2} \sin^{1/2} \phi} \right] \sin \phi d \phi =O(\ell^{-2})
\end{align*}
since
\begin{align*}
 \int_{C/\ell}^{\pi -C/\ell} \frac{\sin^2 \psi_{0,\ell} \cos \psi_{0,\ell} }{ \sin^{1/2} \phi} d \phi =O(\ell^{-1/2}).
\end{align*}
And for $A_{777,\ell}$ we write
\begin{align*}
A_{777,\ell}&= \int_{C/\ell}^{\pi -C/\ell} \frac{ a^3_{7,\ell}(\phi)}{ (1-4 \alpha_2^2/\lambda_{\ell}^2)^{1/2} (1-4 \alpha_1^2/\lambda_{\ell}^2)^{1/2} } \sin \phi d \phi
\end{align*}
where
\begin{align*}
\frac{ a^3_{7,\ell}(\phi)}{ (1-4 \alpha_2^2/\lambda_{\ell}^2)^{1/2} (1-4 \alpha_1^2/\lambda_{\ell}^2)^{1/2} }= 8^3 \frac{\gamma_4^3}{\ell^{4 \cdot 3}}+O(\frac{1}{\ell^3 \phi^3})
\end{align*}
and, from \eqref{g43},
\begin{align*}
A_{777,\ell}&= 8^3 h^3_0(0) \frac{1}{\ell^{1+1/2}} \int_{C/\ell}^{\pi -C/\ell} \frac{\cos^3 \psi_{0,\ell}}{\sin^{1/2} \phi} d \phi=O(\ell^{-2}).
\end{align*}
The last two terms we need to study are $A_{377,\ell}$ and $A_{7777,\ell}$; $A_{377,\ell}$ is defined by
\begin{align*}
A_{377,\ell}&= \int_{C/\ell}^{\pi -C/\ell} \frac{ a_{3,\ell}(\phi) a^2_{7,\ell}(\phi)}{ (1-4 \alpha_2^2/\lambda_{\ell}^2)^{1/2} (1-4 \alpha_1^2/\lambda_{\ell}^2)^{1/2} } \sin \phi d \phi,
\end{align*}
where
\begin{align*}
 \frac{a_{3,\ell}(\phi) a^2_{7,\ell}(\phi)}{ (1-4 \alpha_2^2/\lambda_{\ell}^2)^{1/2} (1-4 \alpha_1^2/\lambda_{\ell}^2)^{1/2} } =-16 \frac{ \frac{\beta_3^2}{\lambda_{\ell}^3 (1-4 \alpha_2^2/\lambda_{\ell}^2)} 8^2 \frac{ \left[ \gamma _{4}+\frac{4\alpha _{2
} \beta _{3}^{2}}{4\alpha _{2}^{2}-\lambda _{\ell }^{2}}\right]^2}{\lambda _{\ell }^4}}{(1-4 \alpha_2^2/\lambda_{\ell}^2)^{1/2} (1-4 \alpha_1^2/\lambda_{\ell}^2)^{1/2}}= -16 \cdot 8^2 \; \frac{\beta_3^2}{\ell^6} \frac{\gamma_4^2}{\ell^8}+O(\frac{1}{\ell^3 \phi^3}).
\end{align*}
Now, by applying \eqref{a1}, \eqref{a2}, \eqref{b32} and \eqref{g42}, we have
\begin{align*}
A_{377,\ell}&= -16 \cdot 8^2 \; \int_{C/\ell}^{\pi -C/\ell} \frac{\beta_3^2}{\ell^6} \frac{\gamma_4^2}{\ell^8} \sin \phi d \phi +O(\ell^{-2})\\
&= -16 \cdot 8^2 \; \int_{C/\ell}^{\pi -C/\ell} \left[ h^2_0(0) \frac{\sin^2 \psi_{0,\ell}}{\ell \sin \phi} \right] \left[ h^2_0(0) \frac{\cos^2 \psi_{0,\ell}}{\ell \sin \phi} \right] \sin \phi d \phi +O(\ell^{-2})\\
%&= -16 \cdot 8^2 \frac{1}{\ell^2} h^4_0(0) \int_{C/\ell}^{\pi -C/\ell} \cos^2 \psi^+_{\ell} \cos^2 \psi^-_{\ell} \frac{1}{\sin \phi} d \phi +O(\ell^{-2})\\
&= -2 \cdot 16 \cdot 8^2 \frac{1}{\ell^2} \frac{2^2}{\pi^2} \frac 1 8 \log \ell +O(\ell^{-2})= -\frac{2 \cdot 2^9}{\pi^2} \frac{1}{\ell^2} \log \ell +O(\ell^{-2}).
\end{align*}
Therefore
\begin{align*}
{ \lambda _{\ell }^{2} \frac {3}{3!} A_{377,\ell}= - \frac{2^9}{ \pi^2} \ell^2 \log \ell +O(\ell^{2}).}
\end{align*}
Finally for $A_{7777,\ell}$ we apply \eqref{a1}, \eqref{a2} and \eqref{g44}, so that
\begin{align*}
A_{7777,\ell}&= \int_{C/\ell}^{\pi -C/\ell} \frac{ a^4_{7,\ell}(\phi)}{ (1-4 \alpha_2^2/\lambda_{\ell}^2)^{1/2} (1-4 \alpha_1^2/\lambda_{\ell}^2)^{1/2} } \sin \phi d \phi
= 8^4 \int_{C/\ell}^{\pi -C/\ell} \frac{\gamma_4^4}{\ell^{4 \cdot 4}} \sin \phi d \phi \\
&=8^4 h^4_0(0) \frac{1}{\ell^2}\int_{C/\ell}^{\pi -C/\ell} \frac{\cos^4 \psi_{0,\ell}}{\sin \phi} d \phi+O(\ell^{-2})
%&= 2 \cdot 8^4 \frac{2^2}{\pi^2} \frac{1}{\ell^2} \frac 3 8 \log \ell +O(\ell^{-2})\\
= \frac{2 \cdot 3 \cdot 2^{11}}{\pi^2} \frac{1}{\ell^2} \log \ell +O(\ell^{-2}),
\end{align*}
and
\begin{align*}
{ \lambda _{\ell }^{2} \frac 1 {4!} A_{7777,\ell} =\lambda _{\ell }^{2} \frac{1}{3\cdot 4} \left[ \frac{3 \cdot 2^{11}}{\pi^2} \frac{1}{\ell^2} \log \ell +O(\ell^{-2})\right] = \frac{ 2^{9}}{\pi^2} \ell^2 \log \ell +O(\ell^{2}).}
\end{align*}
\end{proof}

\noindent We prove now Lemma \ref{der_q_hat} and Lemma \ref{G-e} stated in Section \ref{eval-const}.

\begin{proof}[Proof of Lemma \ref{der_q_hat}]
Let
$$\hat{q}(\mathbf{a}, z_1,z_2,z_3, w_1,w_2,w_3)= \frac{1}{\sqrt{{\rm det}(\Delta (\mathbf{a}) )}} \exp \big\{- \frac 1 2 (z_1,z_2,z_3, w_1,w_2,w_3 ) \Delta(\mathbf{a})^{-1} (z_1,z_2,z_3, w_1,w_2,w_3 )^t \big\},
$$
and
$$
\mathbf{a}_i=(0,\dots,0,a_i,0,\dots,0), \hspace{1cm} i=1,\dots,8,
$$
where $a_i$ is the $i$th perturbing element of $\mathbf{a}$. Since $\hat{q} (\mathbf{a}, z_1,z_2,z_3, w_1,w_2,w_3) $ is an analytic function of the elements of the
vector $ \mathbf{a}$ \cite[Theorem 1.5]{kato}, to simplify the calculations note that, for example for the $j$th derivative with respect to $a_i$, we have
\begin{equation*}
\Big[\frac{\partial^{j}}{\partial a_{i}^{j}} \hat{q}(\mathbf{a}; t_{1},t_{2} ; z_1,z_2,w_1,w_2 )\Big]_{\mathbf{a}=\mathbf{0}}=%
\Big[\frac{\partial ^{j}}{\partial a_{i}^{j}} \hat{q} (\mathbf{a}%
_{i}; t_{1},t_{2} ; z_1,z_2,w_1,w_2 )\Big]_{\mathbf{a}_i=\mathbf{0}}, \;\;\;i=1,\dots,8.
\end{equation*}%
Now, using Leibniz integral rule and a computer-oriented computation to evaluate the derivates of $\hat{q}$, we obtain the statement of Lemma \ref{der_q_hat}.
\end{proof}

\begin{proof}[Proof of Lemma \ref{G-e}]
To prove the lemma it is convenient to introduce the transformation $W_1=Y_1$, $W_2=Y_2$ and $W_3=Y_1+Y_3$,
so that
$$
{\cal I}_r=\E[|Y_1 Y_3-Y_2^2| (Y_1-3 Y_3)^r]=\E[ |W_1 (W_3-W_1) -W_2^2|\, (W_1-3 (W_3-W_1) )^r ].
$$
We write now ${\cal I}_r$ in terms of a conditional expectation as follows:
$$
{\cal I}_r=\E_{W_3}[\E[ |W_1 (W_3-W_1) -W_2^2|\, (W_1-3 (W_3-W_1) )^r | W_3=t] ]
$$
and note that
$$
\E[ |W_1 (W_3-W_1) -W_2^2|\, (W_1-3 (W_3-W_1) )^r | W_3=t]= \E \left[ \left| \frac{t^2}{4}-Z_1^2-Z_2^2\right | \, (4 Z_1-t)^r \right]
$$
where $Z_1,Z_2$ denote standard independent Gaussian variables.

In the case $r=0$ we only have the chi-squared random variable $\zeta=Z_1^2+Z_2^2$ with density
$$f_{\zeta}(v)=\frac 1 2 e^{-\frac{v}{2}}, \hspace{1cm} v \in \mathbb{R},$$
so that we immediately have
$$
\E \left[ \left |\frac{t^2}{4}-Z_1^2-Z_2^2 \right| \right]=\E \left[ \left|\frac{t^2}{4}-\zeta \right| \right]=-2 + 4 e^{-\frac{t^2}{8}}+\frac{t^2}{4},
$$
and, since $W_3$ is a centred Gaussian with density
$$f_{W_3}(t)=\frac{1}{4 \sqrt \pi} e^{-\frac{t^2}{16}},$$
we obtain
$$
{\cal I}_0=\frac{1}{4 \sqrt \pi} \int_{\mathbb{R}} e^{-\frac{t^2}{16}} \left(-2 + 4 e^{-\frac{t^2}{8}}+\frac{t^2}{4} \right) d t=\frac{2^2}{\sqrt 3}.
$$

 For $r=2,4$ the proof is similar with the only difference that now we need to compute the joint density function of $\xi=Z_1$ and $\zeta=Z_1^2+Z_2^2$ that is given by
\begin{align*}
f_{(\xi,\zeta)}(u,v) =\frac{1}{2 \pi} \frac{e^{- \frac v 2}}{\sqrt{v-u^2}} \ind_{\{v\ge 0,\; u \in (-\sqrt v, \sqrt v )\}}.
\end{align*}

\end{proof}

\appendix

\section{Estimates for the first four derivatives of Legendre polynomials} \label{horror}
\noindent We start with the following lemma:

\begin{lemma} \label{hilb1}
For $\alpha+1> -1/2$ and $\alpha + \beta + 1 \ge - 1$, we have
\begin{align*}
\Big( \sin \frac{\phi}{2} \Big)^{\alpha+1} \Big(\cos \frac{\phi}{2}\Big)^{\beta} P_{\ell}^{(\alpha+1,\beta)}(\cos \phi)
=\frac{\Gamma(\ell+\alpha+2)}{\ell !} \Big( \frac{\phi}{\sin \phi} \Big)^{1/2} \Big[\sum_{n=0}^{m-1} A_n(\phi) \frac{J_{\alpha+n+1}(N \phi)}{N^{\alpha+n+1}}+\sigma_m \Big]
\end{align*}
where $$N=\ell+\frac 1 2(\alpha+\beta+2)$$ and
$$\sigma_m=\phi^m O(N^{-m-\alpha-1})$$
the $O$-therm being uniform with respect to $\theta \in [0,\pi-\varepsilon]$, $\varepsilon>0$. The coefficients $A_n(\phi)$
are analytic functions in $0 \le \phi \le \pi-\varepsilon$, and are $O(\phi^n)$ in that interval. In particular, $A_0(\phi)=1$ and
\begin{align*}
A_1(\phi)=\Big[ (\alpha+1)^2 -\frac 1 4 \Big] \Big( \frac{1-\phi \cot \phi}{2 \phi} \Big)- \frac{(\alpha+1)^2-\beta^2}{4} \tan \frac{\phi}{2}.
\end{align*}
\end{lemma}
\noindent For a proof of Lemma \ref{hilb1} see \cite[Lemma 1]{frenzen&wong}. We will apply Lemma \ref{hilb1} with $\alpha=-1$, $\beta=0$ and $m=1,2,3$, i.e.,
\begin{align*}
P_{\ell+u}(\cos \phi)= \Big( \frac{\phi}{\sin \phi} \Big)^{1/2} \Big[\sum_{n=0}^{m-1} A_n(\phi) \frac{J_{n}((\ell+u+1/2) \phi)}{(\ell+u+1/2)^{n}}+\phi^m O((\ell+u+1/2)^{-m}) \Big],
\end{align*}
with $u=0,1,2\dots$
\begin{lemma} \label{bessel}
The following asymptotic representation for the Bessel
functions of the first kind holds:
\begin{align*}
J_n(x)&=\left( \frac{2}{\pi x} \right)^{1/2} \cos(x- n \pi /2-\pi/4) %
 \sum_{k=0}^\infty (-1)^k (n,2k) \; (2 x)^{-2 k} \\
&\;\;-\left( \frac{2}{\pi x} \right)^{1/2} \sin(x-n \pi/2 -\pi/4)
\sum_{k=0}^\infty (-1)^k (n,2k+1)\; (2 x)^{-2 k-1},
\end{align*}
where $\varepsilon>0$, $|\arg x|\le \pi-\varepsilon$, $(n,0)=1$, and
$$(n,k)=\frac{(4 n^2-1) (4 n^2-3^2) \cdots (4 n^2-(2k-1)^2)}{2^{2k} k!}.$$
\end{lemma}
\noindent For a proof of Lemma \ref{bessel} see \cite[Section 5.11]{lebedev}.\\

 We will use the following notation: for $n=0, \dots, m-1$ and $u=0,1,2,\dots$
$$p_{n, \ell+u}(\phi)=\Big( \frac{\phi}{\sin \phi} \Big)^{1/2} A_n(\phi) \frac{J_{n}((\ell+u+1/2) \phi)}{(\ell+u+1/2)^{n}},$$
so that we have
\begin{align} \label{Nform}
P_{\ell+u}(\cos \phi)= \sum_{n=0}^{m-1} p_{n, \ell+u}(\phi) +\phi^m O(\ell^{-m}) .
\end{align}
Let $$h_n(k)= \left( \frac{2}{\pi} \right)^{1/2} (n,k) \frac{1}{2^{k}}, \hspace{0.5cm} \psi_{n,\ell+u}=(\ell+u+1/2) \phi- n \pi /2-\pi/4, \hspace{0.5cm} s_{n,k} (\ell,\phi)= \frac{1}{\sqrt{\sin \phi}} \frac{A_n(\phi)}{\phi^k} \frac{1}{\ell^{k+n+1/2}}.$$

In view of Lemma \ref{bessel} we can rewrite $p_{n, \ell+u}$ as follows
\begin{align*}
p_{n,\ell+u}(\phi)=p_{n,r,\ell+u}(\phi) +\phi^{n-1/2} O(\ell^{-r-n-3/2}),
\end{align*}
where
\begin{align} \label{Nform_2}
p_{n,r,\ell+u}(\phi) &=\cos \psi_{n,\ell+u} \sum_{k=0}^\infty (-1)^k h_n(2k) s_{n,2k} (\ell,\phi) \sum_{i=0}^r \sum_{j=i}^r \binom{-2k-n-1/2}{j} \binom{j}{i} \frac{1}{\ell^j} u^i 2^{i-j} \nonumber \\
&- \sin \psi_{n,\ell+u} \sum_{k=0}^\infty (-1)^k h_n(2k+1) s_{n,2k+1} (\ell,\phi) \sum_{i=0}^r \sum_{j=i}^r \binom{-2k-n-3/2}{j} \binom{j}{i} \frac{1}{\ell^j} u^i 2^{i-j} \nonumber\\
&+\phi^{n-1/2} O(\ell^{-r-n-3/2});
\end{align}
in particular for $u=0$ we have
\begin{align*}
p_{n, \ell}(\phi)&=p_{n,r,\ell}(\phi) +\phi^{n-1/2} O(\ell^{-r-n-3/2})\\
&=\cos \psi_{n,\ell} \sum_{k=0}^\infty (-1)^k h_n(2k) s_{n,2k} (\ell,\phi) \sum_{j=0}^r \binom{-2k-n-1/2}{j} \frac{1}{\ell^j} \frac{1}{2^j}\\
&- \sin \psi_{n,\ell} \sum_{k=0}^\infty (-1)^k h_n(2k+1) s_{n,2k+1} (\ell,\phi) \sum_{j=0}^r \binom{-2k-n-3/2}{j} \frac{1}{\ell^j} \frac{1}{2^{j}}\\&+\phi^{n-1/2} O(\ell^{-r-n-3/2}).
\end{align*}

We will use the following recurrence relations to express the first four derivatives of Legendre polynomials in terms of $P_{\ell+u}$, for $u=0,1,2,3,4$. We have \cite[Section 4.3]{lebedev}:
\begin{lemma} \label{P}
For $\ell=0,1,2 \dots$ %and $x \in \mathbb{R}$
\begin{align} \label{P'}
P'_\ell(x)&= \frac{\ell+1}{(x^2-1)} [x P_\ell(x)-P_{\ell+1}(x)],
\end{align}
\begin{align} \label{P''}
P''_\ell(x)&=\frac{\ell(\ell+1)}{(x^2-1)^2} [ x^2 P_\ell(x)-2 x P_{\ell+1}(x)+ P_{\ell+2}(x) ]+\frac{\ell+1}{(x^2-1)^2} [ (1+2 x^2 ) P_\ell(x)-5 x P_{\ell+1}(x)+2 P_{\ell+2}(x) ]% \nonumber
\end{align}
\begin{align} \label{P'''}
P'''_\ell(x)&= \frac{\ell+1}{(x^2-1)^3} \sum_{u=0}^{2} \ell^u \sum_{v=0}^{3} {_3}\omega_{u,v}(x) P_{\ell+v}(x),
\end{align}
where
\begin{align*}
&_3\omega_{2,0}(x)=- x^3 , \hspace{0.5cm} _3\omega_{2,1}(x)=3 x^2, \hspace{0.5cm} _3\omega_{2,2}(x)=-3 x, \hspace{0.5cm} _3\omega_{2,3}(x)= 1,\\
&_3\omega_{1,0}(x)=- (3 x + 5 x^3) , \hspace{0.5cm} _3\omega_{1,1}(x)=(3+18 x^2) , \hspace{0.5cm} _3\omega_{1,2}(x)=- 18 x, \hspace{0.5cm} _3\omega_{1,3}(x)=5,\\
&_3\omega_{0,0}(x)=-(9x+6 x^3), \hspace{0.5cm} _3\omega_{0,1}(x)=(6+27 x^2), \hspace{0.5cm} _3\omega_{0,2}(x)=- 24 x, \hspace{0.5cm} _3\omega_{0,3}(x)=6,
\end{align*}
\begin{align} \label{P''''}
P''''_\ell(x)&= \frac{\ell+1}{(x^2-1)^4} \sum_{u=0}^3 \ell^u \sum_{v=0}^4 {_4}\omega_{u,v}(x) P_{\ell+v}(x),
\end{align}
where
\begin{align*}
&_4\omega_{3,0}(x)=x^4 , \hspace{0.1cm} _4\omega_{3,1}(x)=-4 x^3, \hspace{0.1cm} _4\omega_{3,2}(x)=6 x^2 , \hspace{0.1cm} _4\omega_{3,3}(x)= -4 x,
\hspace{0.1cm} _4\omega_{3,4}(x)= 1,\\
&_4\omega_{2,0}(x)= 9 x^4+6 x^2, \hspace{0.1cm} _4\omega_{2,1}(x)=-(42 x^3+12 x) , \hspace{0.1cm} _4\omega_{2,2}(x)= 66 x^2+6, \hspace{0.1cm} _4\omega_{2,3}(x)= -42 x,
\hspace{0.1cm} _4\omega_{3,4}(x)= 9,\\
&_4\omega_{1,0}(x)= 26 x^4+42 x^2+3, \hspace{0.1cm} _4\omega_{1,1}(x)= -(146 x^3+78 x), \hspace{0.1cm} _4\omega_{1,2}(x)=231 x^2+30 , \hspace{0.1cm} _4\omega_{1,3}(x)=-134 x ,
\hspace{0.1cm} _4\omega_{3,4}(x)= 26,\\
&_4\omega_{0,0}(x)=24 x^4+72 x^2+9, \hspace{0.1cm} _4\omega_{0,1}(x)=-(168x^3+111 x) , \hspace{0.1cm} _4\omega_{0,2}(x)=246 x^2+36, \hspace{0.1cm} _4\omega_{0,3}(x)=-132 x,
\hspace{0.1cm} _4\omega_{3,4}(x)=24.
\end{align*}
\end{lemma}

We state now the main result of the section:

\begin{lemma} For any constant $C > 0$, we have, uniformly for $\ell \ge 1$ and $ \phi \in [C/\ell, \pi/2]$
\begin{align} \label{P's}
P'_{\ell} (\cos \phi)=\frac{\ell}{\sin^2 \phi} \big[ h_0(0) \, \sin \phi \sin \psi_{0,\ell} \, s_{0,0} (\ell,\phi) \big] + \phi^{-2-1/2} O(\ell^{-1/2}) + O(\phi^{-1}),
\end{align}
\begin{align} \label{P''s}
P''_{\ell} (\cos \phi)&= \frac{\ell^2}{\sin^4 \phi} \big[ - h_0(0) \sin^2 \phi \, \cos \psi_{0,\ell} s_{0,0}(\ell,\phi) \sum_{j=0}^1\binom{-1/2}{j} \frac{1}{2^j} \frac{1}{\ell^j} -2 h_0(0) \sin \phi \, \sin \psi_{0,\ell+1} s_{0,0}(\ell,\phi) \frac{1}{\ell} \big] \\
&\;\; + \frac{\ell}{\sin^4 \phi} \big[ - h_0(0) \sin^2 \phi \, \cos \psi_{0,\ell} \, s_{0,0}(\ell,\phi) \big] + \frac{\ell^2}{\sin^4 \phi} \big[h_0(1) \sin^2 \phi \, \sin \psi_{0,\ell} \, s_{0,1}(\ell,\phi) \big] \nonumber \\
&\;\;+\frac{\ell^2}{\sin^4 \phi} [- h_1(0) \sin^2 \phi \cos \psi_{1,\ell} s_{1,0} (\ell,\phi)] +\frac{\ell}{\sin^4 \phi} \big[ h_0(0) \sin \phi \, \sin \psi_{0,\ell-1} s_{0,0}(\ell,\phi) \big] \nonumber \\
&\;\;+ \phi^{-4-1/2} O(\ell^{-1/2})+O(\phi^{-2}), \nonumber
\end{align}
\begin{align} \label{P'''s}
P'''_{\ell} (\cos \phi)&=\frac{\ell^3}{\sin^6 \phi} \big[ h_0(0) \sin^3 \phi \sin \psi_{0,\ell} s_{0,0}(\ell,\phi) \sum_{j=0}^1 \binom{-1/2}{j} \frac{1}{2^j} \frac{1}{\ell^j} + \sin^3 \phi s_{0,2}(\ell,\phi) f_b(\phi) \\
&\;\;-3 h_0(0) \sin^2 \phi \cos \psi_{0,\ell+1} s_{0,0}(\ell,\phi) \frac{1}{\ell} + \sin \phi s_{0,0}(\ell,\phi) f_b(\phi) \frac{1}{\ell^2} \big] \nonumber\\
&\;\;+\frac{\ell^2}{\sin^6 \phi} \big[ h_0(0) \sin^3 \phi \sin \psi_{0,\ell} s_{0,0}(\ell,\phi) \big] \nonumber\\
&\;\;+ \frac{\ell^3}{\sin^6 \phi} \big[ h_0(1) \sin^3 \phi \cos \psi_{0,\ell} s_{0,1}(\ell,\phi) + \sin^2 \phi s_{0,1}(\ell,\phi) f_b(\phi) \frac{1}{\ell} \big] \nonumber\\
&\;\;+\frac{\ell^3}{\sin^6 \phi} \big[- h_{1}(0) \sin^3 \phi \sin \psi_{1,\ell} s_{1,0}(\ell,\phi) \big] \nonumber\\
&\;\; +\frac{\ell^2}{\sin^6 \phi} \big[ \frac 1 2 h_0(0) \sin^2 \phi (\cos \psi_{0,\ell+1} + 5 \cos \psi_{0,\ell-1}) s_{0,0}(\ell,\phi) + \sin \phi s_{0,0}(\ell,\phi) f_b(\phi) \frac{1}{\ell} \big] \nonumber \\
&\;\;+\frac{\ell^2}{\sin^6 \phi} \big[ \sin^2 \phi s_{0,1}(\ell,\phi) f_b(\phi) \big] + \frac{\ell}{\sin \phi} \big[ \sin \phi s_{0,0}(\ell,\phi) f_b(\phi) \big]
+ \phi^{-6-1/2}O(\ell^{-1/2})+\phi^{-4} O(\ell), \nonumber
\end{align}
\begin{align} \label{P''''s}
P''''_{\ell} (\cos \phi)&= \frac{\ell^4}{\sin^8 \phi} \Big[ \sin^4 \phi \cos \psi_{0,\ell} \sum_{k=0}^1 (-1)^k h_0(2k) s_{0,2k}(\ell,\phi) \sum_{j=0}^1 \binom{-2k-1/2}{j} \frac{1}{2^j} \frac{1}{\ell^j} \\
&\;\;+ \sin^4 \phi s_{0,0}(\ell,\phi) f_b(\phi) \frac{1}{\ell^2} +4 \sin^3 \phi\, \cos \psi_{1,\ell+1} \sum_{k=0}^1 (-1)^k h_0(2k) s_{0,2k}(\ell,\phi) \frac{1}{\ell} \nonumber\\
&\;\;+ \sin^3 \phi\, s_{0,0}(\ell,\phi) f_b(\phi) \frac{1}{\ell^2} + \sin^2 \phi s_{0,0}(\ell,\phi) f_b(\phi) \sum_{j=2}^3 \frac{1}{\ell^j}+ \sin \phi s_{0,0}(\ell,\phi) f_b(\phi) \frac{1}{\ell^3} \Big] \nonumber\\
&\;\;+\frac{\ell^3}{\sin^8 \phi} \Big[ \sin^4 \phi \cos \psi_{0,\ell} \sum_{k=0}^1 (-1)^k h_{0}(2k) s_{0,2k}(\ell,\phi) + \sin^4 \phi s_{0,0}(\ell,\phi) \frac{1}{\ell} \nonumber \\
&\;\;+ \sin^3 \phi s_{0,0}(\ell,\phi)f_b(\phi) \frac{1}{\ell} + \sin^2 \phi s_{0,0}(\ell,\phi) f_b(\phi) \frac{1}{\ell^2}\Big]\nonumber \\
&\;\;+\frac{\ell^4}{\sin^8 \phi} \Big[ -\sin^4 \phi \sin \psi_{0,\ell} \sum_{k=0}^1 (-1)^k h_{0}(2k+1) s_{0,2k+1}(\ell,\phi)- \sin^4 \phi s_{0,1}(\ell,\phi) f_b(\phi) \frac{1}{\ell} \nonumber \\
&\;\;+ \sin^3 \phi s_{0,1}(\ell,\phi) f_b(\phi) \frac{1}{\ell} + \sin^3 \phi s_{0,1}(\ell,\phi) f_b(\phi) \frac{1}{\ell^2}+ \sin^2 \phi s_{0,1}(\ell,\phi) f_b(\phi) \frac{1}{\ell^2} \Big] \nonumber \\
&\;\;+ \frac{\ell^3}{\sin^8 \phi} \Big[ \sin^4 \phi s_{0,1}(\ell,\phi) f_b(\phi) + \sin^3 \phi s_{0,1}(\ell,\phi) f_b(\phi) \frac{1}{\ell} \Big] \nonumber\\
&\;\;+\frac{\ell^4}{\sin^8 \phi} \Big[h_1(0) \sin^4 \phi \cos \psi_{1,\ell} s_{1,0}(\ell,\phi) \sum_{j=0}^1 \binom{-1/2}{j} \frac{1}{2^j} \frac{1}{\ell^j} + \sin^3 \phi s_{1,0}(\ell,\phi) f_b(\phi) \frac{1}{\ell} \Big]\nonumber \\
&\;\;+ \frac{\ell^3}{\sin^8 \phi}\Big[\sin^4 \phi s_{1,0}(\ell,\phi) f_b(\phi) \Big] + \frac{\ell^4}{\sin^8 \phi} \Big[ \sin^4 \phi s_{1,1}(\ell,\phi) f_b(\phi) \Big]+ \frac{\ell^4}{\sin^8 \phi} \Big[ \sin^4 \phi s_{2,0}(\ell,\phi) f_b(\phi) \Big] \nonumber \\
&\;\;+\frac{\ell^3}{\sin^8 \phi} \Big[ -\frac 3 2 h_0(0) \sin^3 \phi (\sin \psi_{0,\ell+1}+3 \sin \psi_{0,\ell-1} ) s_{0,0}(\ell,\phi) \sum_{j=0}^1 \binom{-1/2}{j} \frac{1}{2^j} \frac{1}{\ell^j} \nonumber \\
&\;\; + \sin^3 \phi s_{0,2}(\ell,\phi) f_b(\phi) + \sin^2 \phi s_{0,0}(\ell,\phi) f_b(\phi) \sum_{j=1}^2 \frac{1}{\ell^j} + \sin \phi s_{0,0}(\ell,\phi) f_b(\phi) \frac{1}{\ell^2}\Big] \nonumber\\
&\;\;+ \frac{\ell^2}{\sin^8 \phi} \Big[ - \frac 3 2 h_0(0) \sin^3 \phi (\sin \psi_{0,\ell+1} +3 \sin \psi_{0,\ell-1}) s_{0,0}(\ell,\phi) + \sin^2 \phi s_{0,0}(\ell,\phi) f_b(\phi) \frac{1}{\ell} \Big] \nonumber \\
&\;\;+ \frac{\ell^3}{\sin^8 \phi} \Big[ \sin^3 \phi s_{0,1}(\ell,\phi) f_b(\phi) \sum_{j=0}^1 \frac{1}{\ell^j} + \sin^2 \phi s_{0,1}(\ell,\phi) f_b(\phi) \frac{1}{\ell}\Big] + \frac{\ell^2}{\sin^8 \phi} \Big[ \sin^3 \phi s_{0,1}(\ell,\phi) f_b(\phi) \Big] \nonumber\\
&\;\;+ \frac{\ell^3}{\sin^8 \phi} \Big[ \sin^3 \phi s_{1,0}(\ell,\phi) f_b(\phi) \Big] +\frac{\ell^2}{\sin^8 \phi} \Big[ \sin^2 \phi s_{0,0}(\ell,\phi) f_b(\phi) \sum_{j=0}^1 \frac{1}{\ell^j} + \sin \phi s_{0,0}(\ell,\phi) f_b(\phi) \frac{1}{\ell}\Big] \nonumber \\
&\;\;+ \frac{\ell}{\sin^8 \phi} \Big[ \sin^2 \phi s_{0,0} (\ell,\phi) f_b(\phi) \Big] + \frac{\ell^2}{\sin^8 \phi} \Big[ \sin^2 \phi s_{0,1}(\ell,\phi) f_b(\phi) \Big] +\frac{\ell}{\sin^8 \phi} \Big[ \sin \phi s_{0,0}(\ell,\phi) f_b(\phi) \Big] \nonumber\\
&\;\;+ \phi^{-8-1/2} O(\ell^{-1/2})+\phi^{-5} O(\ell), \nonumber
\end{align}
where $f_b$ denotes a bounded function on $(0, \pi/2]$.
\end{lemma}
\begin{proof}
 \noindent {\it First derivative}\\

\noindent To prove \eqref{P's} we start from \eqref{P'} and we rewrite $P_\ell$ and $P_{\ell+1}$ as in \eqref{Nform} with $m=1$, i.e.,
\begin{align} \label{3:26}
P'_{\ell}(\cos \phi)&= \frac{\ell+1}{\sin^2 \phi} \big[ \cos \phi \, P_\ell(\cos \phi)-P_{\ell+1}(\cos \phi) \big] \nonumber \\
&= \frac{\ell+1}{\sin^2 \phi} \big[ \cos \phi \, p_{0,\ell}(\phi)-p_{0,\ell+1}(\phi) \big]\\
&\;\;+O(\phi^{-1}). \nonumber
\end{align}
We rewrite now \eqref{3:26} in the form \eqref{Nform_2} with $r=0$, i.e.
\begin{align*}
&\cos \phi \, p_{0,\ell}(\phi)-p_{0,\ell+1}(\phi) \\
&=\cos \phi \, p_{0,0,\ell}(\phi)- p_{0,0,\ell+1}(\phi) +\phi^{-1/2} O(\ell^{-3/2}) \nonumber\\
&= \cos \phi \, \Big[ \cos \psi_{0,\ell} \sum_{k=0}^\infty (-1)^k h_0(2k) s_{0,2k} (\ell,\phi) - \sin \psi_{0,\ell} \sum_{k=0}^\infty (-1)^k h_0(2k+1) s_{0,2k+1} (\ell,\phi) \Big] \nonumber \\
&\;\;-\Big[ \cos \psi_{0,\ell+1} \sum_{k=0}^\infty (-1)^k h_0(2k) s_{0,2k} (\ell,\phi) - \sin \psi_{0,\ell+1} \sum_{k=0}^\infty (-1)^k h_0(2k+1) s_{0,2k+1} (\ell,\phi) \Big] \nonumber \\
&\;\;+\phi^{-1/2} O(\ell^{-3/2}) ; \nonumber
\end{align*}
now note that
\begin{align} \label{3:32}
\left\{
\begin{matrix}
\cos \phi \cos \psi_{0,\ell} - \cos \psi_{0,\ell+1}=\sin \phi \sin \psi_{0,\ell}, \\
- \cos \phi \sin \psi_{0,\ell}+ \sin \psi_{0,\ell+1}=\sin \phi \cos \psi_{0,\ell},
\end{matrix}
\right.
\end{align}
\eqref{3:32} implies that
\begin{align*}
\cos \phi \, p_{0,\ell}(\phi)-p_{0,\ell+1}(\phi) &= \sin \phi \sin \psi_{0,\ell} \sum_{k=0}^\infty (-1)^k h_0(2k) s_{0,2k} (\ell,\phi) \\
&\;\;+\sin \phi \cos \psi_{0,\ell} \sum_{k=0}^\infty (-1)^k h_0(2k+1) s_{0,2k+1} (\ell,\phi) \\
&\;\;+\phi^{-1/2} O(\ell^{-3/2})
\end{align*}
and we obtain the estimate in the statement, in fact \eqref{3:26} is such that
\begin{align} \label{lll}
&\frac{\ell+1}{\sin^2 \phi} [ \cos \phi \, p_{0,\ell}( \phi)-p_{0,\ell+1}(\phi) ] \nonumber \\
&=\frac{\ell}{\sin^2 \phi} \sin \phi \sin \psi_{0,\ell}h_0(0) s_{0,0} (\ell,\phi) + \phi^{-2-1/2} O(\ell^{-1/2})+O(\phi^{-1}).
 \end{align}

 \noindent {\it Second derivative}\\

 \noindent We prove now \eqref{P''s}. We start from \eqref{P''} and we rewrite $P_{\ell+u}$ for $u=0,1,2$ in the form \eqref{Nform} with $m=2$, i.e.,
 \begin{align}
P''_{\ell}(\cos \phi)&= \frac{\ell (\ell+1)}{\sin^4 \phi} \big[\cos^2 \phi P_\ell(\cos \phi) - 2 \cos \phi P_{\ell+1}(\cos \phi)+ P_{\ell+2}(\cos \phi) \big] \nonumber \\
 &\;\;+ \frac{\ell+1}{\sin^4 \phi} \big[ (1+2 \cos^2 \phi ) P_\ell(\cos \phi)-5 \cos \phi P_{\ell+1}(\cos \phi)+2 P_{\ell+2}(\cos \phi) \big]\nonumber \\
 &=\frac{\ell (\ell+1)}{\sin^4 \phi} \big[ \cos^2 \phi \, p_{0,\ell}(\phi) - 2 \cos \phi \, p_{0,\ell+1}(\phi)+ p_{0,\ell+2}(\phi) \big]\label{11:39} \\
 &\;\;+ \frac{\ell (\ell+1)}{\sin^4 \phi} \big[ \cos^2 \phi \, p_{1,\ell}(\phi) - 2 \cos \phi \, p_{1,\ell+1}(\phi)+ p_{1,\ell+2}(\phi) \big] \label{11:39_2} \\
 &\;\;+ \frac{\ell+1}{\sin^4 \phi} \big[ (1+2 \cos^2 \phi ) \, p_{0,\ell}( \phi)-5 \cos \phi \, p_{0,\ell+1}( \phi)+2 p_{0,\ell+2}( \phi) \big]\label{11:40}\\
 &\;\;+ \frac{\ell+1}{\sin^4 \phi}\big[ (1+2 \cos^2 \phi ) \, p_{1,\ell}( \phi)-5 \cos \phi \, p_{1,\ell+1}( \phi)+2 p_{1,\ell+2}( \phi) \big]+ O(\phi^{-2}) \label{11:40_2}.
 \end{align}
We first consider the terms \eqref{11:39} and \eqref{11:39_2}; we rewrite them in the form \eqref{Nform_2} with $r=1$. For \eqref{11:39} we obtain:
\begin{align*}
&\cos^2 \phi \, p_{0,\ell}(\phi) - 2 \cos \phi \, p_{0,\ell+1}(\phi)+ p_{0,\ell+2}(\phi) \\
&= \cos^2 \phi \, p_{0,1,\ell}(\phi) - 2 \cos \phi \, p_{0,1,\ell+1}(\phi)+ p_{0,1,\ell+2}(\phi)+ \phi^{-1/2} O(\ell^{-2-1/2})\\
&= \cos^2 \phi \Big[ \cos \psi_{0,\ell} \sum_{k=0}^\infty (-1)^k h_0(2k) s_{0,2k}(\ell,\phi) \sum_{j=0}^1\binom{-2k-1/2}{j} \frac{1}{\ell^j} \frac{1}{2^j} \\
&\;\;-\sin \psi_{0,\ell} \sum_{k=0}^\infty (-1)^k h_0(2k+1) s_{0,2k+1}(\ell,\phi) \sum_{j=0}^1\binom{-2k-3/2}{j} \frac{1}{\ell^j} \frac{1}{2^j} \Big] \\
&\;\;- 2 \cos \phi \Big[ \cos \psi_{0,\ell+1} \sum_{k=0}^\infty (-1)^k h_0(2k) s_{0,2k}(\ell,\phi) \sum_{i=0}^1 \sum_{j=i}^1\binom{-2k-1/2}{j} \frac{1}{\ell^j} \binom{j}{i} 2^{i-j} \\
&\;\;-\sin \psi_{0,\ell+1} \sum_{k=0}^\infty (-1)^k h_0(2k+1) s_{0,2k+1}(\ell,\phi) \sum_{i=0}^1 \sum_{j=i}^1\binom{-2k-3/2}{j} \frac{1}{\ell^j} \binom{j}{i} 2^{i-j} \Big]\\
&\;\;+ \Big[ \cos \psi_{0,\ell+2} \sum_{k=0}^\infty (-1)^k h_0(2k) s_{0,2k}(\ell,\phi) \sum_{i=0}^1 \sum_{j=i}^1\binom{-2k-1/2}{j} \frac{1}{\ell^j} \binom{j}{i} 2^i 2^{i-j} \\
&\;\;-\sin \psi_{0,\ell+2} \sum_{k=0}^\infty (-1)^k h_0(2k+1) s_{0,2k+1}(\ell,\phi) \sum_{i=0}^1 \sum_{j=i}^1\binom{-2k-3/2}{j} \frac{1}{\ell^j} \binom{j}{i} 2^i 2^{i-j} \Big]\\
&\;\;+ \phi^{-1/2} O(\ell^{-2-1/2}),
\end{align*}
now, since
\begin{align}
\label{6:17} \left\{
\begin{matrix}
 \cos^2 \phi \, \cos \psi_{0,\ell} - 2 \cos \phi \cos \psi_{0,\ell+1} + \cos \psi_{0,\ell+2}=- \sin^2 \phi \, \cos \psi_{0,\ell},\\
-\cos^2 \phi \, \sin \psi_{0,\ell} + 2 \cos \phi \sin \psi_{0,\ell+1} - \sin \psi_{0,\ell+2}= \sin^2 \phi \, \sin \psi_{0,\ell},\\
 - 2 \cos \phi \cos \psi_{0,\ell+1} +2 \cos \psi_{0,\ell+2}=-2 \sin \phi \, \sin \psi_{0,\ell+1},\\
 2 \cos \phi \sin \psi_{0,\ell+1} -2 \sin \psi_{0,\ell+2}=-2 \sin \phi \, \cos \psi_{0,\ell+1},
 \end{matrix}
 \right.
\end{align}
in view of \eqref{6:17}, we obtain
\begin{align*}
&\cos^2 \phi \, p_{0,\ell}(\phi) - 2 \cos \phi \, p_{0,\ell+1}(\phi)+ p_{0,\ell+2}(\phi) \\
&=- \sin^2 \phi \, \cos \psi_{0,\ell} \sum_{k=0}^\infty (-1)^k h_0(2k) s_{0,2k}(\ell,\phi) \sum_{j=0}^1\binom{-2k-1/2}{j} \frac{1}{\ell^j} \frac{1}{2^j} \\
&\;\;-2 \sin \phi \, \sin \psi_{0,\ell+1} \sum_{k=0}^\infty (-1)^k h_0(2k) s_{0,2k}(\ell,\phi) \binom{-2k-1/2}{1} \frac{1}{\ell} \\
&\;\; +\sin^2 \phi \, \sin \psi_{0,\ell} \sum_{k=0}^\infty (-1)^k h_0(2k+1) s_{0,2k+1}(\ell,\phi) \sum_{j=0}^1\binom{-2k-3/2}{j} \frac{1}{\ell^j} \frac{1}{2^j} \\
&\;\;-2 \sin \phi \, \cos \psi_{0,\ell+1} \sum_{k=0}^\infty (-1)^k h_0(2k+1) s_{0,2k+1}(\ell,\phi) \binom{-2k-3/2}{1} \frac{1}{\ell} \\
&\;\;+ \phi^{-1/2} O(\ell^{-2-1/2})
\end{align*}
and then the term \eqref{11:39} has the following asymptotic behaviour:
\begin{align} \label{fff}
&\frac{\ell(\ell+1)}{\sin^4 \phi} [ \cos^2 \phi \, p_{0,\ell}(\phi) - 2 \cos \phi \, p_{0,\ell+1}(\phi)+ p_{0,\ell+2}(\phi)] \nonumber \\
&= \frac{\ell^2}{\sin^4 \phi} \Big[ - \sin^2 \phi \, \cos \psi_{0,\ell} h_0(0) s_{0,0}(\ell,\phi) \sum_{j=0}^1\binom{-1/2}{j} \frac{1}{\ell^j} \frac{1}{2^j} -2 \sin \phi \, \sin \psi_{0,\ell+1} h_0(0) s_{0,0}(\ell,\phi) \frac{1}{\ell} \Big] \\
&\;\; + \frac{\ell}{\sin^4 \phi} \Big[ - \sin^2 \phi \, \cos \psi_{0,\ell} \, h_0(0) s_{0,0}(\ell,\phi) \Big] \nonumber \\
&\;\;+ \frac{\ell^2}{\sin^4 \phi} \Big[ \sin^2 \phi \, \sin \psi_{0,\ell} \, h_0(1) s_{0,1}(\ell,\phi) \Big] + \phi^{-4-1/2} O(\ell^{-1/2})+O(\phi^{-2}) \nonumber .
\end{align}
For \eqref{11:39_2} we get
\begin{align*}
& \cos^2 \phi \, p_{1,\ell}(\phi) - 2 \cos \phi \, p_{1,\ell+1}(\phi)+ p_{1,\ell+2}(\phi) \\
&= \cos^2 \phi \, p_{1,1,\ell}(\phi) - 2 \cos \phi \, p_{1,1,\ell+1}(\phi)+ p_{1,1,\ell+2}(\phi)+ \phi^{1/2} O(\ell^{-3-1/2});
\end{align*}
now, note that
\begin{align} \label{5:22}
\left\{
\begin{matrix}
\cos^2 \phi \, \cos \psi_{1,\ell} - 2 \cos \phi \cos \psi_{1,\ell+1} + \cos \psi_{1,\ell+2}=- \sin^2 \phi \, \cos \psi_{1,\ell}, \\
-\cos^2 \phi \, \sin \psi_{1,\ell} + 2 \cos \phi \sin \psi_{1,\ell+1} - \sin \psi_{1,\ell+2}= \sin^2 \phi \, \sin \psi_{1,\ell}, \\
- 2 \cos \phi \cos \psi_{1,\ell+1} + 2 \cos \psi_{1,\ell+2}= \sin \phi \, f_b(\phi), \\
 2 \cos \phi \sin \psi_{1,\ell+1} - 2 \sin \psi_{1,\ell+2}=\sin \phi \, f_b(\phi) ,
 \end{matrix}
 \right.
 \end{align}
where $f_b$ is a bounded function of $\phi \in (0, \pi/2]$. Exploiting as before the trigonometric relations in \eqref{5:22} we obtain that \eqref{11:39_2} is such that
\begin{align} \label{gig}
&\frac{\ell(\ell+1)}{\sin^4 \phi}[ \cos^2 \phi \, p_{1,\ell}(\phi) - 2 \cos \phi \, p_{1,\ell+1}(\phi)+ p_{1,\ell+2}(\phi)] \nonumber \\
&=\frac{\ell^2}{\sin^4 \phi} [- \sin^2 \phi \cos \psi_{1,\ell} h_1(0) s_{1,0} (\ell,\phi)]+ \phi^{-4-1/2} O(\ell^{-1/2})+O(\phi^{-2}).
\end{align}
We apply the same procedure to obtain the asymptotic behaviour of the terms \eqref{11:40} and \eqref{11:40_2}. We rewrite them in the form \eqref{Nform_2} but in this case it is enough to choose $r=0$. For \eqref{11:40} we get
\begin{align*}
&(1+2 \cos^2 \phi ) \, p_{0,\ell}( \phi)-5 \cos \phi \, p_{0,\ell+1}( \phi)+2 p_{0,\ell+2}( \phi)\\
&= (1+2 \cos^2 \phi ) \, \, p_{0,0,\ell}(\phi) - 5 \cos \phi \, p_{0,0,\ell+1}(\phi)+ 2 p_{0,0,\ell+2}(\phi)+ \phi^{-1/2} O(\ell^{-1-1/2})\\
&= (1+2 \cos^2 \phi ) \, \Big[ \cos \psi_{0,\ell} \sum_{k=0}^\infty (-1)^k h_0(2k) s_{0,2k}(\ell,\phi) -\sin \psi_{0,\ell} \sum_{k=0}^\infty (-1)^k h_0(2k+1) s_{0,2k+1}(\ell,\phi) \Big] \\
&\;\;- 5 \cos \phi \Big[ \cos \psi_{0,\ell+1} \sum_{k=0}^\infty (-1)^k h_0(2k) s_{0,2k}(\ell,\phi) -\sin \psi_{0,\ell+1} \sum_{k=0}^\infty (-1)^k h_0(2k+1) s_{0,2k+1}(\ell,\phi) \Big]\\
&\;\;+ 2\Big[ \cos \psi_{0,\ell+2} \sum_{k=0}^\infty (-1)^k h_0(2k) s_{0,2k}(\ell,\phi) -\sin \psi_{0,\ell+2} \sum_{k=0}^\infty (-1)^k h_0(2k+1) s_{0,2k+1}(\ell,\phi) \Big]\\
&\;\;+ \phi^{-1/2} O(\ell^{-1-1/2}),
\end{align*}
and since
\begin{align}
\label{1:10} \left\{
\begin{matrix}
 (1+2 \cos^2 \phi ) \, \cos \psi_{0,\ell} - 5 \cos \phi \cos \psi_{0,\ell+1} +2 \cos \psi_{0,\ell+2}= \sin \phi \, \sin \psi_{0,\ell-1},\\
- (1+2 \cos^2 \phi ) \, \sin \psi_{0,\ell} + 5 \cos \phi \sin \psi_{0,\ell+1} - 2 \sin \psi_{0,\ell+2}= \sin \phi \, \cos \psi_{0,\ell-1},
 \end{matrix}
 \right.
\end{align}
in view of \eqref{1:10} we obtain that
\begin{align*}
&(1+2 \cos^2 \phi ) \, p_{0,\ell}( \phi)-5 \cos \phi \, p_{0,\ell+1}( \phi)+2 p_{0,\ell+2}( \phi)\\
&= \sin \phi \, \sin \psi_{0,\ell-1} \sum_{k=0}^\infty (-1)^k h_0(2k) s_{0,2k}(\ell,\phi) +\sin \phi \, \cos \psi_{0,\ell-1} \sum_{k=0}^\infty (-1)^k h_0(2k+1) s_{0,2k+1}(\ell,\phi) \\
&\;\;+ \phi^{-1/2} O(\ell^{-1-1/2}),
\end{align*}
and then for \eqref{11:40} we obtain:
\begin{align} \label{hhh}
&\frac{\ell+1}{\sin^4 \phi} \big[ (1+2 \cos^2 \phi ) \, p_{0,\ell}( \phi)-5 \cos \phi \, p_{0,\ell+1}( \phi)+2 p_{0,\ell+2}( \phi) \big] \nonumber \\
&=\frac{\ell}{\sin^4 \phi} \big[ \sin \phi \, \sin \psi_{0,\ell-1} h_0(0) s_{0,0}(\ell,\phi) \big] + \phi^{-4-1/2} O(\ell^{-1/2})+O(\phi^{-2}).
\end{align}
Finally for \eqref{11:40_2} one has
\begin{align}
&(1+2 \cos^2 \phi ) \, p_{1,\ell}( \phi)-5 \cos \phi \, p_{1,\ell+1}( \phi)+2 p_{1,\ell+2}( \phi) \nonumber \\
&= (1+2 \cos^2 \phi ) \, \, p_{1,0,\ell}(\phi) - 5 \cos \phi \, p_{1,0,\ell+1}(\phi)+ 2 p_{1,0,\ell+2}(\phi)+ \phi^{1/2} O(\ell^{-2-1/2}), \label{1:37}
\end{align}
and since
\begin{align}
\label{1:28} \left\{
\begin{matrix}
 (1+2 \cos^2 \phi ) \, \cos \psi_{1,\ell} - 5 \cos \phi \cos \psi_{1,\ell+1} +2 \cos \psi_{1,\ell+2}= \sin \phi f_b(\phi), \\%-\sin \phi \, \sin \psi_{1,\ell-1},\\
- (1+2 \cos^2 \phi ) \, \sin \psi_{1,\ell} + 5 \cos \phi \sin \psi_{1,\ell+1} - 2 \sin \psi_{1,\ell+2}= \sin \phi f_b(\phi), %- \sin \phi \, \cos \psi_{1,\ell-1},
 \end{matrix}
 \right.
\end{align}
exploiting, as before, \eqref{1:28} in \eqref{1:37} one obtain that \eqref{11:40_2} is of order
\begin{align} \label{iii}
&\frac{\ell+1}{\sin^4 \phi} \big[ (1+2 \cos^2 \phi ) \, p_{1,\ell}( \phi)-5 \cos \phi \, p_{1,\ell+1}( \phi)+2 p_{1,\ell+2}( \phi) \big] \nonumber \\
&= \phi^{-4-1/2} O(\ell^{-1/2})+O(\phi^{-2}).
\end{align}
By summing up the terms in \eqref{fff}, \eqref{gig}, \eqref{hhh} and \eqref{iii} we obtain the asymptotic expression \eqref{P''s} in the statement. The main steps in the proof of
 \eqref{P's} and \eqref{P''s} are summarised in Table \ref{FirstT}.

\begin{table}\centering
\ra{1.3}
\begin{tabular}{@{}lllllclllll@{}}\toprule
 \multicolumn{4}{c}{\hspace{2.7cm}$P'_{\ell}(\cos \phi)$} & \phantom{a}& \multicolumn{5}{c}{\hspace{3.1cm} $P''_{\ell}(\cos \phi)$} \\
 \cmidrule{1-5} \cmidrule{7-11}
fix $m=1$ in \eqref{Nform} & & & & & & fix $m=2$ in \eqref{Nform} & & & & \\
 \cmidrule{1-5} \cmidrule{7-11}
\eqref{3:26} with $r=0$ &and&\eqref{3:32} & $\implies$ & \eqref{lll} && \eqref{11:39} with $r=1$ &and& \eqref{6:17} & $\implies$ & \eqref{fff} \\
& & & & && \eqref{11:39_2} with $r=1$ &and& \eqref{5:22} & $\implies$ & \eqref{gig} \\
 & & & & && \eqref{11:40} with $r=0$ & and &\eqref{1:10}& $\implies$ & \eqref{hhh} \\
 & & & & && \eqref{11:40_2} with $r=0$ & and& \eqref{1:28} & $\implies$ & \eqref{iii} \\
 \bottomrule
\end{tabular}
\caption{}
 \label{FirstT}
\end{table}
 \vspace{0.3cm}

 \noindent {\it Third derivative}

\noindent We move now to the proof of the asymptotic behaviour of the third and fourth derivative given in formula \eqref{P'''s} and formula \eqref{P''''s} of the statement. For brevity sake we do not give, as before, all details of the proof; the main steps of the proof are summarised in Table \ref{SecondThirdT} and the related formulas written below.\\

\noindent To prove \eqref{P'''s} we start form \eqref{P'''} and we write $P_{\ell+u}$, $u=0,1,2,3$ in the form \eqref{Nform} with $m=2$:
\begin{align} \label{5:26}
P'''_\ell(\cos \phi)= \frac{\ell+1}{\sin^6 \phi} \sum_{u=0}^{2} \ell^u \sum_{n=0}^{1} \sum_{v=0}^{3} {_3}\omega_{u,v}(\cos \phi) p_{n,\ell+v}(\phi) + \phi^{-4} O(\ell).
\end{align}
Now, as described in Table \ref{SecondThirdT}, one can rewrite the $p_{n,\ell+u}$'s in the form \eqref{Nform_2}
with the value of the parameter $r$ chosen so that the error term is small enough (see Table \ref{SecondThirdT}). By exploiting the simplifications produced by the following trigonometric relations:
\begin{align} \label{7:30}
\left\{
\begin{matrix}
\sum_{v=0}^3 {_3}\omega_{2,v}(\cos \phi) \cos \psi_{0,\ell+v} = \sin^3 \phi \sin \psi_{0,\ell},\\
\sum_{v=1}^3 {_3}\omega_{2,v}(\cos \phi) \, v \cos \psi_{0,\ell+v} = -3 \sin^2 \phi \cos \psi_{0,\ell+1},\\
\sum_{v=1}^3 {_3}\omega_{2,v}(\cos \phi) \, v^2 \cos \psi_{0,\ell+v} = \sin \phi \, f_b(\phi),\\
-\sum_{v=0}^3 {_3}\omega_{2,v}(\cos \phi) \sin \psi_{0,\ell+v} = \sin^3 \phi \cos \psi_{0,\ell},\\
-\sum_{v=1}^3 {_3}\omega_{2,v}(\cos \phi) \, v^i \sin \psi_{0,\ell+v} = \sin^{3-i} \phi \, f_b(\phi), \;\;i=1,2,
\end{matrix}
\right.
\end{align}
\begin{align} \label{7:35}
\left\{
\begin{matrix}
\sum_{v=0}^3 {_3}\omega_{2,v}(\cos \phi) \cos \psi_{1,\ell+v} = - \sin^3 \phi \sin \psi_{1,\ell},\\
 \sum_{v=1}^3 {_3}\omega_{2,v}(\cos \phi) \, v^i \cos \psi_{1,\ell+v} = \sin^{3-i} \phi f_b(\phi), \;\;i=1,2,\\
-\sum_{v=0}^3 {_3}\omega_{2,v}(\cos \phi) \sin \psi_{1,\ell+v} = \sin^3 \phi f_b(\phi),\\
- \sum_{v=1}^3 {_3}\omega_{2,v}(\cos \phi) \, v^i \sin \psi_{1,\ell+v} = \sin^{3-i} \phi \, f_b(\phi), \;\;i=1,2,
\end{matrix}
\right.
\end{align}
\begin{align} \label{7:44}
\left\{
\begin{matrix}
& \sum_{v=0}^3 {_3}\omega_{1,v}(\cos \phi) \cos \psi_{0,\ell+v} = 1/2 \sin^2 \phi ( \cos \psi_{0,\ell+1} +5 \cos \psi_{0, \ell-1}),\\
&\sum_{v=1}^3 {_3}\omega_{1,v}(\cos \phi) v \cos \psi_{0,\ell+v} = \sin \phi f_b(\phi),\\
& - \sum_{v=0}^3 {_3}\omega_{1,v}(\cos \phi) \sin \psi_{0,\ell+v} = \sin^2 \phi f_b(\phi),\\
&-\sum_{v=1}^3 {_3}\omega_{1,v}(\cos \phi) v \sin \psi_{0,\ell+v} = \sin \phi f_b(\phi),
\end{matrix}
\right.
\end{align}
\begin{align} \label{7:59}
\left\{
\begin{matrix}
& \sum_{v=0}^3 {_3}\omega_{1,v}(\cos \phi) \cos \psi_{1,\ell+v} = \sin^2 \phi f_b(\phi),\\ %1/2 \sin^2 \phi ( \cos \psi_{1,\ell+1} +5 \cos \psi_{1, \ell-1}),\\
& \sum_{v=1}^3 {_3}\omega_{1,v}(\cos \phi) v \cos \psi_{1,\ell+v} = \sin \phi f_b(\phi),\\
& - \sum_{v=0}^3 {_3}\omega_{1,v}(\cos \phi) \sin \psi_{1,\ell+v} = \sin^2 \phi f_b(\phi),\\% -1/2 \sin^2 \phi ( \sin \psi_{1,\ell+1} +5 \sin \psi_{1, \ell-1}),\\
&-\sum_{v=1}^3 {_3}\omega_{1,v}(\cos \phi) v \sin \psi_{1,\ell+v} = \sin \phi f_b(\phi),
\end{matrix}
\right.
\end{align}
\begin{align} \label{8:14}
\left\{
\begin{matrix}
& \sum_{v=0}^3 {_3}\omega_{0,v}(\cos \phi) \cos \psi_{0,\ell+v} = \sin \phi f_b(\phi),\\
& -\sum_{v=0}^3 {_3}\omega_{0,v}(\cos \phi) \sin \psi_{0,\ell+v} = \sin \phi f_b(\phi), %\hspace{-9cm}= -3 \sin \phi \cos \phi \cos \psi_{0,\ell-1},
\end{matrix}
\right.
\end{align}
\begin{align} \label{8:22}
\left\{
\begin{matrix}
& \sum_{v=0}^3 {_3}\omega_{0,v}(\cos \phi) \cos \psi_{1,\ell+v} = \sin \phi f_b(\phi),\\
 &- \sum_{v=0}^3 {_3}\omega_{0,v}(\cos \phi) \sin \psi_{1,\ell+v} = \sin \phi f_b(\phi),
\end{matrix}
\right.
\end{align}
we can rewrite the terms \eqref{5:26} as follows:
\begin{align}\label{9:50}
 \frac{\ell+1}{\sin^6 \phi} \ell^2 \sum_{v=0}^{3} {_3}\omega_{2,v}(\cos \phi) p_{0,\ell+v}(\phi) &=\frac{\ell^3}{\sin^6 \phi} \big[\sin^3 \phi \sin \psi_{0,\ell} h_0(0) s_{0,0}(\ell,\phi) \sum_{j=0}^1 \binom{-1/2}{j} \frac{1}{\ell^j} \frac{1}{2^j} \\
&\;\;+ \sin^3 \phi \sin \psi_{0,\ell} h_0(2) s_{0,2}(\ell,\phi) -3 \sin^2 \phi \cos \psi_{0,\ell+1} h_0(0) s_{0,0}(\ell,\phi) \frac{1}{\ell} \nonumber\\
&\;\;+f_b(\phi) \sin \phi s_{0,0}(\ell,\phi) \frac{1}{\ell^2} \big] +\frac{\ell^2}{\sin^6 \phi} \big[ \sin^3 \phi \sin \psi_{0,\ell} h_0(0) s_{0,0}(\ell,\phi) \big] \nonumber\\
&\;\;+ \frac{\ell^3}{\sin^6 \phi} \big[ \sin^3 \phi \cos \psi_{0,\ell} h_0(1) s_{0,1}(\ell,\phi) + \sin^2 \phi f_b(\phi) s_{0,1}(\ell,\phi) \frac{1}{\ell} \big] \nonumber\\
&\;\;+ \phi^{-6-1/2}O(\ell^{-1/2})+\phi^{-4} O(\ell), \nonumber
\end{align}
\begin{align} \label{9:50i}
\frac{\ell+1}{\sin^6 \phi} \ell^2 \sum_{v=0}^{3} {_3}\omega_{2,v}(\cos \phi) p_{1,\ell+v}(\phi) &= \frac{\ell^3}{\sin^6 \phi} \big[- \sin^3 \phi \sin \psi_{1,\ell} h_{1}(0) s_{1,0}(\ell,\phi) \big]\\
&\;\;+ \phi^{-6-1/2}O(\ell^{-1/2})+\phi^{-4} O(\ell), \nonumber
\end{align}
\begin{align} \label{9:50ii}
\frac{\ell+1}{\sin^6 \phi} \ell \sum_{v=0}^{3} {_3}\omega_{1,v}(\cos \phi) p_{0,\ell+v}(\phi) &=\frac{\ell^2}{\sin^6 \phi} \big[ \frac 1 2 \sin^2 \phi (\cos \psi_{0,\ell+1} + 5 \cos \psi_{0,\ell-1}) h_0(0) s_{0,0}(\ell,\phi) + \sin \phi f_b(\phi) s_{0,0}(\ell,\phi) \frac{1}{\ell} \big] \\
&\;\;+\frac{\ell^2}{\sin^6 \phi} \big[ - \frac 1 2 \sin^2 \phi (\sin \psi_{0,\ell+1}+5 \sin \psi_{0,\ell-1}) h_0(1) s_{0,1}(\ell,\phi) \big] \nonumber\\
&\;\;+ \phi^{-6-1/2}O(\ell^{-1/2})+\phi^{-4} O(\ell), \nonumber
\end{align}
\begin{align} \label{9:50iii}
\frac{\ell+1}{\sin^6 \phi} \ell \sum_{v=0}^{3} {_3}\omega_{1,v}(\cos \phi) p_{1,\ell+v}(\phi) = \phi^{-6-1/2}O(\ell^{-1/2})+\phi^{-4} O(\ell),
\end{align}
\begin{align} \label{9:50iiii}
\frac{\ell+1}{\sin^6 \phi} \sum_{v=0}^{3} {_3}\omega_{0,v}(\cos \phi) p_{0,\ell+v}(\phi) = \frac{\ell}{\sin \phi} \big[ \sin \phi f_b(\phi) s_{0,0}(\ell,\phi) \big] + \phi^{-6-1/2}O(\ell^{-1/2})+\phi^{-4} O(\ell),
\end{align}
\begin{align} \label{9:50iiiii}
\frac{\ell+1}{\sin^6 \phi} \sum_{v=0}^{3} {_3}\omega_{0,v}(\cos \phi) p_{1,\ell+v}(\phi) = \phi^{-6-1/2}O(\ell^{-1/2})+\phi^{-4} O(\ell).
\end{align}
Formula \eqref{P'''s} in the statement is obtained by summing up the terms \eqref{9:50}-\eqref{9:50iiiii}. \\

\noindent {\it Fourth derivative}\\

\begin{table}\centering
\ra{1.3}
\begin{tabular}{@{}llcll@{}}\toprule
 \multicolumn{2}{c}{\hspace{1cm}$P'''_{\ell}(\cos \phi)$} & \phantom{a}& \multicolumn{2}{c}{\hspace{1cm} $P''''_{\ell}(\cos \phi)$} \\
 \cmidrule{1-2} \cmidrule{4-5}
fix $m=2$ in \eqref{Nform} & && fix $m=3$ in \eqref{Nform} & \\
 \cmidrule{1-2} \cmidrule{4-5}
\eqref{5:26} with $u=2, n=0$ $r=2$ & and \eqref{7:30} $\implies$ \eqref{9:50} && \eqref{11:22} $u=3, n=0$ $r=3$ & and \eqref{11:57} $\implies$ \eqref{6:24} \\
 \eqref{5:26} with $u=2, n=1$ $r=2$ & and \eqref{7:35} $\implies$ \eqref{9:50i}&& \eqref{11:22} $u=3, n=1$ $r=3$ & and \eqref{12:10} $\implies$ \eqref{6:24_1} \\
 \eqref{5:26} with $u=1, n=0$ $r=1$ & and \eqref{7:44} $\implies$ \eqref{9:50ii} &&\eqref{11:22} $u=3, n=2$ $r=3$ & and \eqref{12:18} $\implies$ \eqref{6:24_2}\\
\eqref{5:26} with $u=1, n=1$ $r=1$ & and \eqref{7:59} $\implies$ \eqref{9:50iii} &&\eqref{11:22} $u=2, n=0$ $r=2$ & and \eqref{12:31} $\implies$ \eqref{6:24_3} \\
\eqref{5:26} with $u=0, n=0$ $r=0$ & and \eqref{8:14} $\implies$ \eqref{9:50iiii} && \eqref{11:22} $u=2, n=1$ $r=2$ & and \eqref{12:33} $\implies$ \eqref{6:24_4} \\
\eqref{5:26} with $u=0, n=1$ $r=0$ & and \eqref{8:22} $\implies$ \eqref{9:50iiiii} && \eqref{11:22} $u=2, n=2$ $r=2$ & and \eqref{12:35} $\implies$ \eqref{6:24_5} \\
 & & &\eqref{11:22} $u=1, n=0$ $r=1$ & and \eqref{4:35} $\implies$ \eqref{6:24_6}\\
 & & &\eqref{11:22} $u=1, n=1$ $r=1$ & and \eqref{4:40} $\implies$ \eqref{6:24_7}\\
 & & & \eqref{11:22} $u=1, n=2$ $r=1$ & and \eqref{4:41} $\implies$ \eqref{6:24_8} \\
 & && \eqref{11:22} $u=0, n=0$ $r=0$ & and\eqref{4:43} $\implies$ \eqref{6:24_9} \\
 & && \eqref{11:22} $u=0, n=1$ $r=0$ & and \eqref{4:46} $\implies$ \eqref{6:24_10}\\
 & && \eqref{11:22}$u=0, n=2$ $r=0$ & and \eqref{4:47} $\implies$ \eqref{6:24_11}\\
 \bottomrule
\end{tabular}
\caption{}
 \label{SecondThirdT}
\end{table}

\noindent The proof of formula \eqref{P''''s} goes along the same lines. In view of \eqref{P''''} and by applying \eqref{Nform}, where we fix $m=3$, we have:
 \begin{align} \label{11:22}
P''''_\ell(\cos \phi)= \frac{\ell+1}{\sin^8 \phi} \sum_{u=0}^3 \ell^u \sum_{n=0}^2 \sum_{v=0}^4
{_4}\omega_{u,v}(\cos \phi) p_{n,\ell+v}(\phi) +\phi^{-5} O(\ell).
\end{align}
\noindent We can simplify each term in \eqref{11:22} by observing that:
\begin{align} \label{11:57}
\left\{
\begin{matrix}
&\sum_{v=0}^4 {_4}\omega_{3,v }(\cos \phi) \cos \psi_{0,\ell+v} =\sin^4 \phi \cos \psi_{0,\ell},\\
&\sum_{u=1}^4 u\, {_4}\omega_{3,v }(\cos \phi) \cos \psi_{0,\ell+v} = 4 \sin^3 \phi \cos \psi_{1,\ell+1},\\
&\sum_{u=1}^4 u^i \, {_4}\omega_{3,v }(\cos \phi) \cos \psi_{0,\ell+v} = \sin^{4-i} \phi f_b(\phi), \hspace{0.5cm} i=2,3,\\
& - \sum_{u=0}^4 {_4}\omega_{3,v }(\cos \phi) \sin \psi_{0,\ell+v} =-\sin^4 \phi \sin \psi_{0,\ell},\\
&- \sum_{u=1}^4 u^i \, {_4}\omega_{3,v }(\cos \phi) \sin \psi_{0,\ell+v} = \sin^{4-i} \phi f_b(\phi), \hspace{0.5cm} i=1,2,3,\\
\end{matrix}
\right.
\end{align}
\begin{align} \label{12:10}
\left\{
\begin{matrix}
&\sum_{u=0}^4 {_4}\omega_{3,v }(\cos \phi) \cos \psi_{1,\ell+v} =\sin^4 \phi \cos \psi_{1,\ell},\\
&\sum_{u=1}^4 u^i \, {_4}\omega_{3,v }(\cos \phi) \cos \psi_{1,\ell+v} = \sin^{4-i} \phi f_b(\phi), \hspace{0.5cm} i=1,2,3,\\
&-\sum_{u=0}^4 {_4}\omega_{3,v }(\cos \phi) \sin \psi_{1,\ell+v} =\sin^4 \phi f_b(\phi),\\
& -\sum_{u=1}^4 u^i \, {_4}\omega_{3,v }(\cos \phi) \sin \psi_{1,\ell+v} = \sin^{4-i} \phi f_b(\phi), \hspace{0.5cm} i=1,2,3,\\
\end{matrix}
\right.
\end{align}
\begin{align} \label{12:18}
\left\{
\begin{matrix}
&\sum_{u=0}^4 {_4}\omega_{3,v }(\cos \phi) \cos \psi_{2,\ell+v} =\sin^4 \phi f_b(\phi),\\
&\sum_{u=1}^4 u^i \, {_4}\omega_{3,v }(\cos \phi) \cos \psi_{2,\ell+v} = \sin^{4-i} \phi f_b(\phi), \hspace{0.5cm} i=1,2,3,\\
&-\sum_{u=0}^4 {_4}\omega_{3,v }(\cos \phi) \sin \psi_{2,\ell+v} =\sin^4 \phi f_b(\phi),\\
& -\sum_{u=1}^4 u^i \, {_4}\omega_{3,v }(\cos \phi) \sin \psi_{2,\ell+v} = \sin^{4-i} \phi f_b(\phi), \hspace{0.5cm} i=1,2,3,\\
\end{matrix}
\right.
\end{align}
\begin{align} \label{12:31}
\left\{
\begin{matrix}
&\sum_{u=0}^4 {_4}\omega_{2,v }(\cos \phi) \cos \psi_{0,\ell+v} =-3/2 \sin^3 \phi (\sin \psi_{0,\ell+1}+3 \sin \psi_{0,\ell-1}), \\
&\sum_{u=1}^4 u^i\, {_4}\omega_{2,v }(\cos \phi) \cos \psi_{0,\ell+v} = \sin^{3-i} \phi f_b(\phi),\hspace{0.5cm} i=1,2 \\
& - \sum_{u=0}^4 {_4}\omega_{2,v }(\cos \phi) \sin \psi_{0,\ell+v} = \sin^3 \phi f_b(\phi), \\
&- \sum_{u=1}^4 u^i \, {_4}\omega_{2,v }(\cos \phi) \sin \psi_{0,\ell+v} = \sin^{3-i} \phi f_b(\phi) ,\hspace{0.5cm} i=1,2, \\
\end{matrix}
\right.
\end{align}
\begin{align} \label{12:33}
\left\{
\begin{matrix}
&\sum_{u=0}^4 {_4}\omega_{2,v }(\cos \phi) \cos \psi_{1,\ell+v} = \sin^3 \phi f_b(\phi), \\
&\sum_{u=1}^4 u^i\, {_4}\omega_{2,v }(\cos \phi) \cos \psi_{1,\ell+v} = \sin^{3-i} \phi f_b(\phi),\hspace{0.5cm} i=1,2, \\
& - \sum_{u=0}^4 {_4}\omega_{2,v }(\cos \phi) \sin \psi_{1,\ell+v} = \sin^3 \phi f_b(\phi), \\
&- \sum_{u=1}^4 u^i \, {_4}\omega_{2,v }(\cos \phi) \sin \psi_{1,\ell+v} = \sin^{3-i} \phi f_b(\phi) ,\hspace{0.5cm} i=1,2, \\
\end{matrix}
\right.
\end{align}
\begin{align} \label{12:35}
\left\{
\begin{matrix}
&\sum_{u=0}^4 {_4}\omega_{2,v }(\cos \phi) \cos \psi_{2,\ell+v} = \sin^3 \phi f_b(\phi), \\
&\sum_{u=1}^4 u^i\, {_4}\omega_{2,v }(\cos \phi) \cos \psi_{2,\ell+v} = \sin^{3-i} \phi f_b(\phi),\hspace{0.5cm} i=1,2, \\
& - \sum_{u=0}^4 {_4}\omega_{2,v }(\cos \phi) \sin \psi_{2,\ell+v} = \sin^3 \phi f_b(\phi), \\
&- \sum_{u=1}^4 u^i \, {_4}\omega_{2,v }(\cos \phi) \sin \psi_{2,\ell+v} = \sin^{3-i} \phi f_b(\phi) ,\hspace{0.5cm} i=1,2, \\
\end{matrix}
\right.
\end{align}
 \begin{align} \label{4:35}
\left\{
\begin{matrix}
&\sum_{u=0}^4 {_4}\omega_{1,v }(\cos \phi) \cos \psi_{0,\ell+v} = \sin^2 \phi f_b(\phi), \\
&\sum_{u=1}^4 u\, {_4}\omega_{1,v }(\cos \phi) \cos \psi_{0,\ell+v} = \sin \phi f_b(\phi), \\
& - \sum_{u=0}^4 {_4}\omega_{1,v }(\cos \phi) \sin \psi_{0,\ell+v} = \sin^2 \phi f_b(\phi), \\
&- \sum_{u=1}^4 u \, {_4}\omega_{1,v }(\cos \phi) \sin \psi_{0,\ell+v} = \sin \phi f_b(\phi), \\
\end{matrix}
\right.
\end{align}
\begin{align} \label{4:40}
\left\{
\begin{matrix}
&\sum_{u=0}^4 {_4}\omega_{1,v }(\cos \phi) \cos \psi_{1,\ell+v} = \sin^2 \phi f_b(\phi), \\
&\sum_{u=1}^4 u\, {_4}\omega_{1,v }(\cos \phi) \cos \psi_{1,\ell+v} = \sin \phi f_b(\phi), \\
& - \sum_{u=0}^4 {_4}\omega_{1,v }(\cos \phi) \sin \psi_{1,\ell+v} =\sin^2 \phi f_b(\phi), \\
&- \sum_{u=1}^4 u \, {_4}\omega_{1,v }(\cos \phi) \sin \psi_{1,\ell+v} = \sin \phi f_b(\phi), \\
\end{matrix}
\right.
\end{align}
\begin{align} \label{4:41}
\left\{
\begin{matrix}
&\sum_{u=0}^4 {_4}\omega_{1,v }(\cos \phi) \cos \psi_{2,\ell+v} = \sin^2 \phi f_b(\phi), \\
&\sum_{u=1}^4 u\, {_4}\omega_{1,v }(\cos \phi) \cos \psi_{2,\ell+v} = \sin \phi f_b(\phi), \\
& - \sum_{u=0}^4 {_4}\omega_{1,v }(\cos \phi) \sin \psi_{2,\ell+v} =\sin^2 \phi f_b(\phi), \\
&- \sum_{u=1}^4 u \, {_4}\omega_{1,v }(\cos \phi) \sin \psi_{2,\ell+v} = \sin \phi f_b(\phi), \\
\end{matrix}
\right.
\end{align}
\begin{align} \label{4:43}
\left\{
\begin{matrix}
&\sum_{u=0}^4 {_4}\omega_{0,v }(\cos \phi) \cos \psi_{0,\ell+v} = \sin \phi f_b(\phi), \\
& - \sum_{u=0}^4 {_4}\omega_{0,v }(\cos \phi) \sin \psi_{0,\ell+v} = \sin \phi f_b(\phi), \\
\end{matrix}
\right.
\end{align}
\begin{align} \label{4:46}
\left\{
\begin{matrix}
&\sum_{u=0}^4 {_4}\omega_{0,v }(\cos \phi) \cos \psi_{1,\ell+v} = \sin \phi f_b(\phi), \\
& - \sum_{u=0}^4 {_4}\omega_{0,v }(\cos \phi) \sin \psi_{1,\ell+v} = \sin \phi f_b(\phi), \\
\end{matrix}
\right.
\end{align}
\begin{align} \label{4:47}
\left\{
\begin{matrix}
&\sum_{u=0}^4 {_4}\omega_{0,v }(\cos \phi) \cos \psi_{2,\ell+v} = \sin \phi f_b(\phi), \\
& - \sum_{u=0}^4 {_4}\omega_{0,v }(\cos \phi) \sin \psi_{2,\ell+v} = \sin \phi f_b(\phi), \\
\end{matrix}
\right.
\end{align}
\noindent Now combining \eqref{11:22} with \eqref{11:57}-\eqref{4:47} as described in Table \ref{SecondThirdT} one obtain:
\begin{align}
&\frac{\ell+1}{\sin^8 \phi} \ell^3 \sum_{v=0}^4 {_4}\omega_{3,v}(\cos \phi) p_{0,\ell+v}(\phi) \nonumber \\
&= \frac{\ell^4}{\sin^8 \phi} \Big[ \sin^4 \phi \cos \psi_{0,\ell} \sum_{k=0}^1 (-1)^k h_0(2k) s_{0,2k}(\ell,\phi) \sum_{j=0}^1 \binom{-2k-1/2}{j} \frac{1}{\ell^j} \frac{1}{2^j} \label{6:24} \\
&\;\;+ \sin^4 \phi s_{0,0}(\ell,\phi) \frac{1}{\ell^2} f_b(\phi)+ \sin^3 \phi\, 4 \cos \psi_{1,\ell+1} \sum_{k=0}^1 (-1)^k h_0(2k) s_{0,2k}(\ell,\phi) \frac{1}{\ell} \nonumber\\
&\;\;+ \sin^3 \phi\, s_{0,0}(\ell,\phi) \frac{1}{\ell^2} f_b(\phi) + \sin^2 \phi f_b(\phi) s_{0,0}(\ell,\phi) \sum_{j=2}^3 \frac{1}{\ell^j}+ \sin \phi f_b(\phi) s_{0,0}(\ell,\phi) \frac{1}{\ell^3} \Big] \nonumber\\
&\;\;+\frac{\ell^3}{\sin^8 \phi} \Big[ \sin^4 \phi \cos \psi_{0,\ell} \sum_{k=0}^1 (-1)^k h_{0}(2k) s_{0,2k}(\ell,\phi) + \sin^4 \phi s_{0,0}(\ell,\phi) \frac{1}{\ell} f_b(\phi) \nonumber \\
&\;\;+ \sin^3 \phi s_{0,0}(\ell,\phi) \frac{1}{\ell} f_b(\phi) + \sin^2 \phi f_b(\phi) s_{0,0}(\ell,\phi) \frac{1}{\ell^2}\Big]\nonumber \\
&\;\;+\frac{\ell^4}{\sin^8 \phi} \Big[ -\sin^4 \phi \sin \psi_{0,\ell} \sum_{k=0}^1 (-1)^k h_{0}(2k+1) s_{0,2k+1}(\ell,\phi)- \sin^4 \phi s_{0,1}(\ell,\phi) \frac{1}{\ell} f_b(\phi) \nonumber \\
&\;\;+ \sin^3 \phi f_b(\phi) s_{0,1}(\ell,\phi) \frac{1}{\ell} + \sin^3 \phi f_b(\phi) s_{0,1}(\ell,\phi) \frac{1}{\ell^2}+ \sin^2 \phi f_b(\phi) s_{0,1}(\ell,\phi) \frac{1}{\ell^2} \Big] \nonumber \\
&\;\;+ \frac{\ell^3}{\sin^8 \phi} \Big[ \sin^4 \phi s_{0,1}(\ell,\phi) f_b(\phi)+ \sin^3 \phi f_b(\phi) s_{0,1}(\ell,\phi) \frac{1}{\ell} \Big] \nonumber\\
&\;\;+ \phi^{-8-1/2} O(\ell^{-1/2})+\phi^{-5} O(\ell), \nonumber
\end{align}

\begin{align}
\frac{\ell+1}{\sin^8 \phi} \ell^3 \sum_{v=0}^4 {_4}\omega_{3,v}(\cos \phi) p_{1,\ell+v}(\phi) &=\frac{\ell^4}{\sin^8 \phi} \Big[ \sin^4 \phi \cos \psi_{1,\ell} h_1(0) s_{1,0}(\ell,\phi) \sum_{j=0}^1 \frac{1}{\ell^j}+ \sin^3 \phi f_b(\phi) s_{1,0}(\ell,\phi) \frac{1}{\ell} \Big]\label{6:24_1}\\
&\;\;+ \frac{\ell^3}{\sin^8 \phi}\Big[\sin^4 \phi s_{1,0}(\ell,\phi) f_b(\phi) \Big] + \frac{\ell^4}{\sin^8 \phi} \Big[ \sin^4 \phi f_b(\phi) s_{1,1}(\ell,\phi) \Big] \nonumber \\
&\;\; + \phi^{-8-1/2} O(\ell^{-1/2})+\phi^{-5} O(\ell), \nonumber
\end{align}

\begin{align}
\frac{\ell+1}{\sin^8 \phi} \ell^3 \sum_{v=0}^4 {_4}\omega_{3,v}(\cos \phi) p_{2,\ell+v}(\phi) = \frac{\ell^4}{\sin^8 \phi} \Big[ \sin^4 \phi f_b(\phi) s_{2,0}(\ell,\phi) \Big] + \phi^{-8-1/2} O(\ell^{-1/2})+\phi^{-5} O(\ell), \label{6:24_2}
\end{align}

\begin{align}
&\frac{\ell+1}{\sin^8 \phi} \ell^2 \sum_{v=0}^4 {_4}\omega_{2,v}(\cos \phi) p_{0,\ell+v}(\phi) \nonumber \\
&=\frac{\ell^3}{\sin^8 \phi} \Big[ -\frac 3 2 \sin^3 \phi (\sin \psi_{0,\ell+1}+3 \sin \psi_{0,\ell-1} ) h_0(0) s_{0,0}(\ell,\phi) \sum_{j=0}^1 \binom{-1/2}{j} \frac{1}{\ell^j} \frac{1}{2^j} \label{6:24_3} \\
&\;\; + \sin^3 \phi f_b(\phi) s_{0,2}(\ell,\phi) + \sin^2 \phi f_b(\phi) s_{0,0}(\ell,\phi) \sum_{j=1}^2 \frac{1}{\ell^j} + \sin \phi f_b(\phi) s_{0,0}(\ell,\phi) \frac{1}{\ell^2}\Big] \nonumber\\
&\;\;+ \frac{\ell^2}{\sin^8 \phi} \Big[ \sin^3 \phi f_b(\phi) s_{0,0}(\ell,\phi) + \sin^2 \phi f_b(\phi) s_{0,0}(\ell,\phi) \frac{1}{\ell} \Big] \nonumber \\
&\;\;+ \frac{\ell^3}{\sin^8 \phi} \Big[ \sin^3 \phi f_b(\phi) s_{0,1}(\ell,\phi) \sum_{j=0}^1 \frac{1}{\ell^j} + \sin^2 \phi f_b(\phi) s_{0,1}(\ell,\phi) \frac{1}{\ell}\Big] + \frac{\ell^2}{\sin^8 \phi} \Big[ \sin^3 \phi f_b(\phi) s_{0,1}(\ell,\phi) \Big] \nonumber
\\&\;\;+ \phi^{-8-1/2} O(\ell^{-1/2})+\phi^{-5} O(\ell), \nonumber
\end{align}

\begin{align}
\frac{\ell+1}{\sin^8 \phi} \ell^2 \sum_{v=0}^4 {_4}\omega_{2,v}(\cos \phi) p_{1,\ell+v}(\phi) = \frac{\ell^3}{\sin^8 \phi} \Big[ \sin^3 \phi f_b(\phi) s_{1,0}(\ell,\phi) \Big]+\phi^{-8-1/2} O(\ell^{-1/2})+\phi^{-5} O(\ell), \label{6:24_4}
\end{align}

\begin{align}
&\frac{\ell+1}{\sin^8 \phi} \ell^2 \sum_{v=0}^4 {_4}\omega_{2,v}(\cos \phi) p_{2,\ell+v}(\phi) = \phi^{-8-1/2} O(\ell^{-1/2})+\phi^{-5} O(\ell), \label{6:24_5}
\end{align}

\begin{align}
\frac{\ell+1}{\sin^8 \phi} \ell \sum_{v=0}^4 {_4}\omega_{1,v}(\cos \phi) p_{0,\ell+v}(\phi) &=\frac{\ell^2}{\sin^8 \phi} \Big[ \sin^2 \phi f_b(\phi) s_{0,0}(\ell,\phi) \sum_{j=0}^1 \frac{1}{\ell^j} + \sin \phi f_b(\phi) s_{0,0}(\ell,\phi) \frac{1}{\ell}\Big] \label{6:24_6} \\
&\;\; + \frac{\ell}{\sin^8 \phi} \Big[ \sin^2 \phi f_b(\phi) s_{0,0} (\ell,\phi) \Big] + \frac{\ell^2}{\sin^8 \phi} \Big[ \sin^2 \phi f_b(\phi) s_{0,1}(\ell,\phi) \Big] \nonumber\\ &\;\;+ \phi^{-8-1/2} O(\ell^{-1/2})+\phi^{-5} O(\ell), \nonumber
\end{align}

\begin{align}
&\frac{\ell+1}{\sin^8 \phi} \ell \sum_{v=0}^4 {_4}\omega_{1,v}(\cos \phi) p_{1,\ell+v}(\phi) = \phi^{-8-1/2} O(\ell^{-1/2})+\phi^{-5} O(\ell), \label{6:24_7}
\end{align}

\begin{align}
&\frac{\ell+1}{\sin^8 \phi} \ell \sum_{v=0}^4 {_4}\omega_{1,v}(\cos \phi) p_{2,\ell+v}(\phi) = \phi^{-8-1/2} O(\ell^{-1/2})+\phi^{-5} O(\ell), \label{6:24_8}
\end{align}

\begin{align}
\frac{\ell+1}{\sin^8 \phi} \sum_{v=0}^4 {_4}\omega_{0,v}(\cos \phi) p_{0,\ell+v}(\phi) =\frac{\ell}{\sin^8 \phi} \Big[ \sin \phi f_b(\phi) s_{0,0}(\ell,\phi) \Big]+\phi^{-8-1/2} O(\ell^{-1/2})+\phi^{-5} O(\ell), \label{6:24_9}
\end{align}

\begin{align}
&\frac{\ell+1}{\sin^8 \phi} \sum_{v=0}^4 {_4}\omega_{0,v}(\cos \phi) p_{1,\ell+v}(\phi)= \phi^{-8-1/2} O(\ell^{-1/2})+\phi^{-5} O(\ell), \label{6:24_10}
\end{align}

\begin{align}
&\frac{\ell+1}{\sin^8 \phi} \sum_{v=0}^4 {_4}\omega_{0,v}(\cos \phi) p_{2,\ell+v}(\phi)= \phi^{-8-1/2} O(\ell^{-1/2})+\phi^{-5} O(\ell). \label{6:24_11}
\end{align}
The sum of \eqref{6:24}-\eqref{6:24_11} gives the asymptotic behaviour of the fourth order derivative \eqref{P''''s}.

 \end{proof}

\section{Proof of formula \eqref{propinI}} \label{ProofpropinI}

\subsection{Approximate Kac-Rice formula for counting critical points with value in $I \subseteq \R$} \label{K-RinI}
For counting the number of critical points with corresponding value lying in any interval $I$ in the real line, we define, for $x \ne \pm y$, the two-point correlation function $K_{2,\ell}(x,y)$ as:
\begin{equation*}
K_{2,\ell}(x,y;t_{1},t_{2})=
\mathbb{E}\left[ \left\vert \nabla ^{2}f_{\ell }(x)\right\vert
\cdot \left\vert \nabla ^{2}f_{\ell }(y)\right\vert \Big| \nabla
f_{\ell }(x)=\nabla f_{\ell }(y)={\bf 0},f_{\ell }(x)=t_{1},f_{\ell
}(y)=t_{2}\right] \cdot \varphi _{x,y, \ell}(t_{1},t_{2}, \mathbf{0},\mathbf{0}),
\end{equation*}
where $\varphi _{x,y, \ell}(t_{1},t_{2},\mathbf{0},\mathbf{0})$ denotes
the density of the 6-dimensional vector
$$
\left( f_{\ell }(x),f_{\ell }(y),\nabla f_{\ell }(x),\nabla
f_{\ell }(y)\right)
$$
in $f_{\ell }(x)=t_{1},f_{\ell }(y)=t_{2},\nabla f_{\ell }(x)=\nabla f_{\ell}(y)={\bf 0}$. In \cite{CMW} the following {\it approximate Kac-Rice} formula is derived:
\begin{align} \label{afkaoI}
{\rm Var} \left( \mathcal{N}_I^{c} (f_{\ell })\right) = \int_{\mathcal{W}} \iint_{I \times I} K_{2,\ell}(x,y; t_1, t_2) \; d t_1 d t_2 d x d y - (\E [ \mathcal{N}_I^{c} (f_{\ell })])^2 + O(\ell^2).
\end{align}
Now, exploiting isotropy and observing that the level field $f_{\ell}$ is a linear combination of gradient and second order derivatives,
we have \cite[Section 4.1.2]{CMW}:
\begin{equation*}
K_{2,\ell }(\phi ; t_{1},t_{2}) = \frac{\lambda_{\ell}^{4}}{8} \frac{1}{\pi^{2} \sqrt{(\lambda _{\ell }^{2}-4\alpha _{2,\ell }^{2}(\phi
))(\lambda _{\ell }^{2}-4\alpha _{1,\ell }^{2}(\phi ))}} q (\mathbf{a}_{\ell}(\phi);t_1,t_2),
\end{equation*}
where
\begin{align*}
q (\mathbf{a};t_1,t_2)&=\frac{1}{(2\pi )^{3} \sqrt{\det (\Delta (\mathbf{a}))} }\iint_{\mathbb{R}^{2}\times \mathbb{R}%
^{2}} \left|z_{1} \sqrt 8 t_1 -z_{1}^{2}-z_{2}^{2}\right|\cdot \left| w_{1} \sqrt 8 t_2 -w_{1}^{2}-w_{2}^{2}\right| \\
& \hspace{0.2cm} \times \exp \left\{-\frac{1}{2}
(z_{1},z_{2},\sqrt 8
t_{1}-z_{1},w_{1},w_{2}, \sqrt 8 t_{2}-w_{1}) \Delta
(\mathbf{a})^{-1} (z_{1},z_{2},\sqrt 8
t_{1}-z_{1},w_{1},w_{2}, \sqrt 8 t_{2}-w_{1})^t \right\}\\
& \hspace{0.2cm} \times d z_1 dz_2 d w_1 dw_2.
\end{align*}

\subsection{Taylor expansion and asymptotics for the two-point correlation function}
By performing the Taylor expansion as in Section \ref{sec-taylor} and by applying Lemma \ref{L-zero} and
Lemma \ref{L-uno} as in Section \ref{sec-asymp}, one obtain that, in the high energy limit,
\begin{align} \label{var_crit_ex}
\text{Var}({\cal N}^c(f_{\ell}))&= \frac 1 8 \left[ 2 \ell^3 + \frac{2 \cdot 3^2}{ \pi^2} \ell^2 \log \ell \right] \iint_{I \times I} q(\mathbf{0};t_1,t_2) dt_1 dt_2 \nonumber \\
&\;\;+ \frac 1 8 \left[ - 16 \ell^3 - \frac{2^5 \cdot 3 }{\pi^2} \ell^2 \log \ell \right] \; \iint_{I \times I} \Big[\frac{\partial }{\partial a_{3} } q (\mathbf{a};t_1,t_2)\Big]_{\mathbf{a} =\mathbf{0}} dt_1 dt_2 \nonumber \\
&\;\;+ \frac 1 8 \left[ 32 \ell^3 - \frac{ 2^6}{\pi^2} \ell^2 \log \ell \right] \iint_{I \times I} \Big[\frac{\partial^2 }{\partial a_{7} \partial a_{7} } q (\mathbf{a};t_1,t_2)\Big]_{\mathbf{a} =\mathbf{0}} dt_1 dt_2 \nonumber \\
&\;\;+ \frac 1 8 \left[ \frac{ 3 \cdot 2^7}{\pi^2} \ell^2 \log \ell \right] \iint_{I \times I} \Big[\frac{\partial^2 }{\partial a_{3} \partial a_{3} } q (\mathbf{a};t_1,t_2)\Big]_{\mathbf{a} =\mathbf{0}} dt_1 dt_2 \nonumber \\
&\;\;+ \frac 1 8 \left[ - \frac{ 2^9}{ \pi^2} \ell^2 \log \ell \right] \iint_{I \times I} \Big[\frac{\partial^3 }{\partial a_{3} \partial a_{7} \partial a_{7} } q (\mathbf{a};t_1,t_2)\Big]_{\mathbf{a} =\mathbf{0}} dt_1 dt_2 \nonumber \\
&\;\;+ \frac 1 8 \left[ \frac{ 2^9}{ \pi^2} \ell^2 \log \ell \right] \iint_{I \times I} \Big[\frac{\partial^4 }{\partial a_{7}\partial a_{7} \partial a_{7} \partial a_{7} } q (\mathbf{a};t_1,t_2)\Big]_{\mathbf{a} =\mathbf{0}} dt_1 dt_2 +O(\ell^2).
\end{align}

\subsection{Evaluation of the leading constant}
Let ${\cal I}_{I,r}$, $r=0,2,4$, be the integrals ${\cal I}_{I,r}=\int_{I} p_r(t) dt$ where the functions $p_r$ for $r=0,2,4$ are defined by
\begin{align*}
p_0(t)&=\sqrt{8} \cdot \mathbb{E} \Big[|Y_1 Y_3 - Y_2^2| \Big| Y_1+Y_3=\sqrt{8} t \Big] \cdot \phi_{Y_1+Y_3} (\sqrt{8} t) = \sqrt{\frac{2}{\pi}} [2 e^{-t^2} + t^2-1] e^{- \frac{t^2}{2}},\\
p_2(t)&=\sqrt{8} \cdot \mathbb{E} \Big[(3 t -\sqrt 2 Y_1)^2 |Y_1 Y_3 - Y_2^2| \Big| Y_1+Y_3=\sqrt{8} t \Big] \cdot \phi_{Y_1+Y_3} (\sqrt{8} t) \\
&= \sqrt{\frac{2}{\pi}} [-4+t^2+t^4+ e^{-t^2} 2(4+ 3 t^2)] e^{- \frac{t^2}{2}},\\
p_4(t)&=\sqrt{8} \cdot \mathbb{E} \Big[(3 t -\sqrt 2 Y_1)^4 |Y_1 Y_3 - Y_2^2| \Big| Y_1+Y_3=\sqrt{8} t \Big] \cdot \phi_{Y_1+Y_3} (\sqrt{8} t) \\
&= \sqrt{\frac{2}{\pi}} [(72+96 t^2+38 t^4) e^{-t^2} -36 -12 t^2 +11 t^4+t^6] e^{-\frac{t^2}{2}},
\end{align*}
and $Y=(Y_{1},Y_{2},Y_{3})$ is the centred jointly Gaussian random vector defined in Section \ref{eval-const}. Using Leibniz integral rule and some mechanical computations, one has
\begin{align*}
& \iint_{I \times I} q(\mathbf{0};t_1,t_2) dt_1 dt_2=\frac{1}{2^3} {\cal I}^2_{I,0},\\
& \iint_{I \times I} \Big[\frac{\partial }{\partial a_{3} } q (\mathbf{a};t_1,t_2)\Big]_{\mathbf{a} =\mathbf{0}} dt_1 dt_2 = \frac{1}{2^6} [-3 {\cal I}^2_{I,0}+ {\cal I}_{I,0} {\cal I}_{I,2}], \\
 &\iint_{I \times I} \Big[\frac{\partial^2 }{\partial a_{7} \partial a_{7} } q (\mathbf{a};t_1,t_2)\Big]_{\mathbf{a} =\mathbf{0}} dt_1 dt_2
 = \frac{1}{2^9} [3 {\cal I}_{I,0} - {\cal I}_{I,2} ]^2,\\
&\iint_{I \times I} \Big[\frac{\partial^2 }{\partial a_{3} \partial a_{3} } q (\mathbf{a};t_1,t_2)\Big]_{\mathbf{a} =\mathbf{0}} dt_1 dt_2
 = \frac{1}{2^{11}} [ 2^3 \cdot 3^2 {\cal I}^2_{I,0} - 2^4 \cdot 3 {\cal I}_{I,0} {\cal I}_{I,2} + 2 {\cal I}_{I,0} {\cal I}_{I,4}+2 {\cal I}^2_{I,2} ], \\
&\iint_{I \times I} \Big[\frac{\partial^3 }{\partial a_{3} \partial a_{7} \partial a_{7} } q (\mathbf{a};t_1,t_2)\Big]_{\mathbf{a} =\mathbf{0}} dt_1 dt_2 = \frac{1}{2^{13}} [ -2 \cdot 3^4 {\cal I}_{I,0}^2 +2 \cdot 3^4 {\cal I}_{I,0} {\cal I}_{I,2} - 2 \cdot 3 {\cal I}_{I,0} {\cal I}_{I,4} + 2 {\cal I}_{I,4} {\cal I}_{I,2} -2^2 \cdot 3^2 {\cal I}^2_{I,2} ], \\
 &\iint_{I \times I} \Big[\frac{\partial^4 }{\partial a_{7}\partial a_{7} \partial a_{7} \partial a_{7} } q (\mathbf{a};t_1,t_2)\Big]_{\mathbf{a} =\mathbf{0}} dt_1 dt_2
=\frac{1}{2^{15}} [3^3 {\cal I}_{I,0} - 2 \cdot 3^2 {\cal I}_{I,2}+{\cal I}_{I,4}]^2.
\end{align*}
so that the variance \eqref{var_crit_ex} can be rewritten as
\begin{align*}
\text{Var}({\cal N}_I^c(f_{\ell}))= \frac{1}{2^4}[ 5 {\cal I}_{I,0} - {\cal I}_{I,2}]^2\; \ell^3+ \frac{1}{\pi^2 2^{6}} [51 {\cal I}_{I,0} - 2\cdot 11\; {\cal I}_{I,2}+{\cal I}_{I,4} ]^2\; \ell^2 \log \ell +O(\ell^2).
\end{align*}
Formula \eqref{propinI} follows by observing that
\begin{align} \label{cr}
%\nu^c(t)&=\frac{1}{2^2} [5 {\cal I}_{I,0} - {\cal I}_{I,2}]=\frac{1}{2^2} \sqrt{\frac{2}{\pi}} e^{-\frac{t^2}{2} } [- 1+4 t^2 -t^4 +e^{-t^2}(2-6 t^2) ], \\
\mu^c(t)&=\frac{1}{ \pi 2^3} [51 \, {p}_{0}(t) - 2\cdot 11\; {p}_{2}(t)+{p}_{4}(t) ] =\frac{1}{ \pi 2^3} \sqrt{\frac{2}{\pi}} [(-2 -36 t^2+38 t^4) e^{-t^2} +1+17 t^2-11 t^4+t^6]e^{-\frac{t^2}{2}}.
\end{align}
The proof of \eqref{propinI} for extrema and saddles is analogous and we obtain that $\mu^e$ and $\mu^s$ in \eqref{propinI} are given by
\begin{align}
\mu^e(t)&=\frac{1}{2^3 \pi} \sqrt{\frac{2}{\pi}} [(-1-18 t^2+19 t^4) e^{-t^2} +1+17 t^2-11 t^4+t^6] e^{-\frac{t^2}{2}}, \label{es}\\
\mu^s(t)&=\frac{1}{2^3 \pi} \sqrt{\frac{2}{\pi}} (-1-18 t^2+19 t^4) e^{-\frac{3 t^2}{2}}. \label{sad}
\end{align}


\begin{thebibliography}{99}



\bibitem{adlertaylor}
R. J.~Adler, J. E.~Taylor,
{\em Random Fields and Geometry}.
Springer Monographs in Mathematics, Springer, New York, 2007.



\bibitem{AAR}
G. E.~Andrews, R.~Askey, R.~ Roy,
{\em Special Functions. Encyclopedia of Mathematics and its Applications}.
Cambridge University Press, Cambridge, 1999.

\bibitem{azaiswschebor}
J.-M.~Aza\"{i}s, M.~Wschebor,
{\em Level sets and extrema of random processes and fields}.
John Wiley \& Sons Inc., Hoboken, NJ, 2009.

\bibitem{Belyaev Kereta} Belyaev, D., and Kereta, Z.
On the Bogomolny-Schmit conjecture. J. Phys. A 46, no. 45, (2013), 455003.

\bibitem{Belyaev private} Belyaev, D. Private communication.


\bibitem{BKMP} Baldi, P., Kerkyacharian, G., Marinucci, D., and
Picard, D. Subsampling needlet coefficients on the sphere,
\emph{Bernoulli} 15, 2, (2009), 438--463.

\bibitem{Bogomolny-Schmit} Bogomolny, E., and Schmit, C. Percolation model for nodal domains of chaotic wave functions, Phys.
Rev. Lett. 88, (2002), 114102.


\bibitem{CMW}
V.~Cammarota, D.~Marinucci and I.~Wigman,
{\em On the distribution of the critical values of random spherical harmonics}.
arXiv:1409.1364 [math-ph].


\bibitem{frenzen&wong} Frenzen, C. L., and Wong, R. Asymptotic expansions of the Lebesgue constants for Jacobi series. \emph{Pacific J. Math.}, 122, no. 2 (1986), 391--415.

\bibitem{kato} Kato, T. \emph{Perturbation theory for linear operators}, Classics in Mathematics. Springer-Verlag, Berlin, (1995).

\bibitem{Krishnapur} Krishnapur, M. Private communication

\bibitem{lebedev} Lebedev, N. N. \emph{Special functions and their applications}, Dover Publications, Inc., New York, (1972).


\bibitem{MaPeCUP}
D.~Marinucci, G.~Peccati,
{\em Random Fields on the Sphere: Representations, Limit Theorems and Cosmological Applications}.
London Mathematical Society Lecture Notes, Cambridge University Press, Cambridge, 2011.

\bibitem{Nastasescu} M. Nastasescu. The number of ovals of a real plane curve, Senior Thesis, Princeton 2011. Thesis and
Mathematica code available at: \url{http://www.its.caltech.edu/mnastase/Senior_Thesis.html}

\bibitem{N&S} Nazarov, F., and Sodin, M. On the number of nodal domains of random spherical harmonics.
 \emph{Amer. J. Math.}, 131 no. 5 (2009)

 \bibitem{Nicolaescu_2} Nicolaescu, L. I. Critical sets of random smooth functions on products of spheres arXiv:1008.5085.


\bibitem{rudnickwigman}
Z.~Rudnick and I.~Wigman,
{\em Nodal intersections for random eigenfunctions on the torus}.
arXiv:1402.3621 [math-ph].


\bibitem{rudnickwigmanyesha}
Z.~Rudnick, I.~Wigman and N. Yesha,
{\em Nodal intersections for random waves on the 3-dimensional torus}.
 arXiv:1501.07410 [math-ph].


\bibitem{szego} Sz\"ego, G. \emph{Orthogonal Polynomials}, Fourth edition. American Mathematical Society, Colloquium Publications, Vol. XXIII. American Mathematical Society, Providence, R.I. (1975)

%\bibitem{Berry 1977}
%M. V.~Berry,
%{\em Regular and irregular semiclassical wavefunctions}.
%{\it J. Phys. A 10}, 12, 2083-2091 (1977)

%\bibitem{Wig}
%I.~Wigman,
%{\em Fluctuation of the Nodal Length of Random Spherical Harmonics}.
%{\it Communications in Mathematical Physics}, 298 no. 3 (2010), 787-831

%\bibitem{Wsurvey}
%I.~Wigman,
%{\em On the nodal lines of random and deterministic Laplace eigenfunctions}.
%Spectral geometry, Proc. Sympos. Pure Math., 84, Amer. Math. Soc., Providence, RI. 285-297,
%arXiv:1103.0150 (2012)


%\bibitem{azais} Azais, J. M., and Pham, V.H. The record method for two and three dimensional
%parameters random fields. arXiv 1302.1017




%\bibitem{CL}
%H. Cramer and M.R. Leadbetter,
%{\em Stationary and related stochastic processes. Sample function properties and their applications}.
%John Wiley \& Sons, Inc., New York-London-Sydney, 1967.


%\bibitem{Nicolaescu} Nicolaescu, L. I. Critical sets of random smooth functions on compact manifolds. arXiv:1101.5990 %(articolo con errore sulla costante)


%\bibitem{steinweiss} Stein, E. M., and Weiss, G. (1971) \emph{
%Introduction to Fourier Analysis on Euclidean Spaces.} Princeton Mathematical Series, no. 32. Princeton University Press, Princeton, N.J.


%\bibitem{Wig} Wigman, I. Fluctuation of the Nodal Length of Random
%Spherical Harmonics. \emph{Communications in Mathematical
%Physics}, 298 no. 3 (2010), 787-831



\end{thebibliography}
\end{document}